\documentclass{article}

\usepackage[left=1in,top=1in,right=1in,bottom=1in,head=.1in]{geometry}

\usepackage{authblk}
\usepackage{amsmath}
\usepackage{amssymb}
\usepackage{amsthm}
\usepackage{amsfonts,mathrsfs}
\usepackage{mathrsfs} 

\usepackage{bm}
\usepackage{bbm}

\usepackage{xparse}
\numberwithin{equation}{section}
\usepackage{array}
\usepackage{float}
\usepackage[usenames,svgnames,dvipsnames]{xcolor}
\usepackage{hyperref}   
\hypersetup{
colorlinks=true,
citecolor=ForestGreen,
linkcolor=RoyalBlue,   
urlcolor=RubineRed!90!black
}

\usepackage[utf8]{inputenc}
\usepackage[english]{babel}

\usepackage{cleveref}
\usepackage{amssymb}
\usepackage{pifont}
\usepackage{csquotes}
\usepackage{tikz}
\usepackage[nottoc]{tocbibind}
\usepackage[toc,page]{appendix}
\usepackage[english]{babel}
\usepackage{amsthm}
\usepackage{algorithm}
\usepackage{algpseudocode}
\usepackage{natbib}

\usepackage{lastpage}
\usepackage[skins]{tcolorbox}
\usepackage{pdfpages}

\numberwithin{equation}{section}
\newtheorem{definition}{Definition}[section]
\newtheorem{theorem}{Theorem}[section]
\newtheorem{corollary}{Corollary}[theorem]
\newtheorem{proposition}[theorem]{Proposition}
\newtheorem{lemma}[theorem]{Lemma}
\newtheorem{remark}[theorem]{Remark}

\DeclareMathOperator{\E}{\mathbb{E}}
\DeclareMathOperator{\gmit}{GMtLf(\alpha,\theta)}

\DeclareMathOperator{\Bernoulli}{Bernoulli}

\DeclareMathOperator{\asconv}{\overset{\mathrm{a.s.}}{\longrightarrow}}
\DeclareMathOperator{\iid}{\overset{\mathrm{iid}}{\sim}}

\DeclareMathOperator{\stable}{\mathcal{F}_\infty \textrm{-stable}}

\DeclareMathOperator*{\argmax}{arg\,max}

\newcommand*\diff{\mathop{}\!\mathrm{d}}

\newcommand{\Mit}{{{\mathsf{M}}_{\alpha, \theta}}}
\newcommand{\N}{{N}}
\newcommand{\ep}{{\mathbb{P}_{\alpha, \theta}}}
\newcommand{\normal}{{N}}
\newcommand{\plug}{{\theta_{\mathsf{plug}}}}

\numberwithin{equation}{section} 
\begin{document}

\title{Asymptotic mixed normality of maximum likelihood estimator for Ewens--Pitman partition\thanks{This paper will appear in \emph{Advances in Applied Probability 58.1 (March 2026)}.}}

\author[1]{Takuya Koriyama\thanks{ \href{mailto:tkoriyam@uchicago.edu}{tkoriyam@uchicago.edu}}}

\author[2,3]{Takeru Matsuda
\thanks{
\href{mailto:matsuda@mist.i.u-tokyo.ac.jp}{matsuda@mist.i.u-tokyo.ac.jp}}
}
\author[2,3]{Fumiyasu Komaki
\thanks{\href{mailto:komaki@g.ecc.u-tokyo.ac.jp}{komaki@g.ecc.u-tokyo.ac.jp}}
}

\affil[1]{The University of Chicago}
\affil[2]{The University of Tokyo}
\affil[3]{RIKEN Center for Brain Science}

\date{\vspace{-15pt}}
\maketitle
\begin{abstract}
  {
  This paper investigates the asymptotic properties of parameter estimation for the Ewens–-Pitman partition with parameters $0<\alpha<1$ and $\theta>-\alpha$. Especially, we show that the maximum likelihood estimator (MLE) of $\alpha$ is $n^{\alpha/2}$-consistent and converges to a variance mixture of normal distributions, where the variance is governed by the Mittag-Leffler distribution. Moreover, we show that a proper normalization involving a random statistic eliminates the randomness in the variance. Building on this result, we construct an approximate confidence interval for $\alpha$. 
  Our proof relies on a stable martingale central limit theorem, which is of independent interest. 
  }
\end{abstract}

\tableofcontents

\section{Introduction}\label{sec:intro}
For any positive integer $n\in \mathbb{N}$ and any positive integer $k\le n$, a \textit{partition} of $[n]:= \{1, 2, \dots, n\}$ into $k$ blocks, denoted by $\{U_i: 1\leq i \leq k\}$, is an unordered collection of nonempty disjoint sets whose union is $[n]$. Now, let $\mathcal{P}_{n}^{k}$ be the set of all partitions of $[n]$ into $k$ blocks, and let $\mathcal{P}_n = \cup_{k=1}^{n} \mathcal{P}_{n}^{k}$ be the set of all partitions of $[n]$. {Denoting by $\{U_i: i\ge 1\}$ the element in $\mathcal{P}_n$}, the Ewens--Pitman partition is {a} distribution on $\mathcal{P}_{n}$ with {the following density parameterized by $(\alpha, \theta)$}:
\begin{align}\label{eq:ep_density}
\frac{\prod_{i=1}^{K_n-1} (\theta + i\alpha )}{\prod_{i=1}^{n-1}(\theta + i)} \prod_{j=2}^{n}
  \left\{
    \prod_{i=1}^{j-1} (-\alpha + i)
  \right\}^{S_{n,j}} \quad\text{for} \quad \begin{array}{l}
      S_{n,j} = \sum_{i\geq 1} \mathbbm{1}\{|U_i| = j\}\\
      K_n = \sum_{j=1}^n S_{n,j}
  \end{array}
\end{align}
Note that $S_{n,j}$ is the number of blocks of size $j$ and $K_n$ is the number of nonempty blocks.
The likelihood formula 
\eqref{eq:ep_density} implies that $(S_{n,j})_{j=1}^n$ is a sufficient statistic.

Now we suppose that the true parameter $(\alpha, \theta)$ satisfies $0<\alpha<1$ and $\theta>-\alpha$. Then, the number of nonempty blocks $K_n$ and the number of blocks of size $j$, denoted by $S_{n,j}$, have the following asymptotics as $n\to\infty$ (see \Cref{lm:ep_asym}):
\begin{figure}[b]
  \centering
\includegraphics[width=0.99\linewidth]{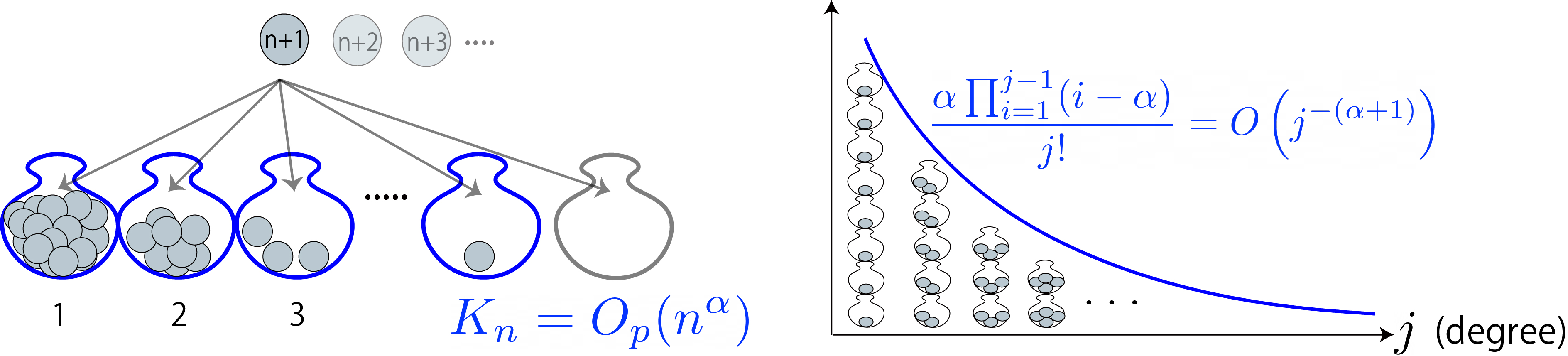}
  \caption{Asymptotic behavior of the Ewens--Pitman partition when $0<\alpha<1, \theta>-\alpha$.
  }
  \label{fig:ep_asymptotics}
\end{figure}
\begin{align}
   \frac{K_n}{n^\alpha} &\asconv \Mit \ \text{(nondegenerate random variable)}\label{eq:intro_sparse}\\
   \forall j\in \mathbb{N}, \quad   \frac{S_{n, j}}{K_n} &\asconv p_\alpha(j):= \frac{\alpha \prod_{i=1}^{j-1}(i-\alpha)}{j!}.\label{eq:intro_power}
\end{align}
See \Cref{fig:ep_asymptotics} for the illustration of the asymptotics \eqref{eq:intro_sparse}-\eqref{eq:intro_power}. We emphasize that the almost sure limit $\Mit = \lim_{n\to\infty} {K_n}/{n^\alpha}$ is not a constant but a nondegenerate positive random variable. Furthermore,  $p_\alpha(j)$ in \eqref{eq:intro_power} is a probability mass function on integers $\mathbb{N}$, and thanks to Stirling's formula $\Gamma(z) \sim \sqrt{{2\pi}/{z}} \left({z}/{e}\right)^z$, we have
$$
p_\alpha(j) \sim  \frac{\alpha}{\Gamma(1-\alpha)} j^{-(\alpha+1)} \quad \text{ as $j\to+\infty$}. 
$$
This implies that the ratio of blocks of particular sizes asymptotically follows a power law of exponents $(\alpha+1)$, as illustrated by the right figure of \Cref{fig:ep_asymptotics}. 

The Ewens--Pitman partition {has many applications } in 
ecology \cite{balocchi2022bayesian, favaro2021near, favaro2009bayesian, sibuya2014prediction},
nonparametric Bayesian inference \cite{caron2017generalized, dahl2017random}, disclosure risk assessment \cite{favaro2021bayesian, hoshino2001applying}, and network analysis \cite{crane2016edge, naulet2021asymptotic},
as well as forensic fields \cite{cereda2022learning}. In those related studies, the estimation of $\alpha$ is of more interest than the estimation of $\theta$ because $\alpha$ controls the asymptotic behavior mainly as we have shown in \eqref{eq:intro_sparse}-\eqref{eq:intro_power}. 
Here, 
\eqref{eq:intro_sparse} implies that the naive estimator $\hat{{\alpha}}_n^{\text{naive}}:= \log K_n / \log n$ is $\log n$-consistent, but it is not rate-optimal, owing to information loss from the sufficient statistic $(S_{n,j})_{j=1}^n$ to $K_n(=\sum_{j=1}^{n} S_{n,j})$.
For the maximum likelihood estimator $(\hat{\alpha}_n,
\hat{\theta}_n)$, \cite{favaro2021near} shows $\hat{\alpha}_n = \alpha + O_p(n^{-\alpha/2}\log n)$, but the exact asymptotic law is unknown. See \Cref{sec:compare} for a detailed review of prior literature. 

\subsection{Contribution}
In this paper, we derive the exact asymptotic distribution of the MLE $(\hat{\alpha}_n, \hat{\theta}_n)$. Here, let us introduce some notation; let $I_\alpha$ be the Fisher information of the heavy-tailed distribution with pmf $p_\alpha(j)$ in \eqref{eq:intro_power}, i.e., 
\begin{align}\label{eq:intro_sibuya_info}
    I_\alpha := - \sum_{j=1}^\infty p_\alpha(j) \cdot \partial_\alpha^2 \log p_\alpha(j) \quad \text{ with }\quad p_\alpha(j):= \frac{\alpha \prod_{i=1}^{j-1}(i-\alpha)}{j!},
\end{align}
and let $f_\alpha :(-1,\infty) \to \mathbb{R}$ be the function
\begin{align}\label{eq:intro_df_f_alpha}
    z\mapsto f_\alpha(z) = \psi(1+z)-\alpha\psi(1+\alpha z),
\end{align}
where $\psi(x)= \Gamma'(x)/\Gamma(x)$ is the digamma function. Note that $f_\alpha$ is bijective (see \Cref{lm:f_alpha}). 
Furthermore, we denote by $\Mit$ the limit of $K_n/n^{\alpha}$
\begin{align}\label{eq:intro_Mit_df}
    \Mit := \lim_{n\to\infty} K_n/n^\alpha. 
\end{align}
Recall that \eqref{eq:intro_sparse} implies $\Mit$ exists almost surely and $\Mit$ is a non-degenerate positive random variable. With the above notation, we show that the asymptotic distribution of the MLE $(\hat{\alpha}_n, \hat{\theta}_n)$ is characterized by (see \Cref{th:smle_asym})
\begin{align}
  \sqrt{n^{\alpha} I_\alpha} \cdot (\hat{\alpha}_n - \alpha) &\to  \N/\sqrt{\Mit} \quad &&(\stable), \label{eq:intro_mle_limit_alpha}\\
  \hat{\theta}_n &\to \alpha \cdot f_\alpha^{-1}(\log \Mit) \quad &&(\text{in probability}), \label{eq:intro_mle_limit_theta}
\end{align}
where $\N$ is a random variable following the standard normal $\normal(0,1)$, which is of $\Mit$, and $f_\alpha^{-1}$ is the inverse of the map $f_\alpha$ in \eqref{eq:intro_df_f_alpha}.  
Note the randomness of the limit distributions comes from the pair $(\N, \Mit)$. Here, $(\stable)$ in \eqref{eq:intro_mle_limit_alpha} is a notion of stochastic convergence stronger than the usual weak convergence (see \Cref{subsec:stable}). \eqref{eq:intro_mle_limit_alpha} and \eqref{eq:intro_mle_limit_theta} imply that $\hat{\alpha}_n$ is $n^{\alpha/2}$-consistent while $\hat{\theta}_n$ is not consistent since $\Mit$ is not constant.

It is important to emphasize that the limit law of $\hat{\alpha}_n$, $\N/\sqrt{\Mit}$, is not normal but a variance mixture of normals. This type of asymptotics is referred to as \textit{asymptotic mixed normality}, which is often observed in ``nonergodic'' or ``explosive'' stochastic processes (see \cite{hausler2015stable}).
By contrast, if we normalize the error $\hat{\alpha}_n-\alpha$ by the random statistic $\sqrt{K_n I_\alpha }$, where $K_n$ is the number of nonempty {blocks} and $I_\alpha$ is the Fisher information in \eqref{eq:intro_sibuya_info}, 
the randomness of variance is canceled out, and the limit law becomes the standard normal:
\begin{align*}
  \sqrt{K_n I_\alpha } (\hat{\alpha}_{n} - \alpha) = \sqrt{K_n/n^\alpha} \cdot \sqrt{n^\alpha I_\alpha} (\hat{\alpha}_{n} - \alpha)
\to \sqrt{\Mit} \cdot \N/\sqrt{\Mit} = \N,
\end{align*}
where we have used 
\eqref{eq:intro_Mit_df}, \eqref{eq:intro_mle_limit_alpha} and 
(generalized) Slutsky's lemma to stable convergence (see \Cref{th:CS_stable}). Informally, $K_n$ (the number of blocks) corresponds to the sample size in typical parametric i.i.d. cases, and $I_\alpha$ quantifies the Fisher information per block. As an immediate application of this result, we get the {approximate} $95\%$ confidence interval $[\hat{\alpha}_n \pm 1.96/\sqrt{K_n I_{\hat{\alpha}_n}}]$. 
We will apply the above result to a hypothesis testing of sparsity for network data (see \Cref{subsec:network}).

\subsection{Prior literature}\label{sec:compare}
Related papers consider the estimation of $\alpha$ under some miss-specified settings. To compare them with our work,  we introduce another representation of the Ewens--Pitman partition. Here, for a nonatomic measure $G$ on $\mathbb{R}$, e.g., $\normal(0,1)$,  the Pitman--Yor process $(\alpha, \theta; G)$ is the discrete random measure $P$ on $\mathbb{R}$ represented by
\begin{align*}
    P:= \sum_{i=1}^\infty p_i \delta_{y_i}, \quad  
    (y_i)_{i=1}^\infty \iid G, \quad  p_i = v_i \prod_{j=1}^{i-1}(1-v_j) \text{ with } v_i \sim \text{Beta}(1-\alpha , \theta + i\alpha),
\end{align*}
where $y_i$ and $v_i$ are independent. Since $P$ is discrete with probability $1$, conditional i.i.d. samples $(X_i)_{i\geq 1} |P \iid P$ induces a partition of $[n]$ by the equivalence relation $i\sim j$ iff $X_i = X_j$. Then, with a nontrivial calculation, we can show that the random partition induced by the Pitman--Yor process$(\alpha, \theta, G)$ has the same density as \eqref{eq:ep_density} of the Ewens--Pitman partition $(\alpha, \theta)$. In fact, any \textit{exchangeable} partitions can be induced by a discrete measure (see \cite{kingman1982coalescent}).

Here, we define the function $L_P :(1, \infty) \to \mathbb{N}$ for any discrete measure $P$ by
\begin{align*}
L_P(x) := \# \left\{
y : P(y) > x^{-1}
\right\}.
\end{align*}
Note that $L_P$ is an increasing function, and the order of $L_P(x)$ as $x\to+\infty$ characterizes the tail behavior of $P$. In particular, $L_P(x) = O(x^\alpha)$ (a.e.) when  $P$ follows Pitman--Yor process$(\alpha, \theta, G)$, i.e., the Pitman--Yor process has a heavy tail of index $\alpha$. More precisely, $x^{-\alpha} L_p(x)$ has a limit as $x\to\infty$ with probability $1$ and the limit is given by
\begin{align}\label{eq:ep_tail}
    \lim_{x\to \infty} x^{-\alpha} 
 L_P(x) =\frac{\Mit}{\Gamma(1-\alpha)} \quad \text{ (with probability $1$)},
\end{align}
where $\Mit$ is a random variable following $\gmit$ (see \cite{pitman2006combinatorial}). 

Previous research has focused on estimating the tail index $\alpha$ of the unknown discrete measure $P$ using the Pitman--Yor process $(\alpha,\theta, G)$ as a prior. Namely, they estimate the tail index $\alpha$ by fitting the Ewens--Pitman partition $(\alpha, \theta)$ to partition data induced by i.i.d. samples from $P$, with $\theta$ regarded as a nuisance parameter. Here, $P$ is allowed to be miss-specified and the assumption on $P$ takes the following form:
\begin{align}\label{eq:tail_general}
\exists L(x) \text{ slowly varying}, \quad \exists r(x) = o(x^\alpha) \quad \text{s.t.} \quad 
|L_P(x) - L(x) x^{\alpha}| \leq r(x).
\end{align}
Observe that \eqref{eq:tail_general} is motivated by \eqref{eq:ep_tail}: the Pitman--Yor process satisfies \eqref{eq:tail_general} with $L(x) = \Mit/\Gamma(1-\alpha)$ and $r(x) = O(x^{\alpha/2} \log x)$, where $\Mit$ is a positive random variable following $\gmit$ (see \cite[Proposition 10]{balocchi2022bayesian}).

Recently, several papers discuss the asymptotics of the MLE $\hat{\alpha}_n$ 
by imposing assumptions on $r(x)$ in \eqref{eq:tail_general};
\cite{favaro2021near} show $\hat{\alpha}_n = \alpha  + O_p(n^{-\alpha/2} \sqrt{\log n})$ and $\hat{\alpha}_n$ is minimax near optimal, under the assumption of \eqref{eq:tail_general} with $L$ being constant and $r(x) = O(x^{\alpha/2} \log x)$.
Comparing it with our rate $\hat{\alpha}_n - \alpha = O_p(n^{-\alpha/2})$ in \Cref{th:smle_asym}, 
we observe that the price for such a misspecification is just $\log n$ factors. 


 In contrast, \cite{balocchi2022bayesian} discuss the asymptotics of $(\hat{\alpha}_n, \hat{\theta}_n)$ under the assumption of \eqref{eq:tail_general} with $L(x) = L$ (a constant) and 
$r(x) = o(x^\alpha/\log x)$, which is weaker than $r(x) = O(x^{\alpha/2} \log x)$ in the previous assumption by \cite{favaro2021near}.
They show 
$(\hat{\alpha}_n, \hat{\theta}_n) \to ^p (\alpha, \Theta)$ where $\Theta$ is a solution to the following nonlinear equation:
\begin{align*}
L \cdot \Gamma(1-\alpha) = \exp(\psi(\Theta/\alpha+1)-\alpha \psi(\Theta + 1))
\end{align*}
If $P$ follows the Pitman--Yor process, we know $L= \Mit/\Gamma(1-\alpha)$ by \eqref{eq:ep_tail}. Combined with the definition $f_\alpha(x) = \psi(x + 1)-\alpha \psi(\alpha x + 1)$ in \eqref{eq:df_f_alpha},  we have
\begin{align*}
    &\log \Mit  = \psi(\Theta/{\alpha} + 1) - \alpha \psi(\Theta + 1) = f_\alpha(\Theta/\alpha)
\end{align*}
and $\Theta = \alpha f_\alpha^{-1} (\log \Mit)$, 
which recover the asymptotics $\hat{\theta}_n\to^p \alpha f_{\alpha}^{-1}(\log\Mit)$ in \Cref{th:smle_asym}. 


After our paper was originally posted, \cite{franssen2022empirical} derived the asymptotic distribution of $\hat{\alpha}_n$ under the assumption \eqref{eq:tail_general} with $r(x) = O( x^{\beta} )$ for some $\beta < \alpha/2$. 
They show that
\begin{align*}
    \sqrt{L_p(n)} (\hat{\alpha}_n - \alpha) \to \normal(0, \tau_1^2/\tau_2^4) \text{ as $n\to\infty$},
\end{align*}
where $\tau_1$ and $\tau_2$ are positive constants. Using our notation, $\tau_1$ can be written by $\tau_2^2 = \Gamma(1-\alpha) I_\alpha$ with $I_\alpha$ being the Fisher Information of the discrete distribution defined by \eqref{eq:sibuya_fisher}, so $\tau_1$ is interpretable.
Combining this with the tail asymptotics $L_p(n) \sim \Mit n^\alpha/\Gamma(1-\alpha) \sim K_n/\Gamma(1-\alpha)$ when $P$ follows the Pitman--Yor process (see \eqref{eq:ep_tail}), we obtain
\begin{align*}
    \sqrt{K_n I_\alpha} (\hat{\alpha}_n - \alpha) = \sqrt{\frac{K_n I_\alpha}{L_p(n)}} \sqrt{L_p(n)} (\hat{\alpha}_n - \alpha) 
    \to \normal \left(0,  \frac{\tau_1^2}{ I_\alpha \Gamma(1-\alpha)}\right).
\end{align*}
Thus, if $\tau_1^2 =\Gamma(1-\alpha) I_\alpha$ holds, the above display coincides with our result in \Cref{cor:mixing_convergence}. However, $\tau_2$ is an involved quantity, so we couldn't check $\tau_1^2 =\Gamma(1-\alpha) I_\alpha$. Furthermore, their assumption $r(x) = o(x^{\beta})$ with $\beta <\alpha/2$ is not satisfied by the Pitman--Yor process, and hence their result does not imply ours. 

{In summary, previous papers studied the estimation of the tail index of the underlying discrete measure by using the Pitman—Yor process as a prior, but to the best of our knowledge, the exact asymptotic distribution of the MLE was unknown.}
Considering the arguments in the previous section, the novelty of this paper is 
the exact asymptotic distribution of the MLE and the confidence interval of $\alpha$. In our proof, as highlighted in \Cref{sec:proof}, we avoid the representation of the Ewens--Pitman partition by the Pitman--Yor process. Instead, we exploit the sequential definition of the Ewens--Pitman partition and its martingale property.

\subsection{Organization}
The remainder of this paper is organized as follows: \Cref{sec:prep} reviews the parameter dependency of the Ewens-–Pitman partition and introduces the concept of stable convergence. In \Cref{sec:main}, we present the main theorem along with its applications to network analysis. Numerical simulations supporting the main theorem are provided in \Cref{sec:Numeric}. The proof strategy is outlined in \Cref{sec:proof}, and \Cref{sec:discussion} concludes with a potential direction for future research.
All proofs are given in the supplementary material. 

\section{Notation and preliminaries}\label{sec:prep}
\subsection{Parameter dependence of the Ewens--Pitman partition}\label{subsec:asym_ep}
In \Cref{sec:intro}, we introduced the Ewens--Pitman partition as a distribution on the set of partitions of $[n]$, denoted by $\mathcal{P}_n$. Here, we introduce an alternative representation of the Ewens--Pitman partition; the Ewens--Pitman partition is a stochastic process over $(\mathcal{P}_n)_{n\geq 1}$ that randomly assigns integers (balls) into blocks (urns) in the following sequential manners: 
\begin{enumerate}
	\item The first ball belongs to urn $U_1$ with probability one.
	\item  Suppose that $n (\geq 1)$ balls are partitioned into $K_n$ occupied urns $\{U_1, \dots U_{K_n}\}$, and let $|U_i|$ be the number of balls in $U_i$. Then, the $(n+1)$th ball is randomly assigned to the existing urns $\{U_1,\dots, U_{K_n}\}$ or new (empty) urn as follows:
	\begin{align*}
			\text{$(n+1)$th ball} \in \left\{\begin{array}{ll}
					U_i & \text{with prob. $\frac{|U_i| -\alpha}{\theta + n}$} \ (\forall i=1, 2, \dots, K_n) \\
					\text{Empty urn} & \text{with prob. $\frac{\theta+K_n\alpha}{\theta + n}$.}
			\end{array}\right.
	\end{align*}
\end{enumerate}
Then, it follows from simple algebra  that the probability of obtaining the partition $\{U_1, \dots, U_{K_n}\}$ of $[n]$ coincides with the likelihood formula \eqref{eq:ep_density}.

Note that the Ewens--Pitman partition has three parameter spaces: $(\mathrm{i})\ \alpha=0, \theta>0$, $(\mathrm{ii})\ \alpha<0, \exists k \in \mathbb{N}\ \text{s.t.} \ \theta = -k\alpha$, and $(\mathrm{iii})\ 0<\alpha<1, \theta>-\alpha$. Below, we briefly explain the parameter dependency. 

\textbf{When $\alpha=0, \theta>0$:} 
This is referred to as the Ewens partition or the standard Chinese restaurant process. By substituting $\alpha=0$ into the likelihood formula \eqref{eq:ep_density}, we observe that the likelihood is proportional to $\theta^{K_n}/(\prod_{i=0}^{n-1}(\theta + i))$, 
so that $K_n$ is a sufficient statistic for $\theta$. Now we define the estimator $\theta^\star_n := K_n/\log n$. Then, we claim that $\theta^\star_n$ is $\sqrt{\log n}$-consistent and asymptotically normal as follows:
$$
(\theta^{-1}\log n)^{1/2} (\theta^\star_n- \theta)\rightarrow \normal(0,1). 
$$
This follows from the following arguments:  By the sequential definition above, $K_n$ can be expressed as independent sum of the Bernoulli random variables
$K_n = \sum_{i=1}^{n} \zeta_i$ where $
    \zeta_i \sim \Bernoulli (\frac{\theta}{\theta + i-1})$. 
Then, 
the Lindeberg-Feller theorem (c.f.\cite[p. 128]{DU19}) gives the asymptotic normality. 

\textbf{When ${\alpha<0, \theta = -k\alpha}$ for some $k\in \mathbb{N}$:} 
In this case, the number of occupied urns $K_n$ is finite, i.e.,
$K_n \rightarrow k$ (a.s.), since the probability of observing a new urn is proportional to $(-\alpha)(k-K_n)$, which is strictly positive until $K_n$ reaches $k$.

\textbf{When $0<\alpha<1, \theta >- \alpha$:}
In this regime, nonstandard asymptotics hold. Before its introduction, let us define some distributions appearing in the asymptotics.
\begin{definition}[Sibuya distribution]\label{df:sibuya}
The Sibuya distribution of parameter $\alpha \in (0,1)$ \cite{sibuya1979generalized}, which is also called the Karlin--Rouault distribution \cite{karlin1967central, rouault1978lois}, is a discrete distribution on $\mathbb{N}$ with its density $p_\alpha(j)$ defined by 
\begin{align}\label{eq:sibuya_density}
  \forall j\in \mathbb{N}, \quad  p_\alpha(j):= \frac{\alpha \prod_{i=1}^{j-1}(i-\alpha)}{j!}.
\end{align}
\end{definition}
Here, Stirling's formula $\Gamma(z) \sim \sqrt{{2\pi}/{z}} \left({z}/{e}\right)^z$ implies that the Sibuya distribution is heavy-tailed in the following sense:
$$
p_\alpha(j) \sim \frac{\alpha}{\Gamma(1-\alpha)}\cdot j^{-(1+\alpha)} \quad \text{as} \quad j\to\infty
$$
Readers may refer to \cite{resnick2007heavy} for its important role in the extreme value theory.

\begin{definition}[Generalized Mittag-Leffler distribution]\label{df:gmit}
Let ${\mathsf{S}_\alpha}$ be a positive random variable of parameter $\alpha\in(0,1)$ with its Laplace transform given by $\E[e^{-\lambda {\mathsf{S}_\alpha}}] = e^{-\lambda^\alpha} (\lambda\geq 0)$.
Then, the law of ${\mathsf{M}_\alpha = \mathsf{S}_{\alpha}^{-\alpha}}$ is referred to as the Mittag-Leffler distribution $(\alpha)$. 
Moreover, for each $\theta>-\alpha$, the generalized Mittag-Leffler distribution $(\alpha, \theta)$, denoted by $\gmit$, is a tilted distribution with its density ${g_{\alpha, \theta}(x)}$ proportional to $x^{\theta/\alpha} g_\alpha(x)$, where $g_\alpha(x)$ is the density of the Mittag-Leffler distribution $(\alpha)$.
\end{definition}
\begin{remark}
The density $g_\alpha(x)$ of the Mittag-Leffler distribution is characterized by the moment
$\int_0^\infty x^p g_\alpha(x) \diff x =\Gamma(p + 1)/\Gamma(p\alpha + 1)$ for all $p>-1$.
Then, it easily follows from the definition $g_{{\alpha, \theta}}(x) \propto x^{\theta/\alpha} g(x)$ that the moment of $\Mit \sim \gmit$ is given by 
\begin{align}
    \forall p>-(1 + \theta/\alpha), \quad   \E[(\Mit)^p] = \frac{\Gamma(\theta + 1)}{\Gamma(\theta/\alpha + 1)} \frac{\Gamma(\theta/\alpha + p + 1)}{\Gamma(\theta + p\alpha + 1)}.\label{eq:gmit_moment}
\end{align}        
\end{remark}
Finally, we introduce the nonstandard asymptotics of the Ewens--Pitman partition when $0<\alpha<1, \theta>-\alpha$. 
\begin{theorem}\label{lm:ep_asym}
We assume $0<\alpha<1, \theta>- \alpha$. Let $S_{n,j}$ be the number of blocks with size $j$, and let $K_n = \sum_{j\geq 1} S_{n,j}$ be the number of nonempty blocks. Then, we have
\begin{itemize}
  \item[(A)] $K_n/n^\alpha \rightarrow \Mit $ a.s. and in the $p$-th moment for all $p>0$, where $\Mit \sim \gmit$ (see \Cref{df:gmit} for the definition of the law $\gmit$).
  \item[(B)] ${S_{n, j}}/{K_n} \rightarrow p_\alpha(j)$ a.s. for all $j \in \mathbb{N}$, where $p_\alpha (j)$ is the density of the Sibuya distribution given by \Cref{df:sibuya}. 
\end{itemize}
\end{theorem}
\begin{proof}[Sketch of proof]
Let $\ep$ denote the law of the Ewens--Pitman partition with parameter $(\alpha,\theta)$. Then, 
(A) can be proved by applying the martingale convergence theorem to the likelihood ratio 
$(\diff \ep/\diff {\mathbb{P}_{\alpha, 0}})|_{\mathcal{F}_n}$ under ${\mathbb{P}_{\alpha, 0}}$, where $\mathcal{F}_n$ is the $\sigma$-field generated by the partition of $n$ balls.
For (B), Kingman's representation theorem implies that the Ewens--Pitman partition can be expressed as the tied observation of 
conditional i.i.d. samples from the Pitman--Yor process (see \Cref{sec:compare} or \cite[p. 440]{ghosal2017fundamentals}). Then, we can analyze $S_{n,j}/K_n$ in the setting of a classical occupancy problem. Readers may refer to \cite[Theorem 3.8]{pitman2006combinatorial} for the detailed proof of (A) and \cite[Lemma 3.11]{pitman2006combinatorial}\cite{gnedin2007notes} for the proof of (B).
\end{proof}

\begin{remark}[{\cite[p. 71]{pitman2006combinatorial}}]\label{rm:abs_cont}
${\mathbb{P}_{\alpha, \theta}}$ are absolutely mutual continuous for each $\theta (> -\alpha)$: the Radon--Nikodym density is given by
$
\diff \ep/\diff {\mathbb{P}_{\alpha, 0}} = ({\mathsf{M}_\alpha})^{\theta/\alpha} {\Gamma(\theta + 1)}/{\Gamma(\theta/\alpha + 1)}$ (${\mathbb{P}_{\alpha, 0}}\text{-a.s.}$), 
where ${\mathsf{M}_{\alpha}}$ is the almost sure limit of $n^{-\alpha}K_n$ under ${\mathbb{P}_{\alpha, 0}}$. This is consistent with \Cref{prop:asym_info} below in the sense that the Fisher information about $\theta$ is bounded as $n$ increases. Roughly speaking,  this result implies that we cannot consistently estimate $\theta$. 
\end{remark}

\subsection{Stable convergence}\label{subsec:stable}
Our main theorem on the asymptotic law of the MLE (\Cref{th:smle_asym}) involves a stable convergence, which is a notion of stochastic convergence stronger than the usual weak convergence. 
In this section, we introduce it in a general format.
Let $(\Omega, \mathcal{F}, P)$ denote a probability space, and let $\mathcal{X}$ be a separable metrizable topological space equipped with its Borel $\sigma$-field $\mathcal{B}(\mathcal{X})$. Furthermore, let $\mathcal{L}^1(\Omega, \mathcal{F}, P) = \mathcal{L}^1$ be the set of $\mathcal{F}$-measurable functions that satisfy $\int |f| dP <+ \infty$, and let $C_b(\mathcal{X})$ be the set of continuous bounded functions on $\mathcal{X}$. With the above notation, the stable convergence is defined as follows:
\begin{definition}
For a sub $\sigma$-field $\mathcal{G} \subset \mathcal{F}$, a sequence of $(\mathcal{X}, \mathcal{B}(\mathcal{X}))$-valued random variables $(X_n)_{n\geq 1}$ is said to converge $\mathcal{G}$-stably to $X$, denoted by $X_n \rightarrow X$ $\mathcal{G}$-stably, iff
\begin{align}\label{eq:stable_convergence_rv}
    \forall f \in \mathcal{L}^1,\quad \forall h \in C_b(\mathcal{X}), \quad \lim_{n\rightarrow\infty} \E[f\E[h(X_n)|\mathcal{G}]]
     = \E[f\E[h(X)|\mathcal{G}]].
\end{align}
 If the limit $X$ is independent of $\mathcal{G}$, $(X_n)_{n\geq 1}$ is said to converge $\mathcal{G}$-mixing, denoted by $X_n \rightarrow X$ $\mathcal{G}$-mixing.
\end{definition}
Note that stable convergence implies the weak convergence, as
the condition \eqref{eq:stable_convergence_rv} with $f=1$ is identical to the definition of weak convergence.
By contrast, if $\mathcal{G}$ is a trivial $\sigma$-field $\{\emptyset, \Omega\}$, then we have $\E[f\E[h(X_n)|\mathcal{G}]] = \int f \diff P \cdot \E[h(X_n)]$ for all $f\in \mathcal{L}^1$. Thus, $X_n \rightarrow X $  {$\mathcal{G}$-stably} coincides with the usual weak convergence $X_n \to^d X$ in the trivial case $\mathcal{G}=\{\emptyset, \Omega\}$. 

The next lemma states that the well-known theorem for weak convergence holds for stable convergence. More precisely, {Slutsky's} lemma holds in a stronger sense. 
\begin{lemma}[{{\cite[p. 34]{hausler2015stable}}}]\label{th:CS_stable}
  For a pair of separable metrizable spaces $(\mathcal{X}, \mathcal{B}(\mathcal{X})), (\mathcal{Y}, \mathcal{B}(\mathcal{Y}))$ with metric $d$, 
let $(X_n)_{n\geq 1}$ be a sequence of $(\mathcal{X}, \mathcal{B}(\mathcal{X}))$-valued random variables, and let $(Y_n)_{n\geq 1}$ be a sequence of $(\mathcal{Y}, \mathcal{B}(\mathcal{Y}))$-valued random variables.
Assuming that a certain random variable $X$ exists s.t. $X_n \rightarrow X$ $\mathcal{G}$-stably, the following statements hold.
\begin{itemize}
    \item[(A)] Let $\mathcal{X} = \mathcal{Y}$. 
    If $d (X_n, Y_n) \rightarrow 0$ in probability, $Y_n \rightarrow X\ \mathcal{G}\text{-stably}.$
    \item[(B)] If {$Y_n \rightarrow Y$} in probability, and $Y$ is {$\mathcal{G}$-measurable}, $(X_n, Y_n) \rightarrow (X, Y) \ \mathcal{G}\text{-stably}$.
    \item[(C)] If $g: \mathcal{X} \rightarrow \mathcal{Y}$ is $(\mathcal{B}(\mathcal{X}), \mathcal{B}(\mathcal{Y}))$-measurable and continuous $P^{X}$-a.s., 
    $g(X_n) \rightarrow g(X) \  \mathcal{G}\text{-stably}.$
\end{itemize}
\end{lemma}
If $\mathcal{G}$ is a tribal $\sigma$-field $\{\emptyset, \Omega\}$, the above assertions are the well-known results for weak convergence. Importantly, (B) allows $Y$ to be any $\mathcal{G}$-measurable random variable and not just a constant. In this sense, Slutsky's lemma holds strongly for stable convergence. 
In our theorem, we will set $\mathcal{G}$ as the limit of the sigma fields generated by the sequential partition generated by the Ewens--Pitman partition. 

\section{Main result}\label{sec:main}
\subsection{Fisher Information}\label{subsec:fisher}
{In the following}  we will assume $0<\alpha<1, \theta >-\alpha$.
Before introducing our main theorem, let us discuss the asymptotic analysis of Fisher information to acquire insights into the parameter estimation of the Ewens--Pitman partition.

Let $I_\alpha$ be the Fisher information of the Sibuya distribution
\begin{align}\label{eq:sibuya_fisher}
I_\alpha :=  -\sum_{j=1}^\infty p_\alpha(j)\cdot  \partial_\alpha^2 \log p_\alpha(j) \quad \text{ with } \quad  p_\alpha(j) = \frac{\alpha \prod_{i=1}^{j-1}(i-\alpha)}{j!},
\end{align}
for all $\alpha\in(0,1)$.
The next proposition provides two formulas for $I_\alpha$.
\begin{proposition}\label{prop:sibuya_fisher}
$I_\alpha$ is continuous in $\alpha\in (0,1)$ and can be written by
\begin{align*}
    I_\alpha \overset{(A)}{=} \frac{1}{\alpha^2}+ \sum_{j=1}^\infty p_\alpha(j) \sum_{i=1}^{j-1} \frac{1}{(i-\alpha)^{2}}
    \overset{(B)}{=} \frac{1}{\alpha^2} + \sum_{j=1}^\infty \frac{p_\alpha(j)}{\alpha(j-\alpha)} > 0.
\end{align*}
\end{proposition}
We will use the two formulas in the proof of our main results. In particular, (A) will appear in the limit of the second derivative of the log-likelihood of the Ewens--Pitman partition, while (B) will be used in the limit of the variance of the first derivative of the log-likelihood. Furthermore, we will see later that our proposed confidence interval of $\alpha$ (see \Cref{cor:mixing_convergence}) requires the computation of $I_\alpha$. In this situation, we recommend (B) in terms of numerical errors: $p_\alpha(j) = O(j^{-\alpha-1})$ by Stirling's formula implies that the numerical error caused by truncating the infinite series of (A) at $n$ is 
 $\sum_{j=n}^\infty p_\alpha(j) \sum_{i=1}^{j-1} (i-\alpha)^{-2} = O(n^{-\alpha})$,  while $\sum_{j=n}^\infty p_\alpha(j) ({\alpha(j-\alpha)})^{-1} = O(n^{-\alpha - 1})$ for (B), where the error in (B) decays faster than that in (A).
We plot $I_\alpha$ in \Cref{fig:I_graph}
using the formula (B) with $j$ truncated at $10^5$ for each $\alpha$. {\Cref{fig:I_graph} suggests that $\alpha\mapsto I_\alpha$ is log-convex.} 
 \begin{figure}[htbp]
     \centering
     \includegraphics[width=0.45\linewidth]{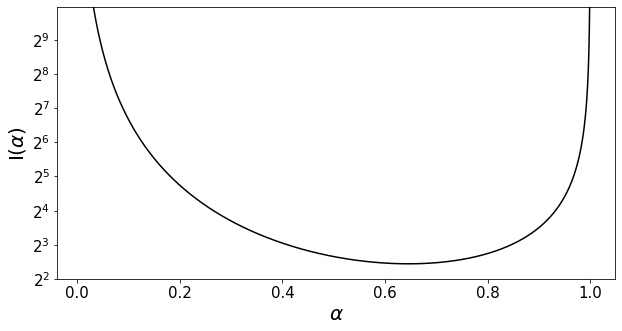}
     \caption{Plot of the Fisher Information $I_\alpha$ in \eqref{eq:sibuya_fisher}}
     \label{fig:I_graph}
 \end{figure} 
 
For the asymptotic analysis of the Fisher information of the Ewens--Pitman partition ahead, we define the function $f_\alpha: (-1,\infty) \rightarrow \mathbb{R}$ for each $\alpha\in (0,1)$ by 
\begin{align}\label{eq:df_f_alpha}
  \forall z\in (-1,\infty), \quad  f_\alpha(z) :=  \psi(1 + z) - \alpha \psi( 1 + \alpha z),
\end{align}
where $\psi(x) = \Gamma'(x)/\Gamma(x)$ is the digamma function. The next lemma claims some basic properties of $f_\alpha$.
\begin{lemma}\label{lm:f_alpha}
The map $f_\alpha: (-1,\infty) \rightarrow \mathbb{R}$ defined by \eqref{eq:df_f_alpha} is bijective and satisfies $f'_\alpha(z) > 0$ and
$f''_\alpha(z) < 0$ for all $z\in(-1,\infty)$.
\end{lemma}
Here, it is important to emphasize that $f_\alpha$ is bijective. We will see later that $f_\alpha$ also appears in the asymptotics of the MLE for $\theta$ through its inverse function $f_\alpha^{-1}$. 
 
Finally, we discuss the Fisher information of the Ewens--Pitman partition. 
We denote the logarithm of the likelihood \eqref{eq:ep_density} by $\ell_n(\alpha, \theta)$, and we define $I_{{\alpha,\alpha}}^{(n)}$, $I_{{\alpha, \theta}}^{(n)}$, and 
$I_{{\theta,\theta}}^{(n)}$ by
\begin{align}\label{eq:df_ep_info}
\begin{split}
        I_{{\alpha,\alpha}}^{(n)} &:= \E[(\partial_\alpha \ell_n(\alpha, \theta))^2], \\  I_{{\alpha, \theta}}^{(n)} &:= \E[\partial_{\alpha} \ell_n(\alpha, \theta)\cdot \partial_{\theta} \ell_n(\alpha, \theta)], \\
        I_{{\theta,\theta}}^{(n)} &:= \E[(\partial_\theta \ell_n(\alpha, \theta))^2],
\end{split}
\end{align}
i.e.,  they are the Fisher information obtained after $n$ balls are partitioned according to the Ewens--Pitman partition $(\alpha,\theta)$. The next proposition derives the leading terms as $n\to \infty$. 
\begin{proposition}\label{prop:asym_info}
Let $I_\alpha$ be the Fisher information of the Sibuya distribution \eqref{eq:sibuya_fisher}, and let $f'_\alpha$ be the derivative of $f_\alpha$ defined by \eqref{eq:df_f_alpha}. Then, the
leading terms of the Fisher information are given by
\begin{align*}
    I_{{\alpha,\alpha}}^{(n)} \sim {n^\alpha} \E[\Mit] I_\alpha, \ \ 
    I_{{\alpha, \theta}}^{(n)} \sim  \alpha^{-1} \log n, \ \ 
    I_{{\theta,\theta}}^{(n)} \rightarrow \alpha^{-2} f'_\alpha(\theta/\alpha) <+\infty,
\end{align*}
where $\E[\Mit]$ is the moment of $\gmit$ given by \eqref{eq:gmit_moment}. 
\end{proposition}
We observe that the Fisher information of $\theta$ is finite, which means that  $\theta$ can not be consistently estimated no matter how large $n$ is. This agrees with the absolute mutual continuity given by \Cref{rm:abs_cont}. On the other hand, the optimal convergence rate of estimators for $\alpha$ is at most $n^{-\alpha/2}$, which is slower than the typical rate $n^{-1/2}$ in typical i.i.d. cases. Note in passing that the cross term of the Fisher information matrix $I_{{\alpha, \theta}}^{(n)}$ is negligible compared to $I_{{\alpha,\alpha}}^{(n)}$, which implies that 
$\alpha$ and $\theta$ are asymptotically orthogonal (see \Cref{fig:vis_mle_pmle}). This supports the well-known fact that the inference of $\theta$ has less effect on $\alpha$ as $n$ increases (see \cite{balocchi2022bayesian, franssen2022empirical}). 

\subsection{Maximum Likelihood Estimator}\label{subsec:mle}
In this section, we derive the exact asymptotic distribution of the maximum likelihood estimator. Recall that $\ell_n(\alpha,\theta)$ is the logarithm of the likelihood \eqref{eq:ep_density}, which can be written as 
\begin{align}
  &\ell_n(\alpha, \theta) 
    = \sum_{i=1}^{K_n-1}\log (\theta + i\alpha) - \sum_{i=1}^{n-1}\log (\theta+i) + \sum_{j=2}^n S_{n,j} \sum_{i=1}^{j-1}\log (i-\alpha).\label{eq:df_loglikelihood}
\end{align}
Then, the MLE $(\hat{\alpha}_n, \hat{\theta}_n)$ is defined as the maxima of $\ell_n$. 
\begin{definition}\label{df:smle}
Define the maximum likelihood estimator $(\hat{\alpha}_n,\hat{\theta}_n)$ by
\begin{align*}
    (\hat{\alpha}_n,\hat{\theta}_n) &\in \argmax_{\alpha\in {(0,1)}, \theta>-\alpha} \ell_n(\alpha, \theta). 
\end{align*}
\end{definition}
As the parameter space $\{(\alpha, \theta): \alpha \in {(0,1)}, \theta>-\alpha\}$ is not compact, the existence and uniqueness of the MLE are not obvious. The next theorem verifies the existence and uniqueness
rigorously. 
{
\begin{proposition}\label{prop:smle_unique}
With probability $1-o(1)$, the MLE uniquely exists.
\end{proposition}
}
Now, let us define the $\sigma$-field $\mathcal{F}_\infty := \sigma (\cup_{n=1}^\infty \mathcal{F}_n)$ where $\mathcal{F}_n$ is the $\sigma$-field generated by the partition of $n$ balls following the Ewens--Pitman partition (see \Cref{subsec:asym_ep}). 
Then, the exact asymptotic distribution of the MLE $(\hat{\alpha}_n, \hat{\theta}_n)$ is characterized as follows:
\begin{theorem}\label{th:smle_asym}
Let $I_\alpha$ be the Fisher information of the Sibuya distribution defined by \eqref{eq:sibuya_fisher}, and let $f_\alpha^{-1}$ be the inverse of the bijective function $f_\alpha$ defined by \eqref{eq:df_f_alpha}. Then, the asymptotics of $(\hat{\alpha}_n, \hat{\theta}_n)$ is given by
\begin{align}
  n^{\alpha/2}(\hat{\alpha}_n - \alpha) &\to (I_\alpha \Mit )^{-1/2} \cdot \N  &&(\stable), \label{eq:asymptotic_law_alpha} \\
  \hat{\theta}_n &\to \alpha \cdot f_\alpha^{-1}(\log \Mit) && (\text{in probability}), \label{eq:asymptotic_law_thtea}
\end{align}
where  $\N\sim \normal(0,1)$ is independent of $\mathcal{F}_\infty$, and 
$\Mit = \lim_{n\to \infty} n^{-\alpha} K_n$ is a nondegenerate positive random variable following $\gmit$ (see \Cref{df:gmit}). 
\end{theorem}
Noting that the leading term of the Fisher information about $\alpha$, denoted by $I_{{\alpha,\alpha}}^{(n)}$, is given by ${n^\alpha} \E[\Mit] I_\alpha$ (see \Cref{prop:asym_info}), \eqref{eq:asymptotic_law_alpha} gives 
\begin{align}\label{eq:mixed_normal}
    \sqrt{I_{{\alpha,\alpha}}^{(n)}} (\hat{\alpha}_{n} - \alpha) = \sqrt{I_{{\alpha,\alpha}}^{(n)}/n^\alpha} \cdot n^{\alpha/2} (\hat{\alpha}_{n} - \alpha) \to \sqrt{\E[\Mit]/\Mit}\cdot \N.
\end{align}
As $\Mit$ is a non-degenerate random variable, \eqref{eq:mixed_normal} implies that the error of the MLE normalized by the Fisher information does not converge to the standard normal but a variance mixture of centered normals. 
This type of asymptotics is referred to as \textit{asymptotic mixed normality}, which is often observed in ``nonergodic'' or ``explosive'' stochastic processes (c.f. \cite{hausler2015stable}).

By contrast, {Slutsky's} lemma for stable convergence (more precisely, (B) of \Cref{th:CS_stable} with $X_n = \sqrt{n^{\alpha}I_\alpha}(\hat{\alpha}_{n} - \alpha)$ and $Y_n = \sqrt{K_n/n^\alpha}$) results in 
\begin{align*}
  \sqrt{K_n I_\alpha } (\hat{\alpha}_{n} - \alpha) = \sqrt{K_n/n^\alpha} \cdot \sqrt{n^\alpha I_\alpha} (\hat{\alpha}_{n} - \alpha)
  \to \sqrt{\Mit} \cdot \N/\sqrt{\Mit} = \N.
\end{align*}
We observe that the randomness of the variance is now canceled out, and the limit is the standard normal. Here, the number of blocks $K_n$ corresponds to the sample size in typical i.i.d. cases and $I_\alpha$ plays the role of the Fisher information per block. 
Furthermore, it immediately follows that $[\hat{\alpha}_n \pm {1.96}/{\sqrt{I_{\hat{\alpha}_n} K_n}}]$ is {an approximate} $95\%$ confidence interval for $\alpha$. We reiterate these observations as a corollary below.

\begin{corollary}\label{cor:mixing_convergence} Let $I_\alpha$ be the Fisher information of the Sibuya distribution defined by \eqref{eq:sibuya_fisher} and
let $K_n$ be the number of blocks generated after $n$ balls are partitioned. Then, 
in the same setting as \Cref{th:smle_asym},
the following mixing convergence holds:
\begin{align}\label{eq:mixing}
  \sqrt{I_\alpha K_n} \cdot (\hat{\alpha}_{n} - \alpha) \rightarrow \N \quad (\mathcal{F}_\infty\text{-mixing}).
\end{align}
Therefore, 
{for any $p\in (0,1)$, letting $ \tau_{1-p/2}$ be the $(1-p/2)$-quantile of standard normal, 
the interval $\hat{I}_n := [\hat{\alpha}_n \pm \tau_{1-p/2}/{\sqrt{I_{\hat{\alpha}_n} K_n}}]$ is an approximate $100(1-p)\%$ confidence interval in the sense of 
$\lim_{n\to+\infty} \Pr\bigl(\alpha\in \hat{I}_n \bigr) = 1-p$. 
}
\end{corollary}

Finally, we discuss the limit law $\alpha f_\alpha^{-1}(\log \Mit)$ of $\hat{\theta}_n$ in \Cref{th:smle_asym}. We claim that the limit distribution of the MLE $\hat{\theta}_n$ of $\theta$ is positively biased. The key observation is the strict convexity of $f_\alpha^{-1}$. Indeed, $f_\alpha$ is strictly concave and strictly increasing (see \Cref{lm:f_alpha}), so $f_\alpha^{-1}$ is strictly convex. Since $\Mit\sim \gmit$ is not constant, Jensen's inequality gives the strict inequality:
  $$
  \E[\alpha f_\alpha^{-1}(\log \Mit)]  > \alpha f_\alpha^{-1} (\E[\log \Mit])
  $$
  {where $\E$ is the expectation with respect to $\Mit\sim \gmit$. Since $\epsilon \mapsto (c^\epsilon-1)/\epsilon$ is increasing for any $c>0$, the monotone convergence theorem yields 
  \begin{align}\label{eq:replica}
      \E[\log \Mit] =\E[\lim_{\epsilon\rightarrow 0} \epsilon^{-1} ((\Mit)^\epsilon - 1)] = \lim_{\epsilon\rightarrow 0}
      \E[\epsilon^{-1}((\Mit)^\epsilon-1)]
  \end{align}
  }
By the moment formula of $\Mit\sim \gmit$ in \eqref{eq:gmit_moment}, we have that  
  \begin{align*} 
      \lim_{\epsilon\rightarrow 0}
      \E[\epsilon^{-1}((\Mit)^\epsilon-1)] &= \partial_\epsilon|_{\epsilon=0} \E[(\Mit)^\epsilon]\\
      &= \partial_\epsilon|_{\epsilon=0} \left(\frac{\Gamma(\theta + 1)}{\Gamma(\theta/\alpha + 1)} \frac{\Gamma(\theta/\alpha + \epsilon + 1)}{\Gamma(\theta + \epsilon\alpha + 1)}\right) && \text{by \eqref{eq:gmit_moment}} \\
      &= f_\alpha(\theta/\alpha) && \text{by \eqref{eq:df_f_alpha}}.
  \end{align*}
  Combining the above display together, we are left with 
  \begin{align*}
      \E[\alpha f_\alpha^{-1}(\log \Mit)] &>\alpha f_\alpha^{-1} (\E[\log \Mit])
      \\
      &= \alpha f_\alpha^{-1} (\lim_{\epsilon\rightarrow 0}
      \E[\epsilon^{-1}((\Mit)^\epsilon-1)])\\
      &= 
      \alpha f_\alpha^{-1} (f_\alpha(\theta/\alpha))\\
      &= \theta.
  \end{align*}
  We present the result obtained above in the following proposition. 
{ 
\begin{proposition}\label{prop:biased}
  The limit distribution of $\hat{\theta}_n$ is biased, i.e., $\E[\alpha \cdot f_{\alpha}^{-1}(\log\Mit)] >  \theta$. 
\end{proposition}
}
In \Cref{sec:Numeric}, we will plot the histogram $\alpha f_{\alpha}^{-1}(\log \Mit)$ by sampling $\Mit \sim \gmit$ and confirm the positive bias $\E[\alpha f_\alpha^{-1}(\log \Mit)] > \theta$. 

\subsection{Quasi Maximum Likelihood Estimator}\label{subsec:qmle}
In the previous section, we considered the simultaneous estimation of $\alpha$ and $\theta$. However, as we mentioned in \Cref{sec:intro}, the estimation of $\alpha$ is of more interest than that of $\theta$, and $\theta$ is sometimes regarded as a nuisance parameter in practice. 
In this section, we consider the MLE of $\alpha$ with $\theta$ being miss-specified. 
Considering the asymptotic orthogonality of  $(\alpha, \theta)$ (recall \Cref{prop:asym_info}), the MLE of $\alpha$ with $\theta$ miss-specified is expected to have the same asymptotic law as the MLE with $\theta$ jointly estimated. In this section, we make these arguments more rigorous:  we claim that they are identical up to the order of $n^{-\alpha/2}$, but they differ in higher order (see \Cref{fig:vis_mle_pmle}). Furthermore, we demonstrate that the MLE with $\theta$ jointly estimated is adaptive to the scale of the nuisance $\theta$. 

\begin{figure}[htbp]
    \centering
    \includegraphics[width=0.80\linewidth]{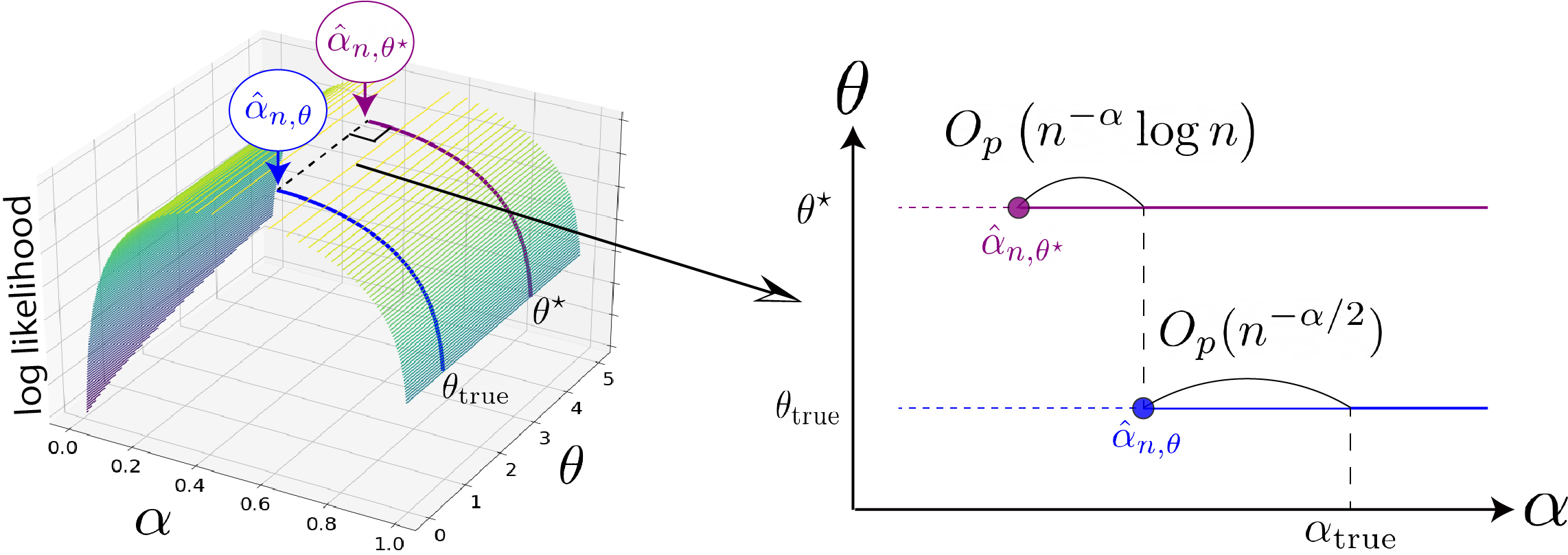}
    \caption{Asymptotic orthogonality of $\alpha$ and $\theta$.}
    \label{fig:vis_mle_pmle}
\end{figure}

First, we define the quasi maximum likelihood estimator (QMLE) as the MLE of $\alpha$ with $\theta$ being miss-specified. 
\begin{definition}[QMLE]\label{df:qmle}
For each $\plug\in(-\alpha,\infty)$,  we define the QMLE with plug-in $\plug$, denoted by $\hat{\alpha}_{n,\plug}$, as follows
\begin{align}
  \hat{\alpha}_{n,\plug}
  &\in \underset{\alpha \in ((-\plug) \vee 0 , 1)}{\argmax}
   \ \ell_n(\alpha, \plug), \label{eq:pmle_df}
\end{align}
where $\ell_n$ is the function defined by \eqref{eq:df_loglikelihood}.
\end{definition}
{Note that the true parameter $(\alpha, \theta)$ must satisfy $\alpha\in (0,1)$ and $\theta+\alpha>0$, so for each $\theta>-1$, $\alpha$ belongs to the subset $((-\theta)\vee 0, 1)$ of $ (0,1)$. Thus, the definition of the QMLE; i.e.,  $\argmax_{\alpha\in (-(\plug)\vee 0, 1)}$, is natural, and it contains the true $\alpha$ as long as we set $\plug\ge 0$. }

We emphasize that the QMLE $\hat{\alpha}_{n, \theta} = \hat{\alpha}_{n,\plug=\theta}$ with $\plug$ being the true $\theta$ is just the MLE of $\alpha$ with $\theta$ known. From now on, we regard this as an oracle estimator of $\alpha$, and we will compare the QMLE $\hat{\alpha}_{n, \plug}$ (with $\plug \neq \theta$) and the MLE $\hat{\alpha}_n$ in \Cref{df:smle} (where $\theta$ is jointly estimated) based on their error to the oracle $\hat{\alpha}_{n, \theta}$. 

The following propositions claim that $\hat{\alpha}_{n,\plug}$ uniquely exists and has the same asymptotic distribution as $\hat{\alpha}_n$. 
\begin{proposition}\label{prop:qmle_unique}
{For any $\plug>-\alpha$, with probability $1-o(1)$, the QMLE $\hat{\alpha}_{n,\plug}$ uniquely exists.} 
\end{proposition}
\begin{proposition}\label{prop:qmle_stable}
\eqref{eq:asymptotic_law_alpha} of \Cref{th:smle_asym} holds for QMLE $\hat{\alpha}_{n, \plug}$, i.e., 
\begin{align*}
      n^{\alpha/2}(\hat{\alpha}_{n, \plug} - \alpha) &\to (I_\alpha \Mit )^{-1/2} \cdot \N \ \  (\stable).
\end{align*}
\end{proposition}\Cref{prop:qmle_stable} implies that the QMLE $\hat{\alpha}_{n, \plug}$ and $\hat{\alpha}_{n}$ are asymptotically equivalent on the scale of $n^{-\alpha/2}$. With this, it appears that jointly estimating $\alpha$ and $\theta$ is of no use. However, the next proposition implies that they differ on the order of $n^{-\alpha} \log n$ and $\hat{\alpha}_n$ is close to the oracle $\hat{\alpha}_{n, \theta}$ regardless of the scale of $\theta$.
\begin{proposition}\label{prop:error}
 For the QMLE $\hat{\alpha}_{n,\plug}$ and the MLE $\hat{\alpha}_n$, 
their asymptotic errors to the oracle $\hat{\alpha}_{n,\plug=\theta}$ are given by
\begin{align}
    \frac{n^\alpha}{\log n}(\hat{\alpha}_{n,\plug} - \hat{\alpha}_{n, \theta}) &\to^p -\frac{\plug- \theta}{\alpha I_\alpha \Mit} \qquad  \text{for all $\plug \in (-\alpha, \infty)$}, \label{eq:qmle_mle_error}  \\
    \frac{n^\alpha}{\log n}(\hat{\alpha}_n - \hat{\alpha}_{n,\theta}) &\to^p -\frac{\alpha f_\alpha^{-1}(\log \Mit) - \theta}{\alpha I_\alpha \Mit}, \label{eq:smle_mle_error}
\end{align}
where $\Mit=\lim_{n\to\infty} n^{-\alpha}K_n$, $I_\alpha$ is defined by \eqref{eq:sibuya_fisher}, and $f_\alpha^{-1}$ is the inverse of $f_\alpha$ defined by \eqref{eq:df_f_alpha}.
\end{proposition}
We observe that the limit error in \eqref{eq:qmle_mle_error} depends on the ``misspecification error'' $\plug- \theta$,  while the corresponding term in \eqref{eq:smle_mle_error} is replaced by $\alpha f_\alpha^{-1}(\log \Mit)-\theta$. Considering that $\alpha f_\alpha^{-1}(\log \Mit)$ is distributed around $\theta$ as in \Cref{fig:theta_limit}, we expect that the error of $\hat{\alpha}_{n, \plug}$ is larger than $\hat{\alpha}_{n}$ if the plug-in $\plug$ is taken far away from the true value $\theta$ by the users. We prove in \Cref{sec:Numeric} that these errors significantly affect coverage and {mean squared error.} 

\subsection{Application to network data analysis}\label{subsec:network}
Here, we discuss the application of \Cref{cor:mixing_convergence} to network data analysis.
In \cite{crane2018probabilistic, crane2016edge}, the authors propose the ``Hollywood process”,  a statistical model for network data. This is a stochastic process over growing networks that sequentially attach edges to vertices in the same manner as the Ewens--Pitman partition, where $n$ is the total degree,  $K_n$ is the number of vertices, and $S_{n,j}$ is the number of vertices with degree $j$. They define that a growing network has sparsity if and only if 
$\lim_{n\to\infty} n K_n^{-\mu} = 0$, where 
$\mu$ is the degree per vertex, e.g., $\mu = 2$ when the network is bivariate. Using the asymptotics $n^{-\alpha} K_n\to \Mit >0$ (a.s.) by \Cref{lm:ep_asym}, the authors claim that the Hollywood process has sparsity if and only if $\mu^{-1} < \alpha <1$. 

Now, we construct a hypothesis testing of the sparsity based on \Cref{cor:mixing_convergence}. We define the null hypothesis $H_0$ and the alternative hypothesis $H_1$ by
\begin{align*}
    \text{(not sparse) } H_0: 0<\alpha\leq \mu^{-1}, \quad  \text{(sparse) } H_1: \mu^{-1} < \alpha < 1.
\end{align*}
For a constant $\delta \in (0, 1)$, we reject the null $H_0$ if 
\begin{align*}
\sqrt{I_{\hat{\alpha}_n} K_n} (\hat{\alpha}_n - \mu^{-1})> \Phi^{-1}(1-\delta),
\end{align*}
where $\Phi$ is the CDF of the standard normal. 
Then, the significance level of this testing is $\delta$, since
the probability of rejecting the null when $\alpha \leq \mu^{-1}$ is upper bounded by $\Pr(\sqrt{I_{\hat{\alpha}_n} K_n} (\hat{\alpha}_n - \alpha )> \Phi^{-1}(1-\delta))$, which converges to $\delta$ from \Cref{cor:mixing_convergence}. 

\section{Numerical simulation}\label{sec:Numeric}
Firstly, we visualize the limit law $\alpha f_\alpha^{-1}(\log \Mit)$ of the MLE $\hat{\theta}_n$. We sample $\Mit \sim \gmit$ using the rejection algorithm proposed by \cite{devroye2009random} with a sample size $10^6$ and plot the histogram of $\alpha f_\alpha^{-1}(\log \Mit)$ in \Cref{fig:theta_limit}.
We observe that the mean of $\alpha f_\alpha^{-1}(\log \Mit)$ is larger than $\theta$, which {verifies \Cref{prop:biased}}. We also plot the probability density function (pdf) of 
$\normal(\theta, \alpha^2 /f_\alpha'(\theta/\alpha))$, where the variance $\alpha^2/f_\alpha'(\theta/\alpha)$ is the inverse of 
$\lim_{n\to+\infty} \E[(\partial_{\theta} \ell_n(\alpha,\theta))^2]$, i.e., the limit of the Fisher information about $\theta$ (see \Cref{prop:asym_info}).
Note that this normal distribution is a naive guess by standard asymptotic statistics. We observe that 
the limit law is close to the normal distribution $\normal(\theta, \alpha^2 /f_\alpha'(\theta/\alpha))$ when $\alpha$ is small or $\theta$ is large. 

\begin{figure}[htbp]
    \centering
    \includegraphics[width=0.99\linewidth]{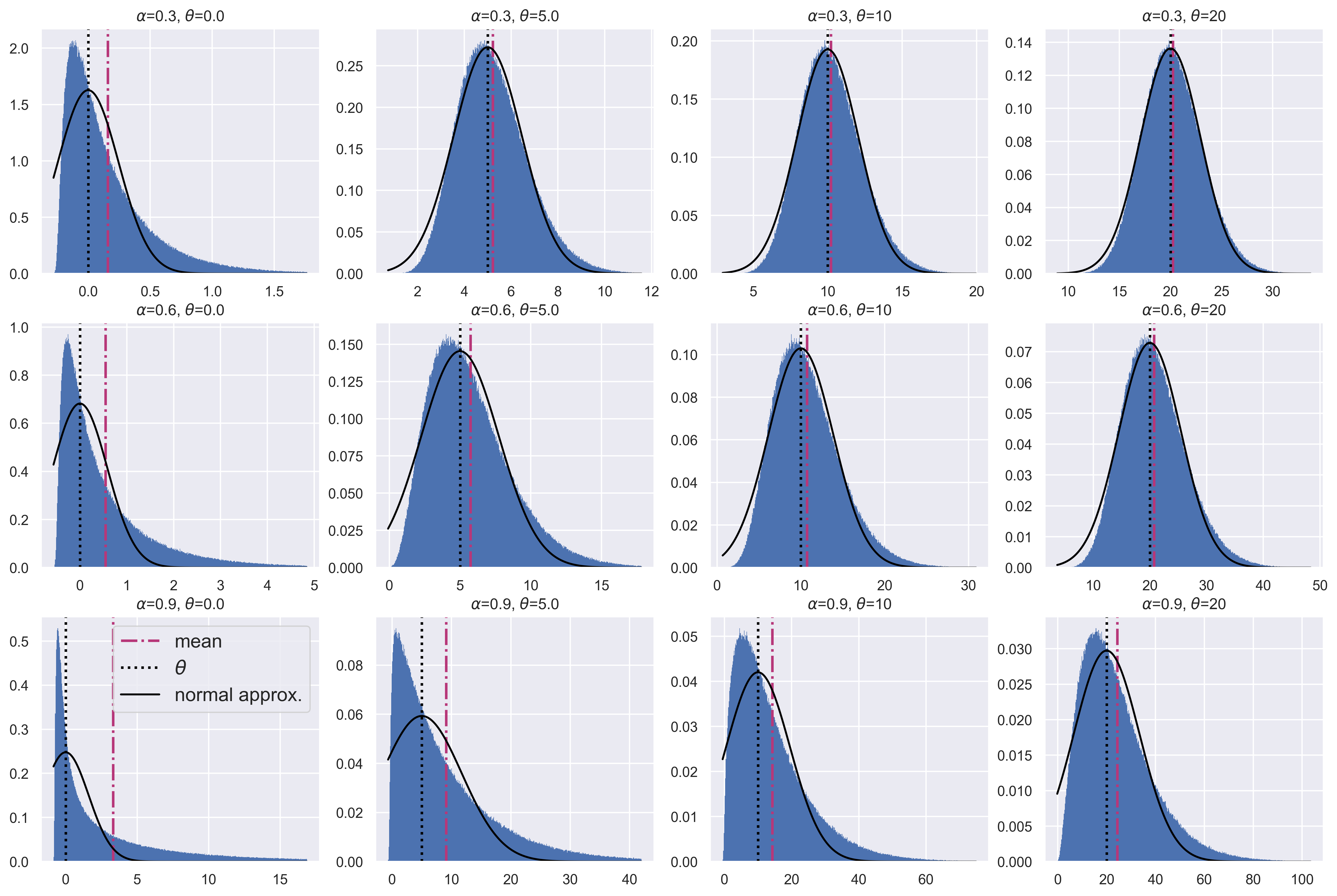}
    \caption{
      {Histogram of $\alpha f_\alpha^{-1}(\log \Mit)$ with sample size $10^6$.
    The solid line is the pdf of $\normal(\theta, \alpha^2 /f_\alpha'(\theta/\alpha))$, where the variance is the inverse of the asymptotic Fisher Information; that is, $\alpha^{-2} f_\alpha'(\theta/\alpha) = \lim_{n\to+\infty}\E[(\partial_{\theta} \ell_n(\alpha,\theta))^2]$. 
      }
    }
    \label{fig:theta_limit}
\end{figure}

{Next, we visualize the asymptotic mixed normality of the MLE $\hat\alpha_{n}$ by plotting the empirical CDF.  Here, we plot two CDFs using different normalization; the CDF of $\sqrt{n^\alpha \E[\Mit] I_\alpha} (\hat\alpha_{n}-\alpha)$
 and the CDF of $\sqrt{K_n I_\alpha} (\hat\alpha_{n}-\alpha)$. Note that $n^\alpha \E[\Mit] I_\alpha$ is the leading term of the Fisher Information about $\alpha$, i.e., $\E[(\partial_{\alpha} \ell_n(\alpha,\theta))^2] = n^\alpha \E[\Mit] I_\alpha + o(n^\alpha)$ (see \Cref{prop:asym_info}). 
These empirical CDFs are computed from $10^5$ Monte Carlo simulations and are plotted after subtracting the CDF of $N(0,1)$. \Cref{fig:cdf_convergence} implies that the CDF of $\sqrt{n^\alpha \E[\Mit] I_\alpha} (\hat\alpha_{n}-\alpha)$ does not converge to the CDF of $N(0,1)$, while the CDF of $\sqrt{K_n I_\alpha} (\hat\alpha_{n}-\alpha)$ does converge to the CDF of $N(0,1)$. These numerical simulations support \Cref{th:smle_asym} and \Cref{cor:mixing_convergence}. 

\begin{figure}[htbp]
  \centering
  \includegraphics[width=0.99\linewidth]{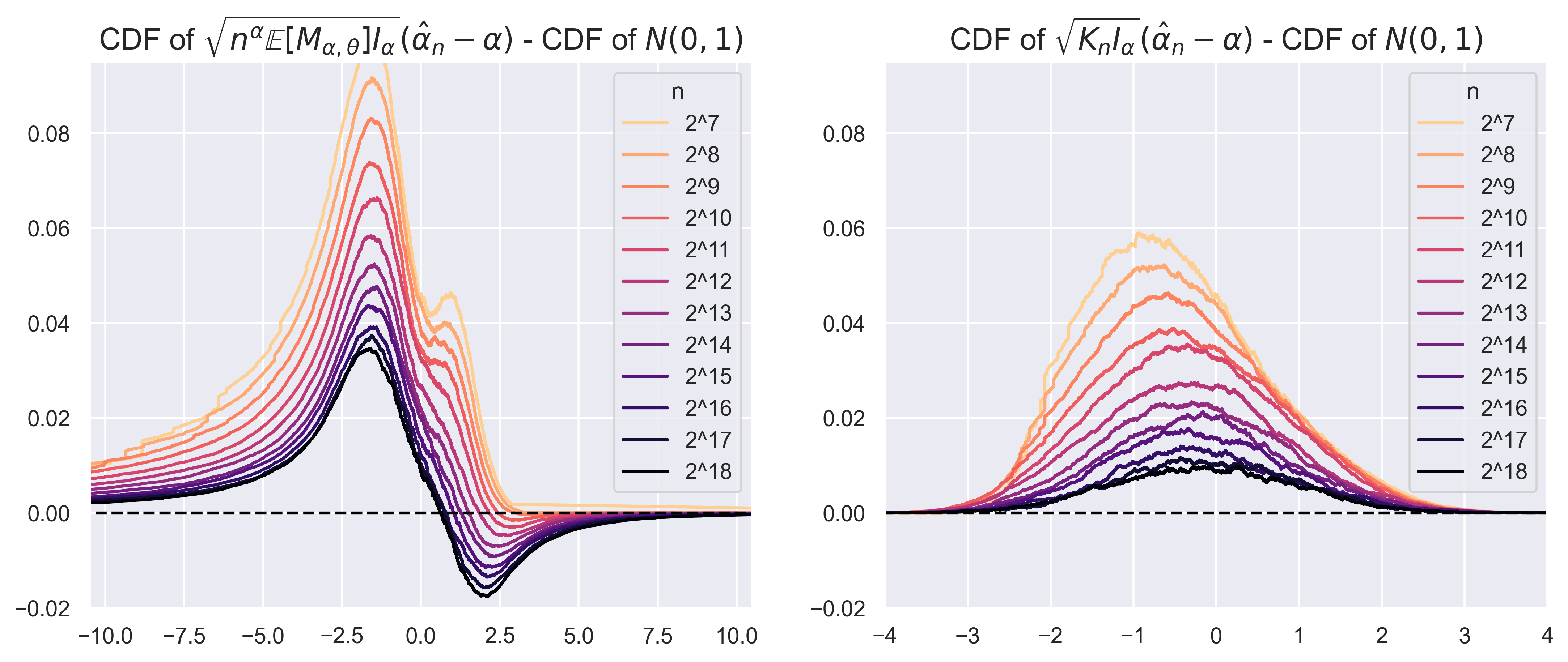}
  \caption{
    {
   The visualization of the asymptotic mixed normality. 
   The left figure plots the difference of CDF of $\sqrt{n^\alpha \E[\Mit] I_\alpha}(\hat{\alpha_{n}}-\alpha)$  to the CDF of 
   $N(0,1)$, while the right figure plots the difference of the CDF of $\sqrt{K_n I_\alpha}(\hat{\alpha}_n-\alpha)$ to the CDF of $N(0,1)$. Simulation setting: 
   $\alpha=0.8$, $\theta=0$, $100000$ Monte Carlo simulations.
    } 
  }
  \label{fig:cdf_convergence}
\end{figure}
}

Finally, we compare $\hat{\alpha}_{n, \theta}$ (the MLE with $\theta$ known), $\hat{\alpha}_{n,0}$ (the QMLE with $\plug = 0$), and $\hat{\alpha}_n$ (the MLE with $\theta$ jointly estimated) based on the {mean squared error (MSE)} and the coverage. Here, we sequentially generate the random partition and compute these estimators as $n$ increases from $n=2^7$ to $2^{17}$. This process is replicated $10^4$ times, and we calculate the {MSE} and the coverage of confidence interval $[\bar{\alpha}_n + \frac{1.96}{\sqrt{K_n I_{\bar{\alpha}_n}}}]$
for each $\bar{\alpha}_n = \hat{\alpha}_{n, \theta}, \hat{\alpha}_{n, 0}, \hat{\alpha}_n$. The results are plotted in \Cref{fig:MSE} and \Cref{fig:coverage}. As $n$ increases, the MSE decreases, and the coverage converges to $0.95$. When $n$ is small and the plugin error $|0 - \theta| = |\theta|$ is large, the MSE and the coverage of the QMLE $\hat{\alpha}_{n, 0}$ are significantly large and small, respectively. In contrast, $\hat{\alpha}_n$ is robust to the scale of $\theta$. These observations support \Cref{prop:error}.

\begin{figure}[htbp]
  \centering
  \includegraphics[width=0.95\linewidth]{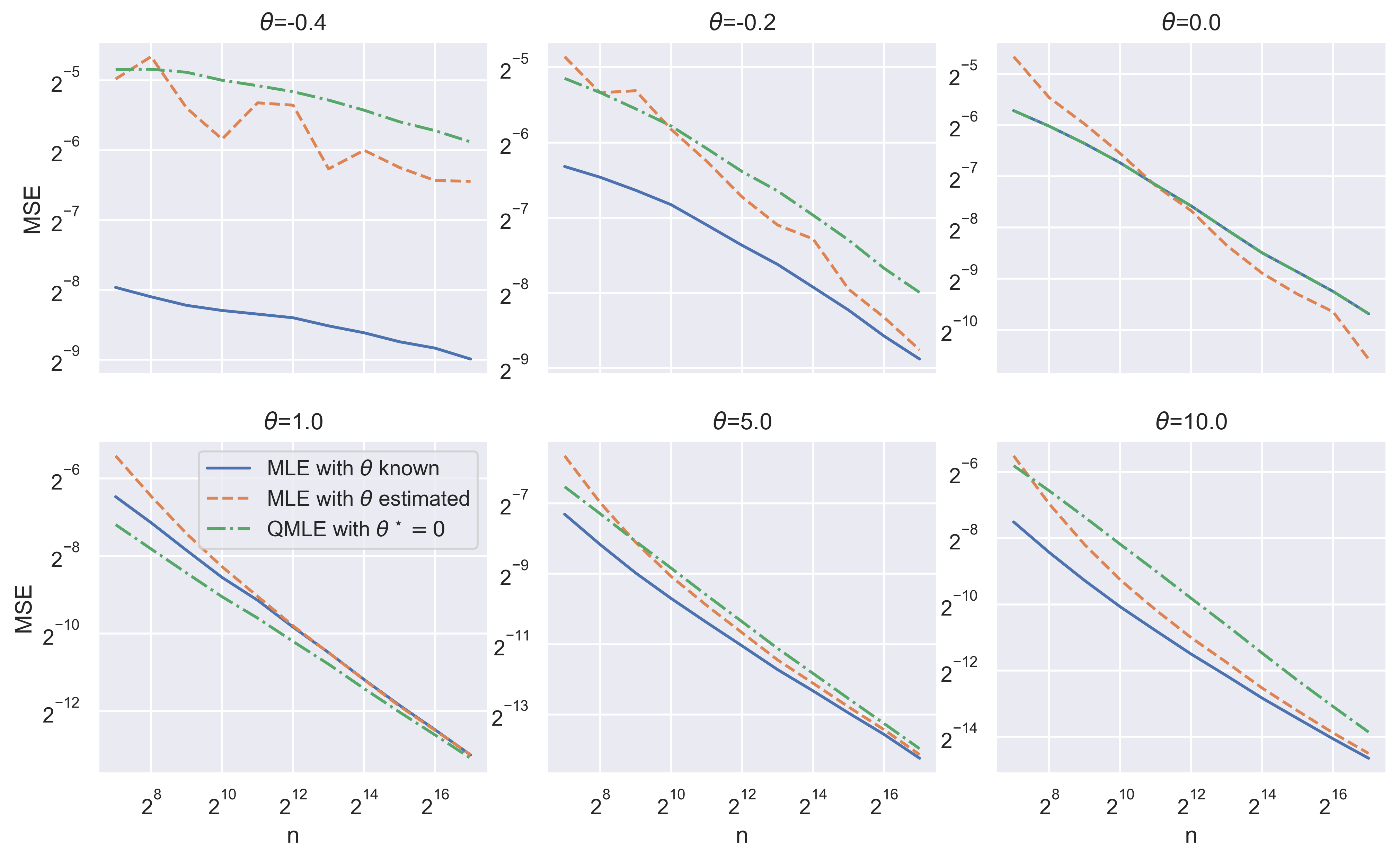}
  \caption{
    {Plots of the MSE of the MLE with $\theta$ known, the MLE with $\theta$ unknown (estimated), and the QMLE with $\plug=0$. We fixed $\alpha$ to $0.6$ and run $10^5$ Monte Carlo simulations.  Note that when $\theta=0$, the QMLE with $\plug=0$ coincides with the MLE with $\theta$ known.}}
  \label{fig:MSE}
    \includegraphics[width=0.95\linewidth]{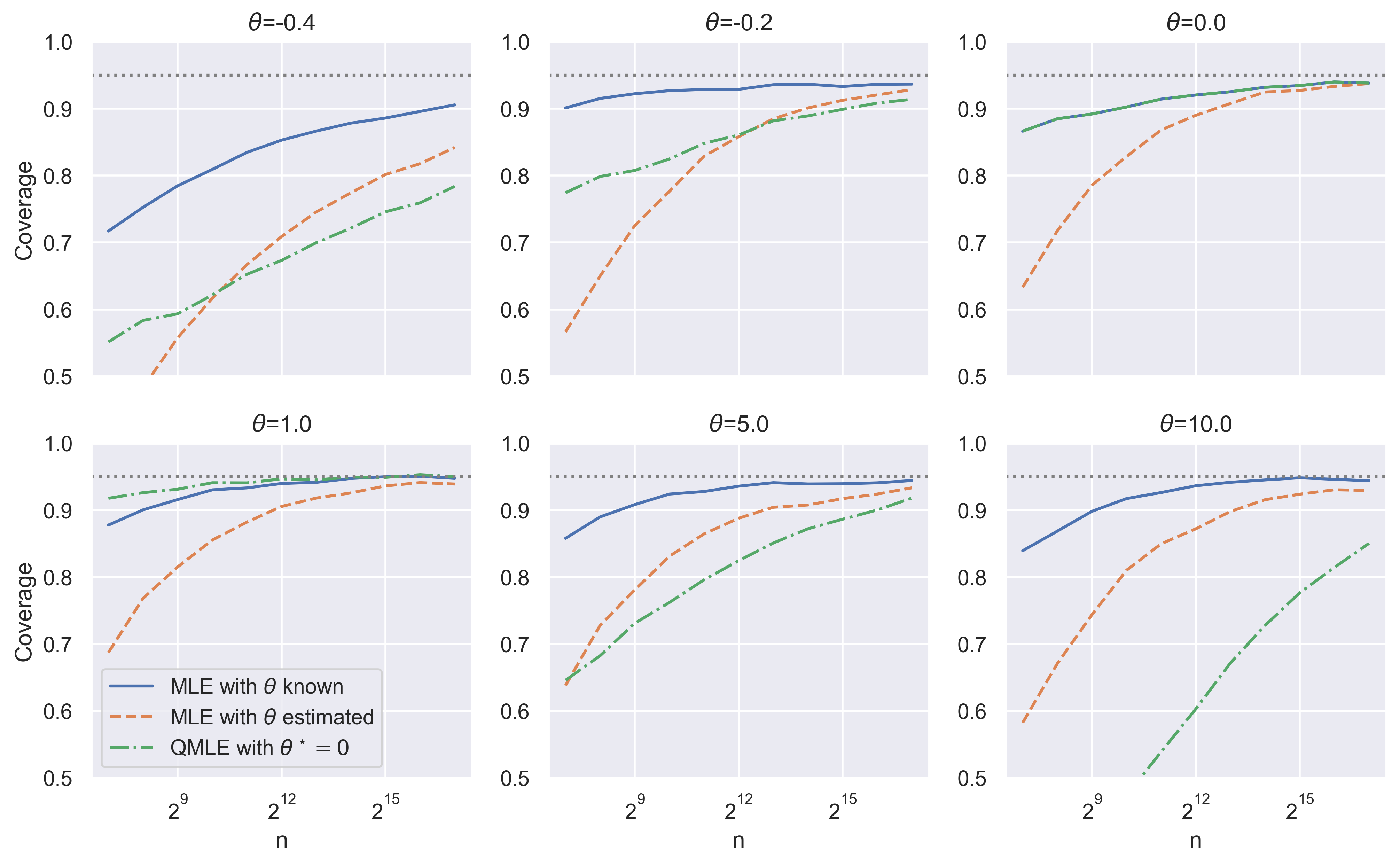}
    \caption{
      {Plots of the coverage of the MLE with $\theta$ known (estimated), the MLE with $\theta$ unknown, and the QMLE with $\plug=0$. We fixed $\alpha$ to $0.6$ and run $10^5$ Monte Carlo simulations. Note that when $\theta=0$, the QMLE with $\plug=0$ coincides with the MLE with $\theta$ known.}
    }
    \label{fig:coverage}
\end{figure}

\section{Proof highlights}\label{sec:proof}
In this section,  we outline the fundamental idea behind the proofs of theorems in \Cref{sec:main}. 
First, we consider the QMLE $\hat{\alpha}_{n,0}$ with $\plug=0$ for simplicity.
\eqref{eq:df_loglikelihood} with $\theta=0$ implies that the log-likelihood with parameter $(\alpha,\theta)=(\alpha,0)$ is given by
\begin{align*}
\ell_n(\alpha, 0) = (K_n-1) \log \alpha + \sum_{j=1}^n S_{n,j} \sum_{i=1}^{j-1}\log (i-\alpha) \quad \text{for all $\alpha>0$}, 
\end{align*}
where $S_{n,j}$ is the number of blocks of size $j$ and $K_n$ is the number of nonempty blocks. 
Then, the score function $\partial_\alpha\ell_n(\alpha, 0)$ is given by
$
\partial_\alpha \ell_n(\alpha, 0) = \frac{K_n-1}{\alpha} - \sum_{j=1}^{n} S_{n,j} \sum_{i=1}^{j-1}\frac{1}{i-\alpha}. 
$
Here, we define the random measure $\mathbb{P}_n$ on $\mathbb{N}$ as the ratio of blocks of size $j$:
\begin{align*}
\forall j \in \mathbb{N}, \quad  \mathbb{P}_n(j) := \frac{S_{n,j}}{\sum_{j'=1}^\infty S_{n,j'}} = \frac{S_{n,j}}{K_n}.
\end{align*}
Note that  $\mathbb{P}_n(j) = 0$ for all $j > n$, as the total number of partitioned balls is $n$. In our proof, we denote $\sum_{j=1}^\infty \mathbb{P}_n(j) f(j)$ by $\mathbb{P}_n f$ for any function $f$ on $\mathbb{N}$. 
Now, we define the random function $\hat{\Psi}_{n,0}(x)=K_n^{-1} \partial_\alpha\ell_n(\alpha, 0)\mid_{\alpha=x}$, which is the score function normalized by $K_n$. Then, the above displays give 
\begin{align*}
    \hat{\Psi}_{n,0}(x) &= \frac{1}{x} -\frac{1}{x K_n} - \sum_{j=1}^n \mathbb{P}_n(j) \sum_{i=1}^{j-1}\frac{1}{i-x} = \frac{1}{x}-\frac{1}{x K_n} - \mathbb{P}_n g_x,
\end{align*}
where $g_x$ is the function on $\mathbb{N}$ defined by $g_x(j) = \sum_{i=1}^{j-1} (i-x)^{-1}$. We observe that $\hat{\Psi}_{n,0}(x)$ is an expectation with respect to the empirical measure $\mathbb{P}_n$, and hence, the asymptotic behavior of the random function $\hat{\Psi}_{n,0}$ is characterized by a suitable convergence of $\mathbb{P}_n$. Here, the convergence $S_{n,j}/K_n \to p_\alpha(j)$ (a.s.) by \Cref{lm:ep_asym} implies 
\begin{align*}
    \forall j \in \mathbb{N}, \quad  \mathbb{P}_n (j) \asconv \mathbb{P} (j) := p_\alpha(j) = \frac{\alpha \prod_{i=1}^{j-1}(i-\alpha)}{j!},
\end{align*}
i.e., the empirical measure $\mathbb{P}_n$ converges to the deterministic measure $\mathbb{P}$ pointwisely. Collectively, we expect that $\hat{\Psi}_{n,0}$ converges to the deterministic function $\Psi$ as follows:
\begin{align*}
   \forall x\in (0,1), \quad  \hat{\Psi}_{n,0}(x) = \frac{1}{x} - \frac{1}{xK_n} - \mathbb{P}_n g_x \underset{\text{in prob.}}{\overset{?}{\to}}\frac{1}{x} - \mathbb{P} g_x =:\Psi(x)
\end{align*}
Here, we emphasize that the above convergence does not follow directly from the pointwise convergence $\mathbb{P}_n(j) \to \mathbb{P}(j)$ because $g_x(j) = \sum_{i=1}^{j-1} (i-x)^{-1}$ is not a bounded function. 
To make the arguments more rigorous, we prove the convergence of $\mathbb{P}_n$ for a suitable set of function $\mathcal{F}$, i.e., 
$| \mathbb{P}_n f - \mathbb{P} f| \to^p 0$ for all $f\in \mathcal{F}$. 
Using this lemma, we show the convergence of $\hat{\Psi}_{n,0}\to \Psi$ and $\hat{\Psi}'_{n,0}\to \Psi'$ in a suitable sense. 
Furthermore, we will argue that $\Psi(\alpha)=0$ and $\Psi'(\alpha) = -I_\alpha < 0$ with $I_\alpha$ being the Fisher information of the Sibuya distribution. Combining all of this, we obtain the consistency of $\hat{\alpha}_n$.
For the asymptotic mixed normality, we use the Martingale CLT for the score function. See \Cref{tb:compare_i.i.d.} for rough comparisons with typical i.i.d. cases. 
\begin{table}[htbp]
    \centering
    \begin{tabular}{c|c|c}
   & $(X_i)_{i=1}^n \iid  \Pr(X;\alpha)$ & Ewens--Pitman partition\\\hline\hline
   Score function & i.i.d. sum & martingale \\\hline
   Empirical CDF  & $F_n(x) = n^{-1} \sum_{i = 1}^n \bm{1}_{\{X_i\leq x\}} $ & $F_n(j) = K_n^{-1} \sum_{j'\leq j} S_{n,j'}$\\\hline
   CDF & $F(x) = \Pr(X\leq x)$ & $F(j) = \sum_{j'\leq j} p_\alpha (j')$\\\hline
   Fisher Information & $nI_\alpha$ & $n^\alpha \E[\Mit] I_\alpha$ \\ \hline
   MLE & $\sqrt{n} (\hat{\alpha}_n-\alpha) \to \normal(0, I_\alpha^{-1})$ & $\sqrt{K_n}(\hat{\alpha}_n-\alpha) \to \normal(0, I_\alpha^{-1})$
    \end{tabular}
    \caption{Comparison with typical i.i.d. parametric models}
    \label{tb:compare_i.i.d.}
\end{table}

We have discussed the QMLE so far, where the unknown $\theta$ is fixed. Now, we consider the MLE $(\hat{\alpha}_n, \hat{\theta}_n)$ that simultaneously estimates $(\alpha, \theta)$. The main difficulty here is that $\hat{\theta}_n$ does not converge to a fixed value, which requires technical arguments. 
The first step is to reduce the dimension of the parameters that we have to consider; we define the function $\hat{y}_n : (0,1) \to (-1, \infty)$ by $\hat{y}_n(x) := \argmax_{y>-x} \ell_n(x, y)$ for all $x\in(0,1)$, 
where $\ell_n(x, y)$ is the log-likelihood \eqref{eq:df_loglikelihood} with $(\alpha,\theta)=(x,y)$. Here we claim that $\hat{y}_n$ is well-defined with a high probability. The gain of introducing $\hat{y}_n$ is that the MLE $(\hat{\alpha}_n, \hat{\theta}_n)$ can be rewritten as the solution of the following one-dimensional maximization problem:
\begin{align*}
  \hat{\alpha}_n \in \argmax_{x\in {(0,1)}} \ell_n(x, \hat{y}_n(x)), \qquad 
  \hat{\theta}_n = \hat{y}_n(\hat{\alpha}_n).
\end{align*}
Here, similarly to $\hat{\Psi}_{n, 0}$, we define the random function $\hat{\Psi}_n$ by $\hat{\Psi}_n(x) := K_n^{-1} \cdot\frac{\diff}{\diff x} \ell_n(x, \hat{y}_n(x))$ for all $x\in (0,1)$.
Then, for the deterministic function $\Psi(x) = x^{-1} - \mathbb{P} g_x$, we again show the convergence of $\hat{\Psi}_{n}\to \Psi$ and $\hat{\Psi}'_n\to \Psi'$ in a suitable sense.

{
\section{Discussion}\label{sec:discussion}
In this paper, we investigated the maximum likelihood estimator (MLE) for the Ewens--Pitman partition and derived its exact asymptotic distribution. 
Specifically, we established the asymptotic mixed normality of the MLE for $\alpha$ and proposed a confidence interval. Below, we outline several promising directions for future research.

\subsection{Quantitative martingale CLT}\label{subsec:clt_rate}
From \Cref{fig:coverage}, we observe that the actual coverage of confidence interval depends on $n$. To investigate the asymptotic behavior of coverage, we aim to identify a rate $c_n\to 0$ such that 
  $$
  \operatorname{d}_{\mathsf{Kol}}\Bigl(
\sqrt{K_n I_{\hat{\alpha}_n}}(\hat\alpha_n - \alpha), N
  \Bigr) \le c_n, \quad c_n \to 0,
  $$
  where $\operatorname{d}_{\mathsf{Kol}}(\cdot, \cdot)$ denotes the Kolmogorov distance, and $N\sim N(0,1)$. Based on the Taylor expansion arguments in our proof, deriving this rate requires establishing a convergence rate for the martingale central limit theorem as follows:
  $$
  \operatorname{d}_{\mathsf{Kol}} \Bigl(
\frac{\sum_{m=1}^n X_m}{\sqrt{\sum_{m=1}^n\E\left[X_{m}^2| \mathcal{F}_{m-1}\right]}}, N
  \Bigr) \le c_n, \quad c_n \to 0,
  $$
  where $X_m$ is a martingale difference sequence. In our case, $X_m$ represents the increments of the score function $\partial_\alpha \ell_n(\alpha,\theta)$. 
  
  If $X_m$ are iid random variables with finite third moment, Berry--Esseen theorem (cf. \cite[Chapter 3]{chen2010normal}) provides $c_n = O(n^{-1/2})$. 
  More generally, if the quadratic variation ${\sum_{m=1}^n\E[X_{m}^2| \mathcal{F}_{m-1}]}$ concentrates around the unconditional variance $\sum_{m=1}^n\E[X_{m}^2]$, then previous results on quantitative martingale CLT (see \cite{mourrat2013rate} and references therein) 
  can be used to estimate the rate $c_n$. 
  However, these results cannot be directly applied  to our setting due to the lack of concentration of 
the quadratic variation around the unconditional variance. Indeed, \Cref{lm:Psi_clt} and \Cref{prop:asym_info} give 
$$
{\sum_{m=1}^n\E[X_{m}^2| \mathcal{F}_{m-1}]} = n^\alpha I_\alpha \Mit  + o_p(n^\alpha), \quad \sum_{m=1}^n \E[X_m^2] = n^\alpha I_\alpha \E[\Mit]  + o(n^\alpha)$$
and $\Pr(\Mit \ne\E[\Mit])=1$ since $\Mit=\lim_{n\to+\infty} K_n/n^\alpha$ follows $\gmit$. 

   We can nevertheless hypothesize the convergence rate $c_n$ by empirically estimating the distance via the empirical measure (see \Cref{fig:clt_rate}). We conjecture from this figure that $c_n$ scales as $n^{-c}$ for some constant $c>0$. 
  \begin{figure}
    \centering
  \includegraphics[width=0.99\linewidth]{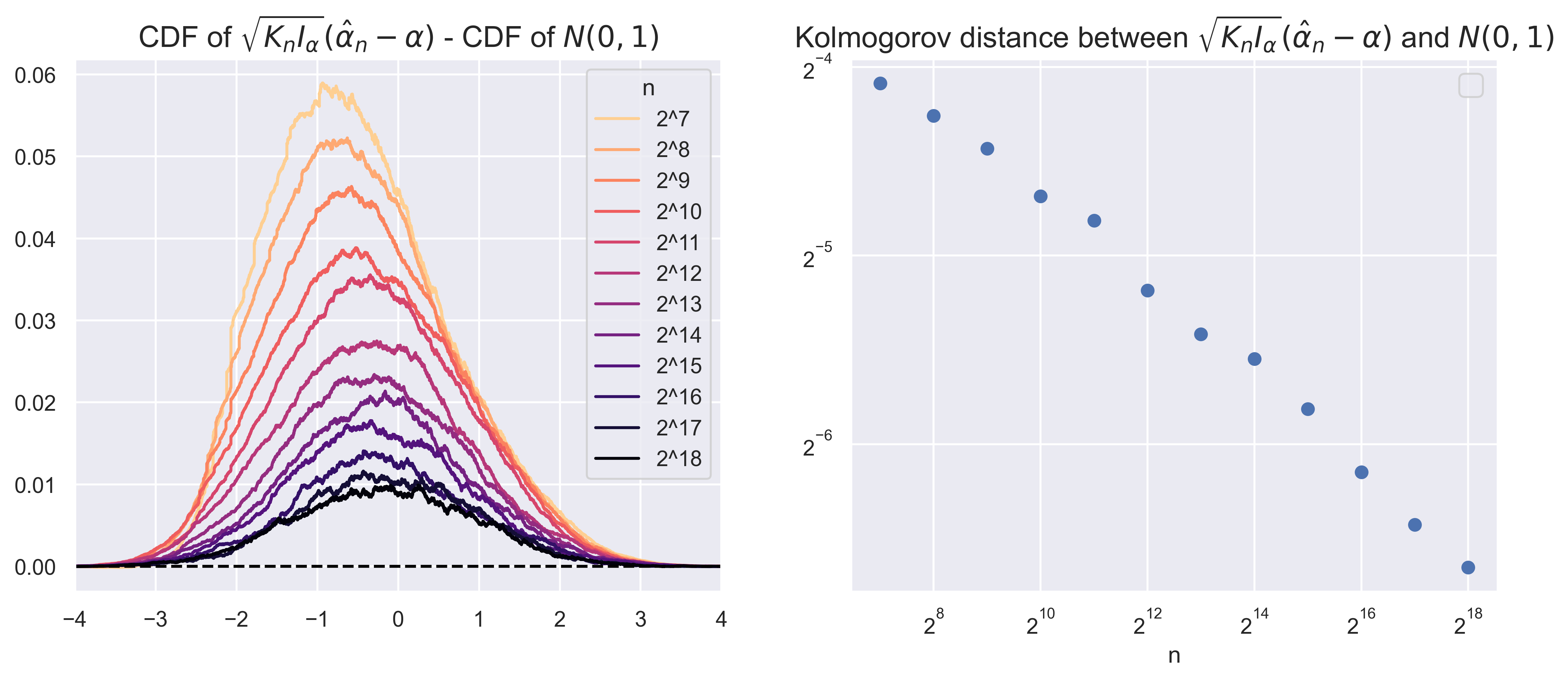}
    \caption{
      {The left figure illustrates the pointwise convergence of the empirical CDF of $\sqrt{K_n I_\alpha}(\hat\alpha_n-\alpha)$ to the CDF of $N(0,1)$. The right figure plots the Kolmogorov distance as $n$ increases. Simulation setting: $\alpha=0,8$, $\theta=0$, $10^5$ Monte Carlo simulations.}
    }
    \label{fig:clt_rate}
  \end{figure}
}

{
\subsection{Extension to the Gibbs partitions}\label{subsec:gibbs}
We are also interested in extending our results to  the Gibbs partition \cite{gnedin2006exchangeable}, which is a class of exchangeable random partitions characterized by the following likelihood:
\begin{align}\label{eq:density_gibbs}
  v_{n,K_n}(\alpha) \cdot \prod_{j=2}^n \left\{
    \prod_{i=1}^{j-1} (-\alpha + i)
  \right\}^{S_{n,j}}
\end{align}
where $\alpha$ is a parameter in $(-\infty, 1)$ and $v_{n,k}(\alpha)$ for $1\le k\le n$ is a non-negative sequence satisfying the backward recursion:
  \begin{align}\label{eq:recursion}
    v_{n, k}(\alpha) = (n-\alpha k) \cdot v_{n+1, k}(\alpha) + v_{n+1,k+1}(\alpha) \quad \text{with} \quad v_{1,1}(\alpha)=1.
  \end{align}
  Equivalently, \eqref{eq:density_gibbs} can be viewed as the marginal likelihood of a random partition generated sequentially, as follows:
  \begin{align*}
    \text{$(n+1)$th ball} \in \begin{cases}
      \text{an existing set $U_i$} & \text{w.p. $\frac{v_{n+1, K_n}(\alpha)}{v_{n,K_n}(\alpha)} \cdot (|U_i|-\alpha)$}, \quad \forall i=1, 2, \dots, K_n \\
      \text{a new set} & \text{w.p. $\frac{v_{n+1, K_n +1}(\alpha)}{v_{n, K_n}(\alpha)}$}
    \end{cases}
  \end{align*}
  The recursion in \eqref{eq:recursion} ensures that the probabilities above sum to $1$. Within this framework, the Ewens--Pitman partition with parameter $(\alpha,\theta)$ emerges as a special class of the Gibbs partitions, where $v_{n,k}$ is explicitly given by $ v_{n,k} (\alpha;\theta) := \frac{\prod_{i=0}^{k-1} (\theta+i\alpha)}{\prod_{i=0}^{n-1}(\theta + i)}$ with $\theta>-\alpha$ and $\alpha\in (0,1)$. 

  
  Importantly, the asymptotic properties of $(K_n, S_{n,j})$ in \Cref{lm:ep_asym} extend to the Gibbs partition; that is, if $\alpha \in (0,1)$, there exists a positive random variable $\mathsf{M}$, referred to as the \textit{$\alpha$-diversity}, such that
  \begin{align}\label{eq:gibbs_asymptotics}
    \frac{K_n}{n^\alpha} \to \mathsf{M} \quad \text{and} \quad \frac{S_{n,j}}{K_n} \to p_\alpha(j) = \frac{\alpha \prod_{i=1}^{j-1}(i-\alpha)}{j!} \quad \text{for each $j\in\mathbb{N}$}
  \end{align}
  almost surely (see \cite[Section 6.1]{pitman2003poisson} and \cite[Lemma 3.11]{pitman2006combinatorial}).

  Suppose that a random partition is generated by the Gibbs partition with an unknown parameter $\alpha\in (0,1)$ and $\{v_{n,k}(\alpha)\}$ satisfying the backward recursion \eqref{eq:recursion}. Let us denote the log-likelihood by $\ell_n^{\text{gibbs}}(\alpha)$:
  $$
  \ell_{n}^{\text{gibbs}}(\alpha) := \log \bigl(v_{n,K_n}(\alpha)\bigr) + \sum_{j=2}^n S_{n,j} \sum_{i=1}^{j-1} \log(i-\alpha).
  $$
  We aim to estimate $\alpha$ by the QMLE $\hat{\alpha}_{n,0} \in \argmax_{\alpha\in (0,1)} \ell_n(\alpha, 0)$ with 
  $$
  \ell_n(\alpha, 0) = (K_n-1) \log \alpha + \sum_{j=1}^n S_{n,j} \sum_{i=1}^{j-1}\log (i-\alpha)
  $$
  Based on the Taylor expansion argument in \Cref{Ap:qmle}, the asymptotic distribution of the QMLE $\hat{\alpha}_{n,0}$ is determined by the asymptotics of $\partial_\alpha \ell_n(\alpha,0)$ and $ \partial_\alpha^2 \ell_n(\alpha, 0)$. For $\partial_\alpha^2 \ell_n(\alpha, 0)$, applying \eqref{eq:gibbs_asymptotics}, we obtain the $\partial_\alpha^2 \ell_n(\alpha, 0)/K_n \to -I_\alpha$ as established in \Cref{lm:random_seq_conv}. 
For $\partial_\alpha \ell_n(\alpha, 0)$, we decompose it as
  \begin{align*}
    \partial_{\alpha} \ell_n(\alpha, 0) &= \partial_{\alpha} \ell_{n}^{\text{gibbs}}(\alpha)  + \partial_{\alpha} \ell_n(\alpha, 0) - \partial_{\alpha} \ell_{n}^{\text{gibbs}}(\alpha)\\
    &= \partial_{\alpha} \ell_{n}^{\text{gibbs}}(\alpha) + \Bigl(\frac{K_n-1}{\alpha} - \frac{v_{n,K_n}'(\alpha)}{v_{n, K_n}(\alpha)}\Bigr). 
  \end{align*}
  Here, $\partial_{\alpha} \ell_{n}^{\text{gibbs}}(\alpha) $ is the score function, which is martingale since the model is well-specified. Thus, its asymptotic distribution can be derived from the martingale CLT as in \Cref{lm:Psi_clt} under suitable conditions on $\{v_{n,k}\}$. 

  To establish the asymptotic distrubution of the QMLE, it remains to verify that $\Bigl(\frac{K_n-1}{\alpha} - \frac{v_{n,K_n}'(\alpha)}{v_{n, K_n}(\alpha)}\Bigr)$ is asymptotically negligible compared to the score function. This requires additional assumptions on $\{v_{n,k}\}$, and formalizing these conditions is left for future research.
}

\bibliographystyle{abbrvnat}
\bibliography{reference}

\section*{Acknowledgments}
The authors would like to thank Koji Tsukuda, Nobuaki Hoshino, Stefano Favaro, and Masaaki Shibuya for helpful comments on our research. 

\section*{Funding}
Takeru Matsuda was supported by JSPS KAKENHI Grant Numbers 19K20220, 21H05205, 22K17865 and JST Moonshot Grant Number JPMJMS2024.
Fumiyasu Komaki was supported by MEXT KAKENHI Grant Number 16H06533, JST CREST Grant Number JPMJCR1763, and AMED Grant Numbers JP21dm0207001 and JP21dm0307009.

\appendix
\section{Analysis of the Fisher Information}
\subsection{Proof of \Cref{prop:sibuya_fisher}}\label{Ap:sibuya_fisher}
Recall that $p_\alpha(j)$ is the probability mass function given by 
\begin{align*}
    \forall j \in \mathbb{N}, \quad  p_\alpha(j) := \frac{\alpha\prod_{i=1}^{j-1}(i-\alpha)}{j!}.
\end{align*}
Then, the Fisher Information $I_\alpha 
= - \sum_{j=1}^\infty p_\alpha(j) \partial_\alpha^2\log p_\alpha(j)
$ can be written as 
\begin{align*}
    I_\alpha = \sum_{j=1}^\infty p_\alpha(j) \left(
    \frac{1}{\alpha^2} + \sum_{i=1}^{j-1} \frac{1}{(i-\alpha)^2}
    \right)
    = \frac{1}{\alpha^2} + \sum_{j=1}^\infty p_\alpha(j) \sum_{i=1}^{j-1} \frac{1}{(i-\alpha)^2}. 
\end{align*}
This completes the proof of (A). Here, $\sum_{j=1}^\infty p_\alpha(j) \sum_{i=1}^{j-1} \frac{1}{(i-\alpha)^2}$ is finite since it is an expectation of a finite function of $j$. By Fubini's theorem, we have 
\begin{align*}
    \sum_{j=1}^\infty p_\alpha(j) \sum_{i=1}^{j-1} \frac{1}{(i-\alpha)^2} &= \sum_{i=1}^\infty \frac{1}{(i-\alpha)^2} \sum_{j=i+1}^\infty p_\alpha(j).
\end{align*}
Thus, for (B) it remains to show the following:
\begin{align}\label{eq:inductive}
   \forall i \in \mathbb{N}, \quad  \sum_{j=i+1}^\infty p_\alpha(j) = \frac{i-\alpha}{\alpha}p_\alpha(i)
\end{align}
We will prove it by induction. Notice that 
\eqref{eq:inductive} holds for $i=1$ since $\sum_{j=2}^\infty p_\alpha(j) = 1-p_\alpha(1) = 1-\alpha = (1-\alpha)p_\alpha(1)/\alpha$. Here we assume  \eqref{eq:inductive} for $i = k \in\mathbb{N}$. Then,
\begin{align*}
    \sum_{j=(k+1) + 1}^\infty p_\alpha(j)
    &= \sum_{j=k+1}^\infty p_\alpha(j) - p_\alpha(k+1)\\
    &= \frac{k-\alpha}{\alpha}p_\alpha(k) - p_\alpha(k+1)\\
    &= \frac{k-\alpha}{\alpha}\cdot \frac{k+1}{k-\alpha} p_\alpha(k+1) - p_\alpha(k+1)\\
    &= \frac{k+1-\alpha}{\alpha} p_\alpha(k+1),
\end{align*}
so \eqref{eq:inductive} holds for $i=k+1$. Therefore, \eqref{eq:inductive} holds for all $i\in \mathbb{N}$.

Finally, let us show the continuity of $I_\alpha$. We write $I_\alpha$ using the formula $(B)$:
\begin{align*}
    I_\alpha = \frac{1}{\alpha^2} + \sum_{j=1}^\infty \frac{p_\alpha(j)}{\alpha(i-\alpha)}.
\end{align*}
Now we claim that $I_\alpha$ converges uniformly on $K$ for any closed subset $K=[s,t] \in (0,1)$, which clearly concludes the proof of continuity.
Since $p_\alpha(j)/\alpha = (\prod_{i=1}^{j-1}(i-\alpha))/j!$ is
non-increasing on $[s,t]$, we observe that
\begin{align*}
\forall j\in \mathbb{N},  \ 
\sup_{\alpha \in K}\left|p_\alpha(j) \frac{1}{\alpha(j-\alpha)}\right| 
&\leq
     \sup_{\alpha \in K}\frac{p_\alpha(j)}{\alpha} \cdot 
    \sup_{\alpha \in K}\frac{1}{j-\alpha}
    = \frac{p_s(j)}{s} \frac{1}{j-t} \leq \frac{p_s(j)}{s(1-t)},
\end{align*}
and $\sum_{j=1}^\infty {p_s(j)}/(s(1-t)) = {1}/({s(1-t)}) < + \infty.$
Then, the Weierstrass M-test implies that $I_\alpha$ converges uniformly on $K$, which concludes the proof.
\subsection{Proof of \Cref{lm:f_alpha}}\label{Ap:f_alpha}
Recall that $f_\alpha: (-1, \infty) \rightarrow \mathbb{R}$ is defined by
\begin{align*}
    \forall z \in (-1,\infty), \quad f_\alpha(z) := \psi(1+z) - \alpha\psi(1+\alpha z),
\end{align*}
where $\psi(z) = \Gamma'(z)/\Gamma(z)$ is the digamma function. Then $\lim_{z\rightarrow 0+} \psi(z) = -\infty$ implies
\begin{align*}
\lim_{z\rightarrow-1+}f_\alpha(z) = \lim_{z\rightarrow-1+} \psi(1+z) - \alpha\psi(1-\alpha)
= -\infty.
\end{align*}
In contrast, by $\psi(z) = \log z + o(1)$ for large $z>0$ (cf. \cite[Section 6]{zwillinger2018crc}), we obtain $f_\alpha(z)=(1-\alpha)\log z + O(1)$ as $z\to+\infty$ and hence 
 \begin{align*}
  \lim_{z\to\infty} f_\alpha (z) = +\infty. 
 \end{align*}
By $\psi^{(1)}(1+z) = \sum_{i=1}^\infty (i + z)^{-2}$ for all $z> -1$, we have
\begin{align*}
    f'_\alpha(z) &= \psi^{(1)} (1+z) - \alpha^2\psi^{(1)}(1+\alpha z)\\
    &
    = \sum_{i=1}^\infty \left\{(i+z)^{-2} - \alpha ^2 (i+z\alpha )^{-2}\right\}\\
    &= \sum_{i=1}^\infty\left\{ (i+z)^{-2} - (i/\alpha  + z)^{-2}\right\}\\
    &>0 && \text{by $\alpha\in(0,1)$},
\end{align*}
which implies that $f_\alpha$ is strictly increasing. Putting all together, we conclude that $f_\alpha(z)$ is bijective from $(-1, \infty)$ to $\mathbb{R}$. 

It remains to show $f_\alpha''(z) < 0$. This follows from the same argument above:
\begin{align*}
    f''_\alpha(z) &= \psi^{(2)} (1+z) - \alpha^3 \psi^{(2)}(1+\alpha z)\\
    & = -2 \sum_{i=1}^{\infty}\left\{
    (i + z)^{-3} - \alpha^3 (i+\alpha z)^{-3} 
    \right\} && \psi^{(2)}(1+z) = -2\sum_{i=1}^\infty (i+z)^{-3} \\
    &= -2 \sum_{i=1}^{\infty}\left\{
    (i + z)^{-3} - (i/\alpha+ z)^{-3} 
    \right\} \\
    &<0 && \text{by $\alpha\in (0,1)$}.
\end{align*}
\subsection{Proof of \Cref{prop:asym_info}}\label{Ap:asym_info}
Recall that Ewens--Pitman partition has the following log-likelihood:
\begin{align*}
   \ell_n(\alpha, \theta) = \sum_{i=1}^{K_n-1}\log (\theta + i\alpha) - \sum_{i=1}^{n-1}\log (\theta+i) + \sum_{j=2}^n S_{n,j} \sum_{i=1}^{j-1}\log (i-\alpha).
\end{align*}
First, we derive the convergence of $\partial_\theta^2 \ell_n(\alpha, \theta)$. This is written as 
\begin{equation*}
  \partial_\theta^2 \ell_n(\alpha, \theta) = -\sum_{i=1}^{K_n-1}\frac{1}{(\theta+i\alpha)^2} + \sum_{i=1}^{n-1} \frac{1}{(\theta + i)^2}.
\end{equation*}  
Note that $\sum_{i=1}^{K_n-1}(\theta+i\alpha)^{-2}$ is a strictly increasing function of $K_n$ and $K_n$ is a nondecreasing random variable. Then the monotone convergence theorem implies that 
\begin{equation}
  \label{eq:exp}
  \E\left[\sum_{i=1}^{K_n-1}\frac{1}{(\theta+i\alpha)^2}\right] \rightarrow \sum_{i=1}^\infty \frac{1}{(\theta+i\alpha)^2} < + \infty. 
\end{equation}
By the above displays and $\psi^{(1)}(1+z) = \sum_{i=1}^\infty (i + z)^{-2}$, the asymptotics of $I_{{\theta,\theta}} ^{(n)} = \E[-\partial_\theta^2 \ell_n(\alpha, \theta)]$ is given by 
\begin{align*}
  I_{{\theta,\theta}} ^{(n)}
  &= \E\left[
  \sum_{i=1}^{K_n-1}\frac{1}{(\theta+i\alpha)^2}
  \right] - \sum_{i=1}^{n-1} \frac{1}{(\theta + i)^2}\\
  &\rightarrow \sum_{i=1}^\infty \frac{1}{(\theta + i\alpha)^2} - \sum_{i=1}^{\infty} \frac{1}{(\theta+i)^2}\\
  &= \frac{1}{\alpha^2} \left(\sum_{i=1}^\infty \frac{1}{(i + \theta/\alpha)^2}
  - \sum_{i=1}^\infty \frac{\alpha^2}{(i + \theta)^2}
  \right) \\
  &= \frac{\psi^{(1)}(1+\theta/\alpha) - \alpha^2 \psi^{(1)}(1+\alpha\cdot\theta/\alpha )}{\alpha^2}\\
  &= \alpha^{-2} f'_\alpha(\theta/\alpha). 
\end{align*}

Next, we derive the leading term for $I_{{\alpha, \theta}}^{(n)}$. 
 Notice that $\E \left[\partial_\theta \ell_n(\alpha, \theta) \right] = 0$ holds since $\ell_n$ is the log-likelihood. Now from this equation we obtain
\begin{equation}
\E\left[\sum_{i=1}^{K_n-1} \frac{1}{\theta + i\alpha}\right] - \sum_{i=1}^{n-1} \frac{1}{\theta + i} = 0.
  \label{eq:exp_1}
\end{equation}
In contrast, $\partial_\theta\partial_\alpha \ell_n(\alpha, \theta)$ is written by
\begin{align*}
  \partial_\theta\partial_\alpha \ell_n(\alpha, \theta)  =  -\sum_{i=1}^{K_n-1}\frac{i}{(\theta +i\alpha)^2}
  = -\frac{1}{\alpha}\sum_{i=1}^{K_n-1} \frac{1}{\theta + i\alpha} + \frac{\theta}{\alpha} \sum_{i=1}^{K_n-1}\frac{1}{(\theta + i\alpha)^2}.
\end{align*}
Thus, \eqref{eq:exp} and \eqref{eq:exp_1} result in
\begin{align*}
  I_{\theta\alpha}^{(n)} &=-\E[\partial_\theta\partial_\alpha \ell_n(\alpha, \theta)]
  \\
  &=\frac{1}{\alpha}\E\left[\sum_{i=1}^{K_n-1} \frac{1}{\theta + i\alpha}\right] - \frac{\theta}{\alpha}\E\left[\sum_{i=1}^{K_n-1}\frac{1}{(\theta+i\alpha)^2}\right]\\
  &= \frac{1}{\alpha}\sum_{i=1}^{n-1} \frac{1}{\theta + i} +O(1)\\
  &= \alpha^{-1}\log n + O(1),
\end{align*}
which completes the proof for $I_{\theta\alpha}^{(n)}$.

Finally, we derive the leading term for $I_{{\alpha,\alpha}}^{(n)}$. Note
\begin{equation*}
  \partial_\alpha^2 \ell_n(\alpha, \theta) = -\sum_{i=1}^{K_n-1}\frac{i^2}{(\theta + i\alpha)^2} - \sum_{j=2}^{n}S_{n, j}\sum_{i=1}^{j-1}\frac{1}{(i-\alpha)^2},
\end{equation*}
where
\begin{align*}
\frac{i^2}{(\theta + i\alpha)^2} = \frac{1}{\alpha^2} -\frac{2\theta}{\alpha^2} \frac{1}{\theta + i\alpha} + \frac{\theta^2}{\alpha^2(\theta+i\alpha)^2}.    
\end{align*}
Then, putting together the above displays and \eqref{eq:exp} and \eqref{eq:exp_1}, one obtains
\begin{align*}
  I^{(n)}_{{\alpha,\alpha}} &= - \E[\partial_\alpha^2 \ell_n(\alpha,\theta)] \\
  &= 
  \E\left[
  \sum_{i=1}^{K_n-1}\frac{i^2}{(\theta + i\alpha)^2}
  \right]
  + \sum_{j=2}^{n}\E[S_{n, j}]\sum_{i=1}^{j-1}\frac{1}{(i-\alpha)^2}\\
  &= \frac{\E[K_n]-1}{\alpha^2}- \frac{2\theta}{\alpha^2} \E\left[\sum_{i=1}^{K_n-1}  \frac{1}{\theta + i\alpha}\right]
  + \frac{\theta^2}{\alpha^2} \E\left[\sum_{i=1}^{K_n-1}\frac{1}{(\theta + i\alpha)^2}\right] + \sum_{j=2}^{n}\E[S_{n, j}]\sum_{i=1}^{j-1}\frac{1}{(i-\alpha)^2}\\
  &= \frac{1}{\alpha^2}\E[K_n] + \sum_{j=2}^{n}\E[S_{n, j}]\sum_{i=1}^{j-1}\frac{1}{(i-\alpha)^2} - 
  \frac{2\theta}{\alpha^2}\sum_{i=1}^{n-1} \frac{1}{\theta + i} + \frac{\theta^2}{\alpha^2} \E\left[\sum_{i=1}^{K_n-1}\frac{1}{(\theta + i\alpha)^2}\right] - \frac{1}{\alpha^2} \\
  &= \frac{\E[K_n]}{\alpha^2} + \sum_{j=2}^{n}\E[S_{n, j}]\sum_{i=1}^{j-1}\frac{1}{(i-\alpha)^2} + O(\log n).
\end{align*}
Here we take $X_n:= \alpha^{-2} K_n+ \sum_{j=2}^{n} S_{n,j} \sum_{i=1}^{j-1}{(i-\alpha)^{-2}}$. Then, one obtains
\begin{align*}
    I_{{\alpha,\alpha}}^{(n)} = \E[X_n] + O(\log n),
\end{align*}
Thus, it remains to show $\E[X_n]/n^\alpha \to \E[\Mit] I_\alpha$. We will show it by the dominated convergence theorem. Note that $X_n$ is bounded by $K_n$ up to a constant:
\begin{align*}
    0 \leq X_n &= \sum_{j=1}^{n} S_{n,j} \left(
    \frac{1}{\alpha^2} + \sum_{i=1}^{j-1}\frac{1}{(i-\alpha)^2}
    \right) \\
    &= 
    \sum_{j=1}^n S_{n,j} \sum_{i=0}^{j-1}  \frac{1}{(i-\alpha)^2}
    \notag\\
    &\leq
    \sum_{j=1}^{n}S_{n,j} \cdot  \sum_{i=0}^{\infty}\frac{1}{(i-\alpha)^2}\\
    &= K_n C_\alpha && \sum_{j}S_{n,j} = K_n
\end{align*}
where $C_\alpha := \sum_{i=0}^{\infty}\frac{1}{(i-\alpha)^2} <+\infty$. 
In contrast, \Cref{cor:finite_conv} with $g(j) = \sum_{i=0}^{j-1} (i-\alpha)^{-2}$ and $K_n/n^\alpha \rightarrow \Mit$ (a.s.) yield
\begin{align*}
    \frac{X_n}{n^\alpha} &= \frac{K_n}{n^\alpha} \sum_{j=1}^\infty\frac{S_{n,j}}{K_n} \sum_{i=0}^{j-1}\frac{1}{(i-\alpha)^2} \\
    &\rightarrow \Mit \sum_{j=1}^\infty p_\alpha(j)
    \sum_{i=0}^{j-1}\frac{1}{(i-\alpha)^2} && \text{$\frac{K_n}{n^\alpha}\to \Mit$ and \Cref{cor:finite_conv}}\\
    &= \Mit\left(\frac{1}{\alpha^2} + \sum_{j=2}^\infty p_\alpha(j) \sum_{i=1}^{j-1}\frac{1}{(i-\alpha)^2}\right) \\
    &= \Mit I_\alpha \ \text{(a.s.)} && \text{by (A) of \Cref{prop:sibuya_fisher}}. 
\end{align*}
Putting together the above displays and $K_n/n^\alpha \rightarrow \Mit$ in mean by (A) of \Cref{lm:ep_asym}, the dominated convergence implies that
\begin{align*}
    n^{-\alpha}\E[X_n] = \E[n^{-\alpha} X_n] \rightarrow \E[\Mit I_\alpha] = \E[\Mit] I_\alpha.
\end{align*}
This concludes the proof.

\section{Convergence of empirical measure}\label{sec:proof_empirical_measure}
We fix $\alpha\in(0,1)$ and $\theta > -\alpha$. We define the random measure and deterministic measure by  
\begin{align*}
    \forall j \in \mathbb{N}, \quad  \mathbb{P}_n(j) := \frac{S_{n,j}}{K_n}, \ \ \ \mathbb{P}(j):=p_\alpha(j) = \frac{\alpha\prod_{i=1}^{j-1} (i-\alpha)}{j!}.
\end{align*}
For any function $f$ on $\mathbb{N}$, 
we write $\mathbb{P}_n f = \sum_{j=1}^\infty \mathbb{P}_n(j) f(j)$ and $\mathbb{P} f = \sum_{j=1}^\infty \mathbb{P}(j) f(j)$. This section aims to show a suitable convergence of $\mathbb{P}_n$ to $\mathbb{P}$.
\begin{lemma}\label{lm:abs_conv}
$\sum_{j=1}^\infty|\mathbb{P}_n(j)-\mathbb{P}(j)| \to 0$ (a.s.). 
\end{lemma}
\begin{proof}
By the point-wise convergence $S_{n,j}/K_n \to p_\alpha(j)$ (a.s.) from (B) of \Cref{lm:ep_asym} and the fact that $p_\alpha(j)$ is a density; that is, $\sum_{j=1}^\infty p_\alpha(j) = 1$, Scheffé's lemma \cite{scheffe1947useful} implies $\sum_{j=1}^\infty|\mathbb{P}_n(j)-\mathbb{P}(j)| \to 0$ (a.s.). 
\end{proof}
This lemma gives the following corollary. 
\begin{corollary}\label{cor:finite_conv}
$\mathbb{P}_n f \to \mathbb{P} f$ a.s. for any bounded function $f$ on $\mathbb{N}$. 
\end{corollary}
For each $x\in[0,1)$, define the function $g_x$ on $\mathbb{N}$ by
\begin{align}\label{eq:df_logj_set}
\forall x\in[0,1), \quad  \forall j \in \mathbb{N}, \quad  g_x (j) := \sum_{i=1}^{j-1} \frac{1}{i-x}
\end{align}
We claim that the leading term of the score function $\partial_\alpha \ell_n(\alpha,\theta)$ can be written as the expectation of $g_x$ with respect to $\mathbb{P}_n$. Notice that the likelihood formula 
\eqref{eq:df_loglikelihood} implies 
$$
\partial_\alpha \ell_n(\alpha, \theta) = 
\sum_{i=1}^{K_n-1} \frac{i}{\theta + i\alpha} - \sum_{j=1}^n S_{n,j} \sum_{i=1}^{j-1}\frac{1}{i-\alpha},
$$
Then, $K_n^{-1}\partial_\alpha \ell_n (\alpha,\theta)$ can be written as 
\begin{align}
    \frac{\partial_\alpha \ell_n (\alpha,\theta)}{K_n}
     &= \frac{1}{K_n}\sum_{i=1}^{K_n-1} \frac{i}{\theta + i\alpha} - \sum_{j=1}^n \frac{S_{n,j}}{K_n} \sum_{i=1}^{j-1}\frac{1}{i-\alpha} \nonumber\\
    &= \frac{K_n-1}{\alpha K_n} - \frac{\theta}{\alpha K_n} \sum_{i=1}^{K_n-1} \frac{1}{\theta + i\alpha} - \sum_{j=1}^n \frac{S_{n,j}}{K_n} \sum_{i=1}^{j-1}\frac{1}{i-\alpha} \nonumber \\
    &= \frac{1}{\alpha} - \mathbb{P}_n g_\alpha 
    -\frac{1}{\alpha K_n} - \frac{\theta}{\alpha K_n} \sum_{i=1}^{K_n-1} \frac{1}{\theta + i\alpha}\label{eq:case_1}.
\end{align}
Notice that the third and fourth terms are negligible since $K_n\to+\infty$. 
The next lemma shows the convergence of $\mathbb{P}_n g_x$ for each $x\in [0,1)$. 
\begin{lemma}\label{lm:logj_conv}
$\mathbb{P}_n g_x \to^p \mathbb{P} g_x <+\infty$ for all $x\in[0,1)$. 
\end{lemma}

Note that \Cref{lm:logj_conv} does not directly follow from \Cref{cor:finite_conv} as $g_x$ is not bounded.
For the proof, we define the deterministic function $\Psi: (0,1) \rightarrow\mathbb{R}$ as
\begin{align} \label{eq:df_Psi}
  \forall x\in (0,1), \quad  \Psi(x) := 
  \frac{1}{x} - \mathbb{P} g_x = \frac{1}{x} - \sum_{j=1}^{\infty} p_{\alpha}(j)\sum_{i=1}^{j-1}\frac{1}{i-x}.
\end{align}
\Cref{lm:Psi_property} below shows some properties of $\Psi$.
\begin{lemma}\label{lm:Psi_property} 
Let $I_\alpha$ be the Fisher information defined by \eqref{eq:sibuya_fisher}. Then, we have 
  \begin{enumerate}
    \item[(A)] $\Psi$ is of class $C^{1}$ on $(0,1)$.
     \item[(B)] $\Psi'(x) = -x^{-2}- \sum_{j=2}^{\infty} p_\alpha(j) \sum_{i=1}^{j-1} (i-x)^{-2} < 0$ and $\Psi'(\alpha) = -I_\alpha$.
     \item[(C)] $\Psi(\alpha) = 0$, i.e., $\mathbb{P} g_\alpha = \alpha^{-1}$
  \end{enumerate}
\end{lemma}
\begin{proof}
  See \Cref{Ap:Psi_property}. 
\end{proof}

\subsection{Proof of \Cref{lm:logj_conv}}
$\mathbb{P} g_x <+\infty$, that is, the finiteness of $\sum_{j=1}^\infty p_\alpha(j) \sum_{i=1}^{j-1} (i-x)^{-1}$, is shown at the beginning of \Cref{Ap:Psi_property}, so we assume it.
Suppose the claim $\mathbb{P}_n g_x \to^p \mathbb{P} g_x$ holds for $x=\alpha$. Then, the triangle inequality implies
 \begin{align*}
     \forall x\in[0,1), \quad |\mathbb{P}_n g_x - \mathbb{P} g_x| \leq |\mathbb{P}_n g_\alpha - \mathbb{P} g_\alpha| + |\mathbb{P}_n (g_x - g_\alpha) - \mathbb{P} (g_x - g_\alpha)|,
 \end{align*}
 where the first term is $o_p(1)$ by the assumption. For the second term, we see that $g_x-g_\alpha$ is a bounded function.
\begin{align*}
    \forall j \in \mathbb{N}, \ |g_x (j) - g_\alpha (j)| = \left|\sum_{j=1}^{j-1} \left(\frac{1}{i-x} - \frac{1}{i-\alpha}\right)\right| \leq \sum_{j=1}^\infty \frac{|x-\alpha|}{(i-x)(i-\alpha)} < +\infty .
\end{align*}
Then, \Cref{cor:finite_conv} applied with $f=g_x-g_\alpha$ implies $|\mathbb{P}_n (g_x - g_\alpha) - \mathbb{P} (g_x - g_\alpha)|= o_p(1)$. Thus, it suffices to show $\mathbb{P}_n g_x \to^p \mathbb{P} g_x$ for $x=\alpha$.
Now \eqref{eq:case_1} and $\alpha^{-1} = \mathbb{P} g_\alpha$ by (C)-\Cref{lm:Psi_property} imply
\begin{align*}
    \mathbb{P}_n g_\alpha - \mathbb{P}g_\alpha &= \mathbb{P}_n g_\alpha - \frac{1}{\alpha}= \underbrace{ - \frac{1}{\alpha K_n} -\frac{\theta}{\alpha K_n}\sum_{i=1}^{K_n-1} \frac{1}{\theta + i\alpha}}_{=o_p(1)} - \frac{\partial_\alpha \ell_n (\alpha,\theta)}{K_n}
\end{align*}
Now, for any $\epsilon>0$, 
\begin{align*}
 \Pr\left(\left|n^{-\alpha}\cdot \partial_\alpha \ell_n(\alpha, \theta) \right| > \epsilon\right)
  &\le n^{-2\alpha}\epsilon^{-2}{I_{{\alpha,\alpha}}^{(n)}}  && \text{by Chebyshev's inequality} \\
  &{=} n^{-2\alpha} \epsilon^{-2} O(n^{\alpha}) && \text{by $I_{{\alpha,\alpha}}^{(n)} = 
  O(n^\alpha)$ from \Cref{prop:asym_info}}\\
  &= O(n^{-\alpha}). 
\end{align*}
This means $ \partial_\alpha \ell_n(\alpha, \theta) = o_p(n^\alpha)$. Combined with $K_n/n^\alpha \to \Mit > 0$ (a.s.), we obtain $K_n^{-1} \partial_\alpha\ell_n(\alpha, \theta) = o_p(1)$ and concludes the proof. 

\subsection{Proof of \Cref{lm:Psi_property}}\label{Ap:Psi_property}
Recall that $\Psi(x)$ is defined as $\Psi(x) := x^{-1} - \mathbb{P} g_x$ where 
$$
\mathbb{P} g_x = \sum_{j=1}^\infty p_\alpha (j) g_x(j), \quad g_x(j) = \sum_{i=1}^{j-1} (i-x)^{-1}, \quad p_\alpha(j) = \frac{\alpha\prod_{i=1}^{j-1}(i-\alpha)}{j!}. 
$$
Now we claim that $\mathbb{P} g_x$ is finite for all $x\in [0,1)$. Recall the density $p_\alpha(j)$ decays $p_{\alpha}(j) = O (j^{-(1+\alpha)})$ by Stirling's formula, whereas $\sum_{i=1}^{j-1}(i-x)^{-1} \sim \log{j}$ as $j\to\infty$. By the finiteness of the integral $\int_{1}^\infty \log{x} \cdot x^{-(1+\alpha)} \diff x = \alpha^{-2}$, we know that $\mathbb{P} g_x$ is finite for all $x\in [0,1)$. In particular, $\Psi(x)=x^{-1}-\mathbb{P} g_x$ is also finite for all $x\in(0,1)$. 

  First, we prove (A)-\Cref{lm:Psi_property}; that is, $\Psi$ is of class 
  $C^1$ on $(0,1)$.
     For any closed subset $K = [s, t] \subset(0,1)$, observe
     \begin{align*}
        \forall j \in \mathbb{N}, \  \sup_{x\in K} \left|\frac{\diff}{\diff x} g_x(j) \right|=
        \sup_{x\in K}\sum_{i=1}^{j-1} \frac{1}{(i-x)^2} 
        < \sum_{i=1}^\infty \frac{1}{(i-t)^2} := C.
     \end{align*}
     Since $\Psi(x):= x^{-1} - \mathbb{P} g_x$ takes a form of  expectation of $g_x$ with respect to $\mathbb{P}$,  
     Weierstrass's M-test implies that $\Psi'(x)$ converges uniformly on $K$. As $K$ is arbitrary, it means that 
      $\Psi'(x)$ converges compactly on $(0,1)$. Thus, we conclude that $\Psi$ is class $C^{1}$ on $(0,1)$.

      Next, we show (B)-\Cref{lm:Psi_property}.  The compact convergence allows us to change the summation and differentiation:
      $$
      \Psi'(x) = \frac{\diff }{\diff x} \left(\frac{1}{x} - \mathbb{P} g_x\right) = -\frac{1}{x^2} - \mathbb{P} \frac{\diff}{\diff x} g_x = -\frac{1}{x^2} - \sum_{j=1}^\infty p_\alpha (j) \sum_{i=1}^{j-1} \frac{1}{(i-x)^2}.
      $$
      By (A)-\Cref{prop:sibuya_fisher}, we know that the RHS is $-I_\alpha$ and $I_\alpha>0$, thereby completing the proof of (B)-\Cref{lm:Psi_property}. 

    Finally, we show (C)-\Cref{lm:Psi_property}. Taking the derivative of $\sum_{j=1}^\infty p_\alpha(j) = 1$ with respect to $\alpha$ gives $\frac{\diff}{\diff \alpha} \sum_{n=1}^\infty p_{\alpha}(j) = 0$.
    Suppose that we can interchange the differential and summation; that is, $\frac{\diff}{\diff \alpha} \sum_{n=1}^\infty p_{\alpha}(j) = \sum_{n=1}^\infty \frac{\diff}{\diff \alpha} p_{\alpha}(j)$. Then,
    the definition $p_\alpha(j) := \alpha \prod_{i=1}^{j-1} (i-\alpha)/(j!)$ implies that
  \begin{align*}
     0 = \sum_{j=1}^{\infty}\frac{\diff}{\diff \alpha} p_{\alpha}(j) = \sum_{j=1}^\infty p_{\alpha}(j)  \left(\frac{1}{\alpha} - \sum_{i=1}^{j-1} \frac{1}{i-\alpha}\right) = \frac{1}{\alpha} - \sum_{j=2}^{\infty} p_\alpha (j) \sum_{i=1}^{j-1} \frac{1}{i-\alpha} = \Psi(\alpha), 
  \end{align*}
  which concludes the proof of $(C)$. Thus, it suffices to justify $\frac{\diff}{\diff \alpha} \sum_{n=1}^\infty p_{\alpha}(j) = \sum_{n=1}^\infty \frac{\diff}{\diff \alpha} p_{\alpha}(j)$. We prove this by the compact convergence of $\sum_{j=1}^\infty \frac{\diff }{\diff \alpha} p_{\alpha}(j)$ on $(0, 1)$, i.e.
 the uniform convergence of $\sum_{j=1}^\infty \frac{d}{d\alpha} p_{\alpha}(j)$ on $K$
  for any closed interval $K = [s,t] \subset(0, 1)$.
  Note that $\alpha^{-1}{p_\alpha(j)}= \prod_{i=1}^{j-1} (i-\alpha)/{j!}$ is non-increasing function on $(0,1)$ for all $j \in \mathbb{N}$. 
  Then, $\sup_{\alpha\in[s,t]} \left| \frac{\diff}{\diff \alpha} p_{\alpha}(j)\right|$ is upper bounded as 
  \begin{align*}
    \sup_{\alpha\in[s,t]} \left| \frac{\diff}{\diff \alpha} p_{\alpha}(j)\right| &= 
    \sup_{\alpha\in[s,t]} \left| 
    p_{\alpha}(j)  \left(\frac{1}{\alpha} - \sum_{i=1}^{j-1} \frac{1}{i-\alpha}\right)
    \right|\\
    &\leq \sup_{\alpha\in[s,t]}\left(
      \frac{p_\alpha(j)}{\alpha}
      \right) \times \left(
        1 + \sup_{\alpha\in[s,t]} \sum_{i=1}^{j-1}\frac{\alpha}{i-\alpha}
        \right)\\
    &= \frac{p_s(j)}{s} \left(1+ \sum_{i=1}^{j-1}\frac{t}{i-t}\right) =: M_j^{s,t},
  \end{align*}
and we observe
  \begin{align*}
      \sum_{j=1}^\infty M_j^{s,t} =\sum_{j=1}^\infty \frac{p_s(j)}{s} \left(1+ t\sum_{i=1}^{j-1}\frac{1}{i-t}\right)
    = \frac{1}{s} + \frac{t}{s}\sum_{j=1}^\infty p_s(j) \sum_{i=1}^{j-1}\frac{1}{i-t}< +\infty.
  \end{align*}
  Here the right-hand side converges based on the same argument at the beginning of this proof. Thus, Weierstrass's M-test implies that $\sum_{j=1}^\infty \frac{\diff }{\diff \alpha} p_{\alpha}(j)$ converges uniformly on $[s,t
  ]$.

\section{Proofs for QMLE}\label{Ap:qmle}
In this section, we prove the asymptotic properties of the QMLE.
\subsection{Proof of \Cref{prop:qmle_unique}}
Here, we prove the existence and uniqueness of QMLE. Let us fix $\plug\in (-\alpha, +\infty)$ and assume $n\geq 2$. By the likelihood formula \eqref{eq:df_loglikelihood}, 
\begin{align*}
  \partial_x^2 \ell_n(x, \plug) = -\sum_{i=1}^{K_n-1}\frac{i^2}{(\plug + ix)^2} - \sum_{j=2}^n {S_{n,j}} \sum_{i=1}^{j-1} \frac{1}{(i-x)^2} < 0
\end{align*}
for all $x\in ((-\plug)\vee 0, 1)$. Here, $\partial_x^2 \ell_n(x, \plug)$ is strictly negative with probability $1$. Indeed, the first term is negative if $K_n > 1$. Otherwise, $S_{n,n}=1$ so $ K_n=1$ and hence the second term is $-\sum_{j=1}^{n-1}(i-x)^{-2}< 0$. Thus, QMLE $\hat{\alpha}_{n,\plug}$ exists and equals the unique solution of $\partial_x \ell_n(\cdot, \plug) = 0$ if and only if $\lim_{x\rightarrow ((-\plug)\vee 0)+} \partial_x \ell_n(x, \plug) > 0 > \lim_{x\rightarrow 1-} \partial_x \ell_n(x, \plug)$. 
The necessary and sufficient condition is given by \cite[Lemma 5.1]{carlton1999applications}
\begin{align*}
  1<K_n<n, \quad \text{and}\quad \plug < \Theta_n := \frac{K_n (K_n-1)}{\sum_{j=2}^{n}2S_{n,j}\sum_{i=1}^{j-1} i^{-1}} = \frac{K_n-1}{2 \mathbb{P}_n g_0},
\end{align*}
where we have used the definition $\mathbb{P}_n g_0=\sum_{j=1}^n \frac{S_{n,j}}{K_n} g_0(j)$ with $g_0(j) = \sum_{i=1}^{j-1} i^{-1}$. We observe that the first condition is satisfied with a high probability since
$K_n/n^\alpha \rightarrow \Mit > 0$ (a.s.) holds with $\alpha\in(0,1)$. For the second condition, using \Cref{lm:logj_conv} applied with $x=0$, we have
\begin{align*}
    \frac{\Theta_n}{K_n-1} = 
    \frac{1}{2\mathbb{P}_n g_0} {\to^p} \frac{1}{2\mathbb{P} g_0} <+\infty.
\end{align*}
Combined with $K_n \to + \infty$ (a.s.), we get $\Pr(\plug\geq\Theta_n) = o(1)$, and hence the second condition $\plug<\Theta_n$ is also satisfied with a high probability.

\subsection{Proof of \eqref{eq:qmle_mle_error} in \Cref{prop:error}}
Here, we derive the asymptotic error between the QMLE $\hat{\alpha}_{n,\plug}$ and the MLE $\hat{\alpha}_{n,\theta} = \hat{\alpha}_{n,\plug=\theta}$ where $\theta$ is well specified.
For each $\plug \in (-\alpha, \infty)$, we define the random function $\hat{\Psi}_{n, \plug}$ by
\begin{align*}
    \forall x \in ((-\plug) \vee 0, 1), \quad \hat{\Psi}_{n,\plug} (x) := K_n^{-1} \cdot \partial_x \ell_n(x, \plug).
\end{align*}
From \eqref{eq:case_1} with $\theta=\plug$, we can write $\hat{\Psi}_{\alpha, \plug}$ as
\begin{align}\label{eq:df_qpsi_n}
    \hat{\Psi}_{n, \plug} (x) = -\frac{1}{x K_n} - \frac{\plug}{x K_n} \sum_{i=1}^{K_n-1} \frac{1}{\plug + ix} + \underbrace{\frac{1}{x}- \mathbb{P}_n g_x}_{=\Psi(x)}.
\end{align}
\Cref{lm:qPsi_n_conv} below claims that $\hat{\Psi}_{n,\plug}$ and its derivative converge to $\Psi$ and its derivative, respectively, in a suitable sense. Here recall that $\Psi:(0,1)\to\mathbb{R}$ was defined by 
$\Psi(x) = x^{-1} - \mathbb{P} g_x$ as in \eqref{eq:df_Psi}.
\begin{lemma}\label{lm:qPsi_n_conv}
For each $\plug \in (-\alpha, \infty)$, it holds that
\begin{enumerate}
    \item[(A)] 
$\hat{\Psi}_{n,\plug}(x) \to^p \Psi(x)$  for all $x\in ((-\plug)\vee 0, 1)$.
    \item[(B)] $ \sup_{x\in I}|\hat{\Psi}'_{n,\plug}(x) - \Psi'(x)| \rightarrow 0$ a.s. for any closed subsets $I \subset ((-\plug)\vee 0, 1)$. 
\end{enumerate}
\end{lemma}
\begin{proof}
  We observe
  \begin{align}
      \hat{\Psi}_{n, \plug} (x) -\Psi(x) = -\frac{1}{x K_n} - \frac{\plug}{x K_n} \sum_{i=1}^{K_n-1} \frac{1}{\plug + ix} - (\mathbb{P}_n g_x - \mathbb{P} g_x ),\label{eq:Psi_n_hat-Psi_x_AP}
  \end{align}
  where $g_x(j)= \sum_{i=1}^{j-1} (i-x)^{-1}$.
Then, $K_n\to \infty$ (a.s.) and $(\mathbb{P}_n g_x - \mathbb{P} g_x )\to^p 0$ by \Cref{lm:logj_conv} imply $\hat{\Psi}_{n, \plug} (x) -\Psi(x) = o_p(1)$, which concludes the proof of (A). 
  
  As for (B), differentiation of \eqref{eq:Psi_n_hat-Psi_x_AP} leads to 
    \begin{align*}
      \hat{\Psi}'_{n, \plug}(x) -\Psi'(x) &= \frac{1}{x^2 K_n} + \frac{\plug}{x^2 K_n} \sum_{i=1}^{K_n-1}  \frac{1}{\plug + ix}
      - \frac{\plug}{x K_n}\sum_{i=1}^{K_n-1}\frac{i}{(\plug + ix)^2}\\
      &- (\mathbb{P}_n  h_x- \mathbb{P} h_x),
    \end{align*}
    where $h_x$ is the bounded function defined by $h_x(j) = \sum_{i=1}^{j-1} (i-x)^{-2}$.
  Thus, $\sup_{x\in I}|
          \hat{\Psi}'_{n, \plug}(x) - \Psi'(x)|$ can be bounded from above as follows:
   \begin{align*}
      \sup_{x\in I}\left|
          \hat{\Psi}'_{n, \plug}(x) - \Psi'(x)
      \right| &\leq \sup_{x\in I}\frac{1}{K_n x^2} + \sup_{x\in I} \frac{|\plug|}{x^2 K_n} 
      \sum_{i=1}^{K_n-1}  \frac{1}{\plug + ix}\\
      &+ \sup_{x\in I}\frac{\plug}{x K_n} \sum_{i=1}^{K_n-1}\frac{i}{(\plug + ix)^2}
      + \sup_{x\in I}\left|
          \mathbb{P}_n  h_x- \mathbb{P} h_x
          \right|.
   \end{align*}
   Let us write $I=[s, t]\subset ((-\plug) \vee 0 ,1)$. Then each term is bounded:
   \begin{align*}
      &0\leq \sup_{x\in I} \frac{1}{K_n x^2} \leq \frac{1}{K_n s^2}\\
      &0\leq\sup_{x\in I} \frac{|\plug|}{x^2 K_n} 
      \sum_{i=1}^{K_n-1}  \frac{1}{\plug + ix} \leq \frac{|\plug|}{s^2} \frac{1}{K_n} 
      \sum_{i=1}^{K_n-1} \frac{1}{\plug + is} = O(K_n^{-1} \log K_n) \\
      &0\leq\sup_{x\in I} \frac{|\plug|}{x K_n} 
      \sum_{i=1}^{K_n-1} \frac{i}{(\plug + ix)^2} \leq \frac{|\plug|}{s K_n}
      \sum_{i=1}^{K_n-1} \frac{i}{(\plug + i s)^2} = O(K_n^{-1} \log K_n)
      \\
      &0 \leq \sup_{x\in I} |\mathbb{P}_n h_x -\mathbb{P} h_x| \le \sum_{j=1}^\infty |\mathbb{P}_n(j) - \mathbb{P}(j)| \cdot \sum_{i=1}^\infty \frac{1}{(i-t)^2}
   \end{align*}
   Then $K_n\to \infty$ (a.s.) and  
   $\sum_{j=1}^\infty |\mathbb{P}_n(j) - \mathbb{P}(j)| \rightarrow 0$ (a.s.) by \Cref{lm:abs_conv} imply that they converge to $0$ (a.s.), which concludes the proof. 
  
\end{proof}
Putting (B)-\Cref{lm:qPsi_n_conv} and \Cref{lm:Psi_property} together, we obtain the following lemma:
\begin{lemma}\label{lm:random_seq_conv}
Let $I_\alpha$ be the Fisher information defined by \eqref{eq:sibuya_fisher}. 
Then for any sequence of random variables $(\bar{\alpha}_n)_{n\geq 1}$ s.t. $\bar{\alpha}_n \to^p \alpha$, and for all $\plug \in (-\alpha, \infty)$, $\hat{\Psi}'_{n, \plug}(\bar{\alpha}_n) \to^p -I_\alpha < 0$ holds.
\end{lemma}
\begin{proof}
$\Psi'(\alpha)=-I_\alpha$ by (B) of \Cref{lm:Psi_property} and the triangle inequality imply 
\begin{align*}
    |\hat{\Psi}'_{n,\plug}(\bar{\alpha}_n)
    + I_\alpha
    | \leq   |
    \Psi'(\bar{\alpha}_n) - \Psi'(\alpha)
    |+ |
    \hat{\Psi}'_{n,\plug}(\bar{\alpha}_n) - \Psi'(\bar{\alpha}_n)
    |,
\end{align*}
    where the first term is $o_p(1)$ from 
    the continuity of $\Psi'$ on $(0,1)$ (see (A) of \Cref{lm:Psi_property}) and $\bar{\alpha}_n -\alpha = o_p(1)$ by the assumption.
    For the second term, 
    take a sufficiently small $\delta$ s.t. $B_\delta(\alpha) := [\alpha \pm \delta] \subset ((-\plug)\vee 0,1)$. Then, for all $\epsilon>0$, 
    \begin{align*}
    \Pr(
    |
    \hat{\Psi}'_{n,\plug}(\bar{\alpha}_n) - \Psi'(\bar{\alpha}_n)
    |>\epsilon
    )\leq
    \Pr(
    \sup_{x\in B_\delta(\alpha)} |\hat{\Psi}'_{n,\plug}(x) - \Psi'(x)|>\epsilon
    ) + \Pr(\bar{\alpha}_n \not \in B_\delta),
    \end{align*}
    where the second term is $o(1)$ by $\bar{\alpha}_n \to^p \alpha$ and the first term is also $o(1)$ from (B) of \Cref{lm:qPsi_n_conv} applied with $I= B_\delta(\alpha)$. This concludes the proof.
\end{proof}
The next Lemma shows the consistency of QMLE.
\begin{lemma}\label{lm:qmle_consistency} 
$\hat{\alpha}_{n,\plug}\to^p \alpha$  for each $\plug \in (-\alpha, \infty)$.
\end{lemma}
\begin{proof}[Proof]
\Cref{prop:qmle_unique} implies 
$\hat{\Psi}'_{n,\plug}(x)$ is strictly decreasing in $x\in((-\plug)\vee 0,1)$ and $\hat{\Psi}'_{n,\plug}(\hat{\alpha}_{n,\plug}) = 0$ with a high probability. Under this event, for sufficiently small $\epsilon>0$ s.t. $[\alpha\pm\epsilon] \subset ((-\plug)\wedge 0, 1)$, it holds that 
\begin{align*}
    &\Pr\left(
        |\hat{\alpha}_{n,\plug} - \alpha| > \epsilon
        \right)\\
        & \leq \Pr(\hat{\alpha}_{n,\plug} < \alpha-\epsilon) + \Pr(\hat{\alpha}_{n,\plug} > \alpha+\epsilon)\\
        &\leq \Pr(0>\hat{\Psi}_{n,\plug}(\alpha-\epsilon)) + \Pr(0<\hat{\Psi}_{n,\plug}(\alpha+\epsilon))\\
        &= \Pr(\Psi(\alpha-\epsilon) - \hat{\Psi}_{n,\plug}(\alpha-\epsilon) > \Psi(\alpha-\epsilon)) +
        \Pr(\hat{\Psi}_{n,\plug}(\alpha+\epsilon) - \Psi(\alpha+\epsilon) > -\Psi(\alpha+\epsilon)). 
\end{align*}
Note $\Psi(\alpha-\epsilon) > 0 > \Psi(\alpha+\epsilon)$ since $\Psi(\alpha)=0$ and $\Psi$ is strictly decreasing (see (B) and (C) of \Cref{lm:Psi_property}). Then, (A) of \Cref{lm:qPsi_n_conv} implies that 
the upper bounds are $o(1)$. This concludes the proof.
\end{proof}
The next Lemma derives the asymptotic error between the two QMLE $(\hat{\alpha}_{n,\plug}, \hat{\alpha}_{n,0})$, where $\hat{\alpha}_{n,0}= \hat{\alpha}_{n,\plug=0}$. 
\begin{lemma}\label{lm:qmle_qmle0_error}
For each $\plug\in(-\alpha, \infty)$, it holds that
\begin{align*}
    \frac{n^\alpha}{\log n} (\hat{\alpha}_{n, \plug}  - \hat{\alpha}_{n,0})
        \to^p \frac{-\plug}{\alpha I_\alpha \Mit}.  
\end{align*}
\end{lemma}
Writing $\hat{\alpha}_{n,\plug} - \hat{\alpha}_{n,\theta}$ by $\hat{\alpha}_{n,\plug}-\hat{\alpha}_{n,0} - (\hat{\alpha}_{n,\theta} - \hat{\alpha}_{n,0})$ and applying \Cref{lm:qmle_qmle0_error} for each term, we get 
 \eqref{eq:qmle_mle_error} in \Cref{prop:error}. Below, we show \Cref{lm:qmle_qmle0_error}. 

\begin{proof}[Proof of \Cref{lm:qmle_qmle0_error}]
By \Cref{prop:qmle_unique}, we know that 
$\hat{\Psi}_{n, 0}(\hat{\alpha}_{n,0}) = 0$ and $\hat{\Psi}_{n, \plug}(\hat{\alpha}_{n, \plug}) = 0$ hold with a high probability. Under this event, Taylor's theorem implies that there exists $\bar{\alpha}_n$ between $\hat{\alpha}_{n,\plug}$ and $\hat{\alpha}_{n,\theta}$ s.t. 
$$0 = \hat{\Psi}_{n, 0}(\hat{\alpha}_{n,0}) = \hat{\Psi}_{n,0}(\hat{\alpha}_{n,\plug}) + \hat{\Psi}'_{n,0}(\bar{\alpha}_n) (\hat{\alpha}_{n,0}-\hat{\alpha}_{n,\plug}).$$
By \Cref{lm:qmle_consistency}, we know $\hat{\alpha}_{n,0}, \hat{\alpha}_{n,\plug}\to^p\alpha$, so in particular $\bar{\alpha}_n \to^p \alpha$. Then, \Cref{lm:random_seq_conv} applied with $\plug = 0$ implies  $-\hat{\Psi}'_{n,0}(\bar{\alpha}_n)\to^p I_\alpha$. In contrast, the expression of $\hat{\Psi}_{n,\plug}$ given by \eqref{eq:df_qpsi_n} leads to
$$\hat{\Psi}_{n, 0}(\hat{\alpha}_{n,\plug}) = \hat{\Psi}_{n,0 } (\hat{\alpha}_{n,\plug}) - \hat{\Psi}_{n,\plug}(\hat{\alpha}_{n,\plug}) = 
 \frac{1}{K_n\hat{\alpha}_{n,\plug}} \sum_{i=1}^{K_n-1} \frac{\plug}{\plug+ i \hat{\alpha}_{n,\plug}}.$$
Putting the above displays together, we obtain
\begin{align*}
     \frac{n^\alpha}{ \log n} (\hat{\alpha}_{n,\plug} - \hat{\alpha}_{n,0})
    &= \frac{n^\alpha}{\log n}\frac{\hat{\Psi}_{n,0}(\hat{\alpha}_{n,\plug})}{\hat{\Psi}'_{n,0}(\bar{\alpha}_n)}\\
    &= \frac{n^\alpha}{\log n} \frac{1}{\hat{\Psi}'_{n,0}(\bar{\alpha}_n)} \frac{1}{K_n\hat{\alpha}_{n,\plug}} \sum_{i=1}^{K_n-1} \frac{\plug}{\plug+ i \hat{\alpha}_{n,\plug}}\\
    &=\underbrace{\frac{n^\alpha}{K_n}}_{\to^P \Mit^{-1}} \cdot  \underbrace{\frac{\log K_n}{\alpha \log n}}_{\to^P 1} \cdot \underbrace{{\frac{\alpha}{(\hat{\alpha}_{n,\plug})^2}}}_{\to^P \alpha^{-1}}
    \cdot \underbrace{\frac{\plug}{\hat{\Psi}'_{n,0}(\bar{\alpha}_n)}}_{\to^P -\plug/I_\alpha}
    \cdot\frac{1}{\log K_n}\sum_{i=1}^{K_n-1} \frac{1}{\plug/\hat{\alpha}_{n,\plug} + i}\\
    &\overset{p}{\sim} \frac{-\plug}{\alpha I_\alpha \Mit} \cdot \frac{1}{\log K_n}\sum_{i=1}^{K_n-1} \frac{1}{\plug/\hat{\alpha}_{n,\plug} + i},
\end{align*}
so it remains to show $(\log K_n)^{-1}\sum_{i=1}^{K_n-1} (\plug/\hat{\alpha}_{n,\plug} + i)^{-1} \to^p 1$. Here, we introduce the basic property of the digamma function. 
\begin{lemma}\label{lm:digamma_ineq}
For the digamma function $\psi(z) = \Gamma'(z)/\Gamma(z)$, we have $\psi(n+\delta_n) - \log n = o(1)$ for all $\delta_n = o(n)$.
\end{lemma}
\begin{proof}
Use $\psi(n) - \log n = o(1)$ (cf. \cite[Section 6]{zwillinger2018crc}). 
\end{proof}
Applying  \Cref{lm:digamma_ineq} to $(\plug/\hat{\alpha}_{n,\plug})/K_n = o_p(1)$ and the basic equation
$\sum_{i=1}^{m-1} (i+c) =  \psi(m + c)-\psi(1+c)$, we obtain
 \begin{align*}
     \sum_{i=1}^{K_n-1} \frac{1}{i + \plug/\hat{\alpha}_{n,\plug}} - \log K_n 
     &= \psi(K_n + \plug/\hat{\alpha}_{n,\plug}) - \psi(1+ \plug/\hat{\alpha}_{n,\plug}) - \log K_n\\
     &= o_p(1) - \psi(1+\plug/\alpha_{n,\plug}) \\
     &\to^p - \psi(1+\plug/\alpha) && \text{by $\hat{\alpha}_{n,\plug}\to^p \alpha$}
 \end{align*}
 and $\psi(1+\plug/\alpha)$ is finite since $\plug > -\alpha$. 
 Combining this and $\log K_n \to \infty$ a.s., we obtain
$(\log K_n)^{-1}\sum_{i=1}^{KFe_n-1} (i + \plug/\hat{\alpha}_{n,\plug})^{-1}\to^p 1$, thereby completing the proof of \Cref{lm:qmle_qmle0_error}. 
\end{proof}

\subsection{Proof of \Cref{prop:qmle_stable}}
Here, we prove the asymptotic mixed normality of QMLE. The next Lemma gives stable Martingale CLT for general settings.
\begin{lemma}[{{\cite[p. 109]{hausler2015stable}}}]\label{lm:stable_martingale_clt}
Let $(X_n)_{n\geq 1}$ be a martingale difference sequence with respect to a filtration
$\mathscr{F} = (\mathcal{F}_n)_{n=1}^\infty$ 
and define $\mathcal{F}_{\infty} := \sigma(\cup_{n=1}^\infty \mathcal{F}_n)$.
For a sequence of positive real number $(a_n)_{n\geq 1}$ with $a_n \rightarrow \infty$, 
we assume the following two conditions:
\begin{enumerate}
  \item[$(\textrm{i})$] $a_n^{-2}\sum_{m=1}^{n} \E[X_m^2|\mathcal{F}_{m-1}] \to^p \eta^2$  for some random variable $\eta \geq 0$.
  \item[$(\textrm{ii})$] $a_n^{-2}\sum_{m=1}^{n} \E[X_m^2 \mathbbm{1}\{|X_m| \geq \epsilon a_n\}|\mathcal{F}_{m-1}] \to^p  0$ for all $\epsilon>0$.
\end{enumerate}
Then, 
${a_n}^{-1} \sum_{m=1}^{n} X_m \rightarrow \eta \N \ \mathcal{F_\infty}\text{-stably}$ holds, where $\N\sim \normal(0,1)$ is independent of $\mathcal{F}_\infty$.
\end{lemma}

\begin{corollary}\label{cor:stable_martingale_clt}
  In the setting of \Cref{lm:stable_martingale_clt}, if $X_n$ is uniformly bounded for all $n \in \mathbb{N}$, 
  $(\textrm{i})$ is sufficient for  ${a_n}^{-1} \sum_{m=1}^{n} X_m \rightarrow \eta \N \ (\stable)$ to hold.
\end{corollary}

\begin{proof}
It suffices to check that $(\textrm{ii})$ in \Cref{lm:stable_martingale_clt} is satisfied. Take a constant $C$ s.t. $|X_n| \leq C$ a.e. Then, $a_n \rightarrow \infty$ implies that $\forall\epsilon>0$, $\exists N_\epsilon\in\mathbb{N}$ s.t.
$
    n \geq N_\epsilon \Rightarrow a_n > C/\epsilon.
$
Observe that $n>N_\epsilon \Rightarrow |X_n| <  C < a_n\epsilon$. Then, for all $n>N_\epsilon$, 
$$\sum_{m=1}^{n} \E[X_m^2 \mathbbm{1}\{|X_m| \geq \epsilon a_n\}|\mathcal{F}_{m-1}]
    = \sum_{m=1}^{N_\epsilon} \E[X_m^2 \mathbbm{1}\{|X_m| \geq \epsilon a_n\}|\mathcal{F}_{m-1}]
    \leq N_\epsilon C^2,$$ which implies that  $a_n^{-2}\sum_{m=1}^{n} \E[X_m^2 \mathbbm{1}\{|X_m| \geq \epsilon a_n\}|\mathcal{F}_{m-1}]
    \leq \alpha_n^{-2} N_\epsilon C^2 \rightarrow 0$ (a.s.), especially in probability. This completes the proof. 
\end{proof}
Applying \Cref{cor:stable_martingale_clt} with $X_n$ taken to be the increment of the score function, we obtain the asymptotic mixed normality:
\begin{lemma}\label{lm:Psi_clt}
For $I_\alpha$ and $\ell_n(\alpha, \theta)$ defined by \eqref{eq:sibuya_fisher} and \eqref{eq:df_loglikelihood} respectively, we have
  \begin{align*}
    n^{-\alpha/2} \cdot \partial_\alpha \ell_n(\alpha, \theta) &\overset{}{\rightarrow}  
    \sqrt{\Mit I_\alpha}\cdot  \N \ \  \mathcal{F}_{\infty}\text{-stably as }  n\rightarrow \infty, 
  \end{align*}
  where $\Mit = \lim_{n\rightarrow\infty} n^{-\alpha}K_n$ and $\N\sim \normal (0,1)$ is independent of $\mathcal{F}_\infty$.
\end{lemma}

\begin{proof}
We take $X_n$ to be the increment of
  $\partial_\alpha \ell_n(\alpha, \theta)$
  with $a_n = n^{\alpha/2}$ and $\eta^2 = \Mit I_\alpha$ in \Cref{cor:stable_martingale_clt}. 
Since the score functions are $\mathcal{F}_n$-martingale, $(X_n)_{n=1}^\infty$ is an $\mathcal{F}_n$-martingale difference sequence. Then, it remains to show that $X_n$ is a bounded random variable and $n^{-\alpha} \sum_{m=1}^{n}\E\left[X_{m}^2| \mathcal{F}_{m-1}\right] \to^p I_\alpha \Mit$.

First, we check that $X_n$ is a bounded random variable. Recall 
\begin{align*}
    \partial_\alpha \ell_n(\alpha, \theta) = \sum_{i=1}^{K_n-1}\frac{i}{\theta + i\alpha} - \sum_{j=1}^n S_{n,j}\sum_{i=1}^{j-1}\frac{1}{i-\alpha}.
\end{align*}
The sequential definition given by \Cref{subsec:asym_ep} implies that the $(m+1)$-th ball belongs to a new urn with probability $(\theta + K_{m} \alpha)/(\theta + {m})$, or one of the urns of size $l$ with probability $S_{m,l}({l-\alpha})/({\theta + m})$ for each $l=1, 2, \dots, m$. In the former case, $K_{m+1} = K_{m} + 1$ and $S_{m+1,j} = S_{m,j} (\forall j\geq 2)$ holds, and in the latter case  $ K_{m+1} = K_{m}, S_{m+1,l} = S_{m,l} - 1$,  $S_{m+1,l+1} = S_{m,l+1} + 1$, 
    and $S_{m+1,j} = S_{m,j} (\forall j \notin \{l, l+1\})$. Therefore, for all $m\geq 2$, 
 \begin{align}\label{eq:increment}
    X_{m+1} |\mathcal{F}_{m}
    &= \left\{\begin{array}{ll}
       \sum_{i=1}^{K_{m}+1-1} \frac{i}{\theta+i\alpha}- \sum_{i=1}^{K_{m}-1} \frac{i}{\theta+i\alpha}& \text{w.p. } \frac{\theta + K_{m}\alpha}{\theta + m}\nonumber\\
         \sum_{i=1}^{l-1}\frac{1}{i-\alpha} - \sum_{i=1}^{l+1-1} \frac{1}{i-\alpha}& \text{w.p. } \frac{(l-\alpha)}{\theta + m}S_{m, l}, \ l \in [m]
     \end{array}
     \right.\\
    &= \left\{\begin{array}{ll}
        \frac{K_{m}}{\theta + K_{m}\alpha} & \text{w.p.} \frac{\theta + K_{m}\alpha}{\theta + m}\\
        -\frac{1}{l-\alpha} & \text{w.p. } \frac{(l-\alpha)}{\theta + m}S_{m, l}
        \text{ for all }l \in \{1, 2, \dots, m\}
     \end{array}
     \right.
 \end{align}
 and $X_1 = 0$.
Since $(l-\alpha)^{-1}\leq (1-\alpha)^{-1}$ for all $l\in \mathbb{N}$ and $k/(\theta + k \alpha) \leq ({\theta+\alpha})^{-1}\vee{\alpha}^{-1}$ for all $k\in \mathbb{N}$, $|X_n|$ is upper bounded by the constant $C:= (1-\alpha)^{-1}\vee ({\theta+\alpha})^{-1}\vee{\alpha}^{-1}$. 
 
Next, we show $\alpha^{-n} \sum_{m=1}^{n}\E\left[X_{m}^2| \mathcal{F}_{m-1}\right] \to^p I_\alpha \Mit$. 
Define $\sigma_n^2 := \sum_{m=1}^n\E\left[X_{m}^2| \mathcal{F}_{m-1}\right] = \sum_{m=0}^{n-1} \E\left[X_{m+1}^2| \mathcal{F}_{m}\right]$. Then, 
\eqref{eq:increment} implies that
\begin{align*}
    {\sigma_n^2} 
    &= 0^2 +\sum_{m=1}^{n-1} \left\{
    \frac{\theta + K_m\alpha}{\theta + m} \left(\frac{K_m}{\theta + K_m\alpha}\right)^2 +
    \sum_{j=1}^{m} \frac{j-\alpha}{\theta + m} S_{m, j} \left(\frac{1}{j-\alpha}\right)^2\right\}\\
    &= \sum_{m=1}^{n-1}
    \left\{
    \frac{K_m^2}{(\theta + m)(\theta + K_m \alpha)} + \sum_{j=1}^m \frac{S_{m,j}}{(\theta+m)(j-\alpha)} 
    \right\}
    \\
    &= \sum_{m=1}^{n-1} c_m Y_m
  \end{align*}
  where $c_m$ and $Y_m$ are defined as 
$$
c_m := \alpha m^{\alpha-1}, \quad Y_m :=  \frac{1}{\alpha m^{\alpha-1}} \Bigl\{
  \frac{K_m^2}{(\theta + m)(\theta + K_m \alpha)} + \sum_{j=1}^m \frac{S_{m,j}}{(\theta+m)(j-\alpha)} 
  \Bigr\}. 
$$
Note that $n^{-\alpha} \sum_{m=1}^{n-1} c_m \to 1$ as $n\to+\infty$. For $Y_m$, using the empirical measure $\mathbb{P}_m(j) = \frac{S_{m,j}}{K_m}$ and \Cref{cor:finite_conv}, we have that as $m\to+\infty$, 
\begin{align*}
  Y_m &= \frac{m}{m+\theta} \frac{K_m}{m^\alpha} \left\{\frac{K_m}{\alpha^2(\theta/\alpha + K_m)} + \mathbb{P}_m h \right\} && h(j) := \frac{1}{\alpha(j-\alpha)}\\
  &\to \Mit (\alpha^{-2} + \mathbb{P} h) && \text{by $m^{-\alpha} K_m \to \Mit$ and $\mathbb{P}_m h\to \mathbb{P}h$ }\\
  &= \Mit I_\alpha && \text{by \Cref{prop:sibuya_fisher}}
\end{align*}
Applying the triangle inequality to $n^{-\alpha} \sigma_n^{2}-I_\alpha \Mit = n^{-\alpha}\sum_{m=1}^{n-1} c_m Y_m - I_\alpha \Mit$, we have
\begin{align*}
    |n^{-\alpha} \sigma_n^{2}-I_\alpha \Mit| \leq n^{-\alpha} \sum_{m=1}^{n-1}c_m |Y_m - I_\alpha \Mit| +
     I_\alpha \Mit \left|n^{-\alpha} \sum_{m=1}^{n-1} c_m - 1\right|,
\end{align*}
where the second term is $o_p(1)$ by $n^{-\alpha} \sum_{m=1}^{n-1} c_m \to 1$. 
As for the first term, $Y_m\to \Mit I_\alpha$ (a.s.) implies that with probability $1$, 
$$
\forall \epsilon>0,\  \exists N_\epsilon \in \mathbb{N}\  \text{ such that } m> N_\epsilon \Rightarrow |Y_m - \Mit I_\alpha| <\epsilon.
$$
Therefore, for any $\epsilon>0$, as $n\to+\infty$, 
  \begin{align*}
    n^{-\alpha} \sum_{m=1}^{n-1} c_m |Y_m-\Mit I_\alpha|
    &\leq n^{-\alpha} \sum_{m=1}^{N_\epsilon} c_m |Y_m-\Mit I_\alpha| + \epsilon\cdot  n^{-\alpha} \sum_{m=1}^{n-1} c_m \\
    &= n^{-\alpha} O_p(1) + \epsilon \cdot n^{-\alpha} \sum_{m=1}^{n-1} c_m \\
    &\to^p 0 + \epsilon \cdot 1 =\epsilon.
  \end{align*}
 This means $n^{-\alpha}\sum_{m=1}^{n-1} c_m (Y_m-\Mit I_\alpha) \to^p 0$, thereby completing the proof. 
\end{proof}
Finally, we prove \Cref{prop:qmle_stable}; that is, the asymptotic mixed normality of QMLE.
$\hat{\alpha}_{n,\plug} - \hat{\alpha}_{n,\theta} = O_p(n^{-\alpha}\log n)$ by 
\eqref{eq:qmle_mle_error} in \Cref{prop:error} implies 
\begin{align*}
    n^{\alpha/2}(\hat{\alpha}_{n,\plug} - \alpha)& = n^{\alpha/2} (\hat{\alpha}_{n,\plug} - \hat{\alpha}_{n,\theta}) + n^{\alpha/2} ( \hat{\alpha}_{n,\theta} - \alpha)\\
    &= O_p(n^{-\alpha/2}\cdot\log n) + n^{\alpha/2} ( \hat{\alpha}_{n,\theta} - \alpha) \\
    &= o_p(1) + n^{\alpha/2} ( \hat{\alpha}_{n,\theta} - \alpha), 
\end{align*}
 for all $\plug\in (-\alpha,\infty)$.
  Therefore, it suffices to prove the claim for $\plug=\theta$.
  Let $\hat{\Psi}_{n,\theta} (\alpha) = K_n^{-1} \partial_\alpha \ell_n(\alpha, \theta)$. Then  $\hat{\Psi}'_{n,\theta}(x)<0$ for all $x\in(-\theta\vee 0,1)$ and $\hat{\Psi}_{n,\theta}(\hat{\alpha}_{n,\theta}) = 0$ with a high probability by \Cref{prop:qmle_unique}. Under this event, Taylor's theorem implies that
 there exists $\bar{\alpha}_n$ between $ \alpha$ and $\hat{\alpha}_{n,\theta}$ (hence $\bar{\alpha}_n - \alpha = o_p(1)$ by \Cref{lm:qmle_consistency} with $\plug=\theta$) s.t. $ \hat{\Psi}_{n,\theta}(\alpha) + \hat{\Psi}'_{n, \theta}(\bar{\alpha}_n) (\hat{\alpha}_{n,\theta} - \alpha)=0$. Combining this and $\hat{\Psi}_{n,\theta} (\alpha) = K_n^{-1} \partial_\alpha \ell_n(\alpha, \theta)$, we obtain
  \begin{align*}
    \sqrt{I_\alpha n^\alpha }\cdot  (\hat{\alpha}_{n,\theta} - \alpha)  = 
    \sqrt{I_\alpha n^\alpha} \cdot
    \frac{\hat{\Psi}_{n,\theta}(\alpha)}{
      -\hat{\Psi}'_{n,\theta}(\bar{\alpha}_n)
    }
    =  \underbrace{\frac{\partial_\alpha \ell_n(\alpha, \theta)
    }{\sqrt{n^\alpha I_\alpha}}}_{=:X_n}
    \cdot \underbrace{\frac{I_\alpha}{-\hat{\Psi}'_{n,\theta}(\bar{\alpha}_n)}\frac{n^\alpha}{K_n}}_{=:Y_n},
  \end{align*}
  where $X_n \rightarrow \sqrt{\Mit} \N$ $(\stable)$ by \Cref{lm:Psi_clt}, while $Y_n \to ^p \Mit^{-1}$ holds since $n^\alpha/K_n \to \Mit^{-1}$ (a.s.) and 
  $\hat{\Psi}'_{n,\theta}(\bar{\alpha}_n) \to^p -I_\alpha$ by \Cref{lm:random_seq_conv} and applied with $\plug = \theta$. Then, {Slutsky's} lemma applied with the $\mathcal{F}_\infty $-measurability of $\Mit$ and continuous mapping theorem (see \Cref{th:CS_stable}) result in
  \begin{align*}
      \sqrt{n^{\alpha} I_\alpha} (\hat{\alpha}_{n,\theta}-\alpha) = X_n \cdot Y_n \to \sqrt{\Mit} \N \cdot \Mit^{-1} = \N/\sqrt{\Mit} \ (\stable),
  \end{align*}
which concludes the proof.

\section{Proofs for MLE}
We have discussed QMLE so far, where the unknown $\theta$ is fixed to some value $\plug\in (-\alpha, +\infty)$. In this section, we consider the MLE 
$$
(\hat{\alpha}_n, \hat{\theta}_n) \in \argmax_{x\in (0,1), y>-x} \ell_n(x,y)
$$
where we simultaneously estimate $(\alpha, \theta)$. 

First, we reduce the dimension of the parameter that we have to consider: 
we define the function $\hat{y}_n$ : $(0,1) \to (-1, \infty)$ by
\begin{align*}
    \forall x\in (0,1), \quad  \hat{y}_n(x) \in \argmax_{y > -x} \ell_n(x, y),
\end{align*}
where $\ell_n(x, y) = \ell_n(\alpha, \theta)|_{(\alpha, \theta)=(x, y)}$ is the log likelihood. \Cref{lm:y_n} below shows that $\hat{y}_n$ is well-defined. 
\begin{lemma}[{\cite[Lemma 5.3]{carlton1999applications}}]\label{lm:y_n}
Suppose $1<K_n<n$. Then, for all $x\in (0,1)$, the nonlinear equation
\begin{align*}
 \partial_y \ell_n(x, y) = \sum_{i=1}^{K_n-1} \frac{1}{y+ix} - \sum_{i=1}^{n-1} \frac{1}{y+i} = 0 \quad \text{for} \quad y>-x
\end{align*}
admits a unique solution, and $\partial^2_y \ell_n(x, \cdot )$ takes a negative value at the solution.
\end{lemma}
Observe that the condition $1<K_n<n$ holds with a high probability since $n^{-\alpha}K_n \to \Mit > 0$ (a.s.). 
\Cref{lm:y_n} claims that the log-likelihood is an unimodal function of $y$ for each $x\in (0,1)$. This means that $\hat{y}_n$ is well-defined and uniquely characterized by the stationary condition:
\begin{align}\label{eq:df_y_n}
    \forall x\in (0,1), \quad  \partial_y \ell_n(x, \hat{y}_n(x)) = \sum_{i=1}^{K_n-1} \frac{1}{\hat{y}_n(x)+ix} - \sum_{i=1}^{n-1} \frac{1}{\hat{y}_n(x)+i} = 0.
\end{align}
Furthermore, the implicit function theorem with $\partial_y^2 \ell_n(x, y)\mid_{y=\hat{y}_n(x)} < 0$ implies that $\hat{y}_n$ is differentiable. 
The gain of introducing $\hat{y}_n$ is that the MLE $(\hat{\alpha}_n,\hat{\theta}_n)$ defined as the maxima of the log-likelihood can be formulated as the solution of the one-dimensional maximization problem: 
\begin{align}\label{eq:smle_Psi_hat_n}
  \hat{\alpha}_n \in  \argmax_{x\in (0,1)} \ell_n(x, \hat{y}_n(x)), \ \ \
  \hat{\theta}_n = \hat{y}_n(\hat{\alpha}_n).
\end{align}
Here, in a manner similar to $\hat{\Psi}_{n, \plug}$ \eqref{eq:df_qpsi_n}, we define the random function $\hat{\Psi}_n$ by
\begin{align}\label{eq:df_Psi_n_hat}
    \forall x\in (0,1), \quad  \hat{\Psi}_n(x) := \frac{1}{K_n}\cdot\frac{\diff}{\diff x} \ell_n(x, \hat{y}_n(x)).
\end{align}
Note that $ \hat{\Psi}_n(x) $ can be written as 
\begin{align*}
  \hat{\Psi}_n(x) 
  &=\frac{1}{K_n}\partial_x\ell_n(x, \hat{y}_n(x)) + \frac{1}{K_n} \partial_y\ell_n(x, \hat{y}_n(x))\cdot \hat{y}'_n(x) \\
  &= \frac{1}{K_n} \partial_x\ell_n(x, \hat{y}_n(x)) && \text{by $\partial_y \ell_n(x, \hat{y}_n(x)) = 0$}\\
  &= \frac{1}{x}  -\frac{1}{K_n x} - \frac{1}{K_n x}\sum_{i=1}^{K_n-1}
  \frac{\hat{y}_n(x)}{\hat{y}_n(x) + ix}- \mathbb{P}_n g_x && \text{by \eqref{eq:case_1}}\\
  &=\frac{1}{x}  -\frac{1}{K_n x} - \frac{1}{K_n x}\sum_{i=1}^{n-1}
  \frac{\hat{y}_n(x)}{\hat{y}_n(x) + i} - \mathbb{P}_n g_x && \text{by \eqref{eq:df_y_n}}
\end{align*}
Here, we define the function $\hat{z}_n: (0,1) \rightarrow (-1,\infty)$ as
\begin{align}\label{eq:df_z_n}
  \hat{z}_n: (0,1)\to (-1,\infty), \quad x \mapsto \hat{y}_n(x)/x
\end{align}
Then, using \eqref{eq:df_z_n} and $\Psi(x) = x^{-1} - \mathbb{P} g_x$, we can write $\hat{\Psi}_n-\Psi$ as 
\begin{align}\label{eq:hat_Psi_x-Psi}
    \hat{\Psi}_n (x) - \Psi(x) =  -\frac{1}{K_n x} - \frac{1}{K_n}\sum_{i=1}^{n-1}
  \frac{\hat{z}_n(x)}{x\hat{z}_n(x) + i} - (\mathbb{P}_n g_x - \mathbb{P} g_x).
\end{align}
\Cref{lm:uniform_control_delta_n} below provides a uniform control of $\hat{\Psi}_n'(x)$ and $\hat{z}_n(x)$ over a local neighborhood of $\alpha$.  

\begin{lemma}\label{lm:uniform_control_delta_n}
  For any positive constant $\delta_n = o(1/\log n)$, we have that
  \begin{enumerate}
    \item[(A)] $\sup_{x\in [\alpha \pm \delta_n]} |f_\alpha(\hat{z}_n(x)) -\log (K_n/n^\alpha)| = o_p(1)$. 
    \item[(B)] $\sup_{x\in[\alpha\pm\delta_n]} |\hat{\Psi}'_n(x) - \Psi'(x)| = o_p(1)$. 
  \end{enumerate}
\end{lemma}
\begin{proof}
  See \Cref{proof:uniform_control_delta_n_A} and \Cref{proof:uniform_control_delta_n_B}.
\end{proof}

{
\subsection{Proof of \Cref{prop:smle_unique}}\label{subsec:proof_smle_unique}
As we discussed at the beginning of this section, under the event $\{1<K_n<n\}$, which holds with a high probability, it holds that
$$
(\hat{\alpha}_n, \hat{\theta}_n) \in \argmax_{\alpha\in (0,1), \theta>-\alpha} \ell_n(\alpha, \theta) \Leftrightarrow
\Bigl(\hat{\alpha}_n \in \argmax_{x\in (0,1)} \ell_n(x, \hat{y}_n(x)) \text{ and } \hat{\theta}_n = \hat{y}_n(\hat{\alpha}_n)\Bigr).
$$
By \cite[Proposition 3 and 6]{balocchi2022bayesian}, 
there exist some positive constants $(c, \delta)$ such that with high probability, $\argmax_{x\in (0,1)} \ell_n(x, \hat{y}_n(x))$ is nonempty and is a subset of $[\hat{\alpha}_{n,0} \pm c n^{-\delta}]$. Here, $\hat{\alpha}_{n,0}$ is the QMLE with $\plug=0$ and \Cref{prop:qmle_stable} implies $\hat{\alpha}_{n,0}=\alpha + O_p(n^{\alpha/2})$. Therefore, it holds that 
$$
\Pr\Bigl( \emptyset \ne \argmax_{x\in (0,1)} \ell_n(x, \hat{y}_n(x)) \text{ and } \argmax_{x\in (0,1)} \ell_n(x, \hat{y}_n(x)) \subset [\alpha \pm c n^{-\delta}]\Bigr) \to 1
$$
for a sufficiently small $\delta>0$ and large $c>0$. Let $I=[\alpha/2, (\alpha+1)/2]\subset (0,1)$ so that $[\alpha\pm cn^{-\delta}]\subset I$ for sufficiently large $n$. Since $\Psi$ is $C^1$ and strictly decreasing (see \Cref{lm:Psi_property}), we can take a positive constant $A>0$ such that 
$
A = 2^{-1}\inf_{x\in I} |\Psi'(x)| = -2^{-1} \sup_{x\in I}\Psi'(x).
$
In contrast, \Cref{lm:uniform_control_delta_n} implies $\sup_{x\in [\alpha\pm cn^{-\delta}]} |\hat{\Psi}'_n(x) - \Psi'(x)| \leq A$ with a high probability. Putting them together, we have
\begin{align*}
    \sup_{x\in [\alpha\pm cn^{-\delta}]} \hat{\Psi}'_n(x) 
    &\leq \sup_{x\in I} \Psi'(x) + \sup_{x\in [\alpha\pm cn^{-\delta}]} (\hat{\Psi}'_n(x) - \Psi'(x)) && \text{by $\hat{I}\subset I$}\\
    &= -2A + \sup_{x\in [\alpha\pm cn^{-\delta}]} (\hat{\Psi}'_n(x) - \Psi'(x)) \\
    & \le -A && \text{by } \sup_{x\in [\alpha\pm cn^{-\delta}]} |\hat{\Psi}'_n(x) - \Psi'(x)| \leq A
\end{align*}
Therefore, we conclude that the event $\Omega_n$ defined as 
$$
\Omega_n := \Bigl\{\emptyset \ne \argmax_{x\in (0,1)}\ell_n (x, \hat{y}_n(x)) \subset [\alpha\pm cn^{-\delta}]\Bigr\} \cap \Bigl\{\sup_{x\in [\alpha\pm cn^{-\delta}]} \Psi_n'(x) <0
\Bigr\}
$$
holds with a high probability. Now, we claim that under $\Omega_n$, $\argmax_{x\in (0,1)}\ell_n (x, \hat{y}_n(x))$ has a unique element. 
We proceed by contradiction. Suppose $\argmax_{x\in (0,1)}\ell_n (x, \hat{y}_n(x))$ has two distinct elements $x, x'$. Noting $\Psi(x)=\frac{1}{K_n} \frac{d}{dx}\ell_n(x, \hat{y}_n(x))$, this means $\hat{\Psi}_n(x)=\hat{\Psi}_n(x')=0$, and by the mean value theorem, there exists some point $\bar{x}$ between $x$ and $x'$ such that $\Psi'_n(\bar{x}) = 0$. However, due to $x, x'\subset [\alpha\pm c n^{-\delta}]$, $\bar{x}$ also belong to $[\alpha\pm c n^{-\delta}]$, which contradicts the assumption $\sup_{x\in [\alpha\pm cn^{-\delta}]} \Psi_n'(x) <0$. Therefore, under $\Omega_n$, $\argmax_{x\in (0,1)}\ell_n (x, \hat{y}_n(x))$ has a unique element. Since $\Pr(\Omega_n)\to 1$, this completes the proof. 

}
\subsection{Proof of \eqref{eq:smle_mle_error} in \Cref{prop:error} and \Cref{th:smle_asym}}

By the results in \Cref{subsec:proof_smle_unique}, we get a rough rate $\hat{\alpha}_n = \alpha + O_p(n^{-\delta})$ for a sufficiently small $\delta>0$. Combined with \Cref{lm:uniform_control_delta_n}, we obtain 
$f_\alpha(\hat{z}_n(\hat{\alpha}_n)) -\log (K_n/n^\alpha) = o_p(1)$. 
With $K_n/n^\alpha \to\Mit$ and $\hat{\theta}_n = \hat{y}_n(\hat{\alpha}_n)=\hat{\alpha}_n \hat{z}_n(\hat{\alpha}_n)$, we get $\hat{z}_n(\hat{\alpha}_n) \to^p f_\alpha^{-1}(\log \Mit)$ and $\hat{\theta}_n \to^p \alpha f_\alpha^{-1}(\log \Mit)$. 

Let us show \eqref{eq:smle_mle_error} in \Cref{prop:error} first. 
Recall that \Cref{lm:qmle_qmle0_error} gives the error of $(\hat{\alpha}_{n,\plug}, \hat{\alpha}_{n,0})$ as follows:
\begin{align*}
    \frac{n^\alpha}{\log n} (\hat{\alpha}_{n, \plug}  - \hat{\alpha}_{n,0})
        \to^p \frac{-\plug}{I_\alpha \Mit}
\end{align*}
Thus, if we can show 
 \begin{align}\label{eq:mle_qmle_0_err}
        \frac{n^\alpha}{\log n} (\hat{\alpha}_n  - \hat{\alpha}_{n,0})
        \to^p \frac{-f_\alpha^{-1}(\log \Mit)}{I_\alpha \Mit}
\end{align}
then \eqref{eq:smle_mle_error} follows from the two displays above. Below, we show \eqref{eq:mle_qmle_0_err}. By \eqref{eq:Psi_qPsi_0} and $\Psi_n(\hat\alpha_n) = 0$, we have 
$$
\hat{\Psi}_{n,0}(\hat{\alpha}_n) = \hat{\Psi}_{n,0}(\hat{\alpha}_n) - \hat{\Psi}_n(\hat{\alpha}_n)
  = \frac{1}{K_n} \sum_{i=1}^{n-1}\frac{\hat{z}_n(\hat{\alpha}_n)}{x\hat{z}_n(\hat{\alpha}_n) + i}.
$$
By Taylor's theorem, there exists some $\bar{\alpha}_n$ between $\hat{\alpha}_n$ and $\hat{\alpha}_{n,0}$ such that $\bar{\alpha}_n\to^p\alpha$ and 
\begin{align*}
 \hat{\alpha}_n -\hat{\alpha}_{n,0} =
  \frac{\hat{\Psi}_{n,0}(\hat{\alpha}_n)}{\hat{\Psi}'_{n,0}(\bar{\alpha}_n)}
  = \frac{1}{K_n\cdot\hat{\Psi}'_{n,0}(\bar{\alpha}_n)} \sum_{i=1}^{n-1} \frac{\hat{z}_n(\hat{\alpha}_n)}{x\hat{z}_n(\hat{\alpha}_n)+ i}.
\end{align*} 
This gives 
\begin{align*}
    \frac{n^\alpha}{\log n} (\hat{\alpha}_n  - \hat{\alpha}_{n,0})
        &= \frac{1}{\hat{\Psi}_{n,0}(\bar{\alpha}_n)}\cdot \frac{n^\alpha}{K_n} \cdot \frac{\hat{z}_n(\hat{\alpha}_n)}{\log n} \sum_{i=1}^{n-1}\frac{1}{\hat{\alpha}_n \hat{z}_n(\hat{\alpha}_n) + i}\\
        &\overset{p}{\sim} \frac{1}{-I_\alpha}\cdot  \frac{1}{\Mit} \cdot f_\alpha^{-1}(\log \Mit) \cdot \frac{1}{\log n}  \sum_{i=1}^{n-1}\frac{1}{\hat{\alpha}_n \hat{z}_n(\hat{\alpha}_n) + i},
\end{align*}
where $\overset{p}{\sim}$ follows from $-\hat{\Psi}_{n,0}(\bar{\alpha}_n) \to^p I_\alpha$ by 
\Cref{lm:random_seq_conv} with $\plug=0$ and $\hat{z}_n(\hat{\alpha}_n) \to^p  f_\alpha^{-1} (\log \Mit)$.  Now we claim
 $(\log n)^{-1} \sum_{i=1}^{n-1}(\hat{\alpha}_n \hat{z}_n(\hat{\alpha}_n) + i)^{-1}\to^p 1$: \Cref{lm:digamma_ineq} with $\hat{\alpha}_n \hat{z}_n(\hat{\alpha}_n) = O_p(1) = o_p(n)$ implies that
\begin{align*}
\sum_{i=1}^{n-1} \frac{1}{\hat{\alpha}_n \hat{z}_n(\hat{\alpha}_n) + i} - \log n &= \psi(n +\hat{\alpha}_n \hat{z}_n(\hat{\alpha}_n)) - \log n - \psi(1 +\hat{\alpha}_n \hat{z}_n(\hat{\alpha}_n))\\
  &\to^p 0 - \psi(1+\alpha f_\alpha^{-1} (\log \Mit)),
\end{align*}
so $(\log n)^{-1} \sum_{i=1}^{n-1}(\hat{\alpha}_n \hat{z}_n(\hat{\alpha}_n) + i)^{-1} - 1 = o_p(1)$. This concludes the proof of \eqref{eq:mle_qmle_0_err} and finishes the proof of \eqref{eq:smle_mle_error} in \Cref{prop:error}. 

Finally, combining \eqref{eq:mle_qmle_0_err} and \Cref{prop:qmle_stable} together, we obtain
\begin{align*}
\sqrt{n^\alpha I_\alpha}(\hat{\alpha}_n - \alpha) &= \sqrt{n^\alpha I_\alpha}(\hat{\alpha}_n - \hat{\alpha}_{n,0}) + \sqrt{n^\alpha I_\alpha}(\hat{\alpha}_{n,0} - \alpha)\\
  &= O_p(n^{-\alpha/2} \log n) +  \sqrt{n^\alpha I_\alpha}(\hat{\alpha}_{n,0} - \alpha) && \text{ by \eqref{eq:mle_qmle_0_err}}\\
  &\to 0 + \N/\sqrt{\Mit} && \text{ by \Cref{prop:qmle_stable}},
\end{align*}
which concludes the proof of \Cref{th:smle_asym}.

\subsection{Proof of \Cref{lm:uniform_control_delta_n}-(A)}\label{proof:uniform_control_delta_n_A}
\begin{lemma}\label{Ap:z_n_nonasym}
Suppose $1<K_n<n$. Then, for all $x\in(0,1)$, it holds that
\begin{align*}
|\hat{z}_n(x)| \leq 1 \vee \left(
   A_n(x) \wedge 
   B_n(x)\wedge C_n(x;\alpha)
    \right) \quad \text{where } \begin{cases}
      A_n (x) &:= \frac{n K_n}{x(n-K_n)} \\
      B_n (x) &:= K_n \left\{
      \left(
      \frac{n+\hat{z}_n(x)}{1+\hat{z}_n(x)}
      \right)^x -1
      \right\}^{-1} \\
      C_n(x;\alpha) &:= n^{\frac{\alpha-x}{1-x}}\left(
      2\frac{K_n+\hat{z}_n(x)}{n^\alpha}
      \right)^{\frac{1}{1-x}}.
    \end{cases}
\end{align*}
\end{lemma}
\begin{proof}
We fix $x\in(0,1)$.
Suppose $\hat{z}_n(x) < 1$. Then the assertion is obvious from $1>\hat{z}_n(x) = \hat{y}_n(x)/x > -x/x = -1$, so we consider the case $\hat{z}_n(x) \geq 1$. Recall that \eqref{eq:df_y_n} and $\hat{y}_n(x)=x\hat{z}_n(x)$ give 
\begin{align*}
     \sum_{i=1}^{n-1}\frac{x}{x\hat{z}_n(x) + i} = \sum_{i=1}^{K_n-1}\frac{1}{\hat{z}_n(x) + i}.
\end{align*}
Combined with the basic inequality $\frac{(m-1)}{\delta + (m-1)} <\sum_{i=1}^{m-1}\frac{1}{\delta+i} < \frac{m-1}{\delta+1}$, we obtain 
\begin{align*}
     \frac{x(n-1)}{x\hat{z}_n(x) + n-1} \le \frac{K_n-1}{\hat{z}_n(x) + 1}.
\end{align*}
Rearranging the above display, we get $\hat{z}_n(x)\le A_n(x)$. 
In contrast, by elementary calculus, we have that for all $ m \in \mathbb{N}_{\geq 2}$ and for all $\delta>0$, 
    \begin{align*}
        \log \left(\frac{\delta + m}{\delta + 1}\right) < \sum_{i=1}^{m-1}\frac{1}{\delta + i} < \log\left(
        \frac{\delta + m-1}{\delta}\right) < \log \left(
        \frac{\delta + m}{\delta}
        \right).
    \end{align*}
    Applying this inequality to $\sum_{i=1}^{n-1}\frac{x}{x\hat{z}_n(x) + i} = \sum_{i=1}^{K_n-1}\frac{1}{\hat{z}_n(x) + i}$, using the assumption $\hat{z}_n(x) \geq 1>0$, we have 
    \begin{align*}
        x\log \left(\frac{\hat{z}_n(x)+n}{\hat{z}_n(x)+1}\right)<\sum_{i=1}^{n-1}\frac{x}{\hat{z}_n(x) + i}\leq\sum_{i=1}^{n-1}\frac{x}{x\hat{z}_n(x) + i} = \sum_{i=1}^{K_n-1}\frac{1}{\hat{z}_n(x) + i} <\log\left(\frac{\hat{z}_n(x)+K_n}{\hat{z}_n(x)}\right)
    \end{align*}
    and hence 
    \begin{align}\label{eq:z_n_inequality}
      \left(\frac{\hat{z}_n(x)+n}{\hat{z}_n(x)+1}\right)^x < \frac{\hat{z}_n(x)+ K_n}{\hat{z}_n(x)}.
    \end{align}
    Rearranging this, we obtain $\hat{z}_n(x)\le B_n(x)$. 
    Furthermore, \eqref{eq:z_n_inequality} and the assumption $\hat{z}_n(x) \geq 1 >0$ yield
    \begin{align*}
        \frac{\hat{z}_n(x)}{(\hat{z}_n(x)+1)^x} < \frac{\hat{z}_n(x)+K_n}{(\hat{z}_n(x)+n)^x} < \frac{\hat{z}_n(x)+K_n}{(0+n)^x} = \frac{\hat{z}_n(x)+K_n}{n^x}
    \end{align*}
    and 
    \begin{align*}
      \hat{z}_n(x)^{1-x} <\left(\frac{1+\hat{z}_n(x)}{\hat{z}_n(x)}\right)^x  \frac{\hat{z}_n(x)+K_n}{n^x} 
      \leq \left(\frac{1+1}{1}\right)^x  \frac{\hat{z}_n(x)+K_n}{n^x} < n^{\alpha-x} \cdot 2\frac{\hat{z}_n(x)+K_n}{n^\alpha}.
    \end{align*}
    Rearranging the above display, we get $\hat{z}_n(x)\le C_n(x;\alpha)$. 

\end{proof}

Next, we prove \Cref{lm:z_n} below. Note that \Cref{lm:uniform_control_delta_n}-(A) follows as a corollary of \Cref{lm:z_n}-(C). 

\begin{lemma}\label{lm:z_n}
$\hat{z}_n$ satisfies the followings:
\begin{enumerate}
    \item[(A)] $\hat{z}_n(x)$ is strictly decreasing in $x$, and its derivative is  given by
    \begin{align*}
        z'_n(x) = \frac{\sum_{i=1}^{n-1}i\cdot(x\hat{z}_n(x)+i)^{-2}} {{-\sum_{i=1}^{K_n-1}(\hat{z}_n(x)+i)^{-2} + \sum_{i=1}^{n-1}x^2\cdot (x\hat{z}_n(x)+i)^{-2}}}.
    \end{align*}
    \item [(B)] 
    $|\hat{z}_n(x)| = 1 \vee O_p(n^{\frac{a-x}{1-x}})
    $.
    \item[(C)] For all $\delta_n=o(1/\log n)$,  
    \begin{align*}
      \sup_{x\in [\alpha \pm \delta_n]}|\hat{z}_n(x)| = O_p(1) \quad \text{and} \quad 
       \sup_{x\in [\alpha \pm \delta_n]} |f_\alpha(\hat{z}_n(x)) -\log (K_n/n^\alpha)| = o_p(1). 
\end{align*}
\end{enumerate}
\end{lemma}
\begin{proof}
\textbf{Proof of (A):}
Suppose $1<K_n<n$. Then, for all $x\in(0,1)$, \Cref{lm:y_n} implies that $\hat{y}_n(x)$ is the unique solution of 
\begin{align*}
    \partial_y\ell_n(x,y)= \sum_{i=1}^{K_n-1}\frac{1}{y + ix} - \sum_{i=1}^{n-1} \frac{1}{y + i}=0, \quad y>-x
\end{align*}
and $\partial^2_y \ell_n(x, \hat{y}_n(x)) < 0$. Therefore, the implicit function theorem implies
\begin{align*}
    \hat{y}_n'(x) =  -\frac{\partial_x \partial_y\ell_n(x,\hat{y}_n(x))}{\partial_y^2 \ell_n(x,\hat{y}_n(x))}.
\end{align*}
The derivative of $\hat{z}_n(x) = \hat{y}_n(x)/x$ can be written as
\begin{align}\label{eq:hat_z_derivative}
    \hat{z}_n'(x) = \frac{-\hat{y}_n(x) + x\hat{y}_n'(x)}{x^2} = \frac{-\hat{y}_n(x) \partial_y^2 \ell_n(x,\hat{y}_n(x)) - x\partial_x \partial_y\ell_n(x,\hat{y}_n(x))}{x^2 \partial_y^2 \ell_n(x,\hat{y}_n(x)) }.
\end{align}
Here, the numerator is calculated as  
\begin{align*}
    &-\hat{y}_n(x) \partial_y^2 \ell_n(x,\hat{y}_n(x)) - x\partial_x \partial_y\ell_n(x,\hat{y}_n(x))\\
    &=
    \sum_{i=1}^{K_n-1}\frac{\hat{y}_n(x)}{(\hat{y}_n(x)+ix)^2} - \sum_{i=1}^{n-1}\frac{\hat{y}_n(x)}{(\hat{y}_n(x)+i)^2} + \sum_{i=1}^{K_n-1}\frac{ix}{(\hat{y}_n(x)+ix)^2}\\
    &= \sum_{i=1}^{K_n-1}\frac{1}{\hat{y}_n(x) + ix} - \sum_{i=1}^{n-1}\frac{\hat{y}_n(x)}{(\hat{y}_n(x) + i)^2}\\
    &= \sum_{i=1}^{n-1}\frac{1}{\hat{y}_n(x) + i} - \sum_{i=1}^{n-1}\frac{\hat{y}_n(x)}{(\hat{y}_n(x) + i)^2} \qquad \text{by $\sum_{i=1}^{K_n-1}\frac{1}{\hat{y}_n(x) + ix} = \sum_{i=1}^{n-1} \frac{1}{\hat{y}_n(x) + i}$ }\\
    &= \sum_{i=1}^{n-1}\frac{i}{(\hat{y}_n(x)+i)^2},
\end{align*}
As for the denominator of \eqref{eq:hat_z_derivative}, with $\partial_y^2 \ell_n(x, \hat{y}_n(x))<0$, we have 
\begin{align*}
   0>x^2\partial_y^2 \ell_n(x, \hat{y}_n(x)) &= -\sum_{i=1}^{K_n-1}\frac{x^2}{(\hat{y}_n(x)+ix)^2}
    + \sum_{i=1}^{n-1}\frac{x^2}{(\hat{y}_n(x)+i)^2}.
\end{align*}
Therefore, $\hat{z}_n'(x)$ is strictly negative, and with $\hat{y}_n(x)=x\hat{z}_n(x)$, we obtain 
\begin{align*}
 0 > \hat{z}'_n(x) = \frac{\sum_{i=1}^{n-1} {i}/{(x \hat{z}_n(x)+i)^2} }{-\sum_{i=1}^{K_n-1}(\hat{z}_n(x)+i)^{-2}
    + \sum_{i=1}^{n-1}{x^2}/{(x \hat{z}_n(x)+i)^2}}.
\end{align*}
This completes the proof of (A).

\textbf{Proof of (B):}
We fix $x\in(0,1)$. 
For $A_n(x)$ in \Cref{Ap:z_n_nonasym}, $K_n = O_p(n^\alpha) = o_p(n)$ implies 
\begin{align*}
    \frac{A_n(x)}{n} = \frac{K_n}{x(n-K_n)} = o_p(1),
\end{align*}
and hence $|\hat{z}_n(x)| \leq 1 \vee A_n(x) = o_p(n)$. Considering this, for $B_n(x)$ in \Cref{Ap:z_n_nonasym}, it holds that
\begin{align*}
    \frac{B_n(x)}{n^\alpha} = \frac{K_n}{n^\alpha} \cdot \left\{
    \left(
    \frac{n+\hat{z}_n(x)}{1+\hat{z}_n(x)}
    \right)^x -1
    \right\}^{-1} \to^p  \Mit \cdot\  0 = 0,
\end{align*}
which implies $|\hat{z}_n(x)| \leq 1 \vee B_n(x) = o_p(n^\alpha)$.
Then, for $C_n(x;\alpha)$ in \Cref{Ap:z_n_nonasym}, we get
\begin{align*}
    n^{-\frac{\alpha-x}{1-x}} \cdot C_n(x;\alpha) = \left(
    2\frac{K_n+\hat{z}_n(x)}{n^\alpha}
    \right)^{\frac{1}{1-x}} \to^p \left(
    2 \Mit + 0
    \right)^{\frac{1}{1-x}} = (2\Mit)^{\frac{1}{1-x}}, 
\end{align*}
Therefore, 
\begin{align*}
    |\hat{z}_n(x)| \leq 1 \vee  C_n(x;\alpha) = 1 \vee O_p(n^{\frac{\alpha-x}{1-x}})
    =\left\{
    \begin{array}{ll}
       O_p(n^{\frac{\alpha-x}{1-x}}) & 0<x<\alpha \\
        O_p(1) & \alpha\leq x<1
    \end{array}
    \right.,
\end{align*}
thereby completing the proof. 

\textbf{Proof of (C)}.
Note that we can take $(s,t)$ satisfying $0<s<\alpha<t<1$ and
$[\alpha\pm\delta_n] \subset [s,t] \subset (0,1)$ for sufficiently large $n$ by $\delta_n = o(1/\log n) = o(1)$. Considering $\hat{z}_n(x)$ is monotone on $(0,1)$ by (A), (B) implies
\begin{align*}
    \sup_{x\in[\alpha\pm\delta_n]} |\hat{z}_n(x)| \leq \sup_{x\in[s,t]} |\hat{z}_n(x)| \leq |\hat{z}_n(s)| + |\hat{z}_n(t)| = O_p(n^{\frac{\alpha-s}{1-s}}) + O_p(1)= o_p(n^\alpha).
\end{align*}
Then, for $C_n(x; \alpha)$ in \Cref{Ap:z_n_nonasym}, it holds that
\begin{align*}
    \log C_n(x=\alpha\pm\delta_n;\alpha) &=\frac{\mp\delta_n}{1-\alpha\mp\delta_n}\log n + \frac{1}{1-\alpha\mp\delta_n} \log\left( 2\frac{K_n+\hat{z}_n(\alpha\pm\delta_n)}{n^\alpha}\right) \\
    &\rightarrow 0 + \frac{1}{1-\alpha} \log (2\Mit + 0)
\end{align*}
and hence $C_n(\alpha\pm\delta_n;\alpha) = O_p(1)$.
Therefore, the monotonicity of $\hat{z}_n(x)$ on $(0,1)$ and \Cref{Ap:z_n_nonasym} imply
\begin{align*}
    \sup_{x\in\alpha\pm\delta_n} |\hat{z}_n(x)| \leq \max |\hat{z}_n(\alpha\pm\delta_n)| \leq 1 \vee \max C_n(\alpha\pm\delta_n; \alpha) = O_p(1).
\end{align*}
Finally, we show $\sup_{x\in [\alpha\pm\delta_n]}|f_\alpha(\hat{z}_n(x)) - \log (K_n/n^\alpha)|=o_p(1)$. 
$\hat{z}'_n(x) < 0$ by (A) and $f'_\alpha(z) > 0$ by \Cref{lm:f_alpha} imply
$(f_\alpha\circ \hat{z}_n)'(x) < 0$ for all $x\in(0,1)$, so $f_\alpha\circ \hat{z}_n$ is monotone on $(0,1)$. This implies
$$
  \sup_{x\in [\alpha\pm\delta_n]}|f_\alpha(\hat{z}_n(x)) - \log (K_n/n^\alpha)| \le \max |f_\alpha(\hat{z}_n(\alpha\pm\delta_n)) - \log (K_n/n^\alpha)|,
$$
Let us denote $\hat{z}_n(\alpha\pm\delta_n)$ by $\hat{z}_n^{\pm}$ and $\alpha\pm\delta_n$ by $\alpha_n^\pm$. With $\alpha^{\pm}_n \hat{z}_n^\pm = \hat{y}_n(\alpha_n^\pm)$, \eqref{eq:df_y_n} can be written as 
$$
0 =\sum_{i=1}^{K_n-1}\frac{1}{\hat{z}_n^\pm + i} - \sum_{i=1}^{n-1}\frac{\alpha_n^{\pm}}{\alpha_n^{\pm} \hat{z}_n^\pm + i}. 
$$
Using the digamma function $\psi$ and $f_\alpha(z) = \psi(1+z) - \alpha\psi(1+\alpha z)$, we can write the above display as 
\begin{align*}
    f_\alpha(\hat{z}_n^\pm) - \log (K_n/n^\alpha) &= 
    (\psi(K_n + \hat{z}_n^\pm) - \log K_n) -\alpha(\psi(n+\alpha_n^\pm\hat{z}_n^\pm)-\log n)\\
    &- \sum_{i=1}^{n-1}\frac{\alpha_n^\pm-\alpha}{\alpha_n^\pm\hat{z}_n^\pm+i}
    -  \sum_{i=1}^{n-1}\frac{\alpha (\alpha-\alpha_n^\pm)\hat{z}_n^\pm}{(\alpha_n^\pm\hat{z}_n^\pm + i)(\alpha \hat{z}_n^\pm + i)}.
\end{align*}
With $\hat{z}_n^\pm = O_p(1)$ and \Cref{lm:digamma_ineq}, we know that the first and second terms are $o_p(1)$. As for the third and fourth terms,
$\alpha_n^\pm \in [s,t] \subset (0,1)$ for sufficiently large $n$ and $\hat{z}_n^\pm = O_p(1)$  result in
\begin{align*}
    \left|\sum_{i=1}^{n-1}\frac{\alpha_n^\pm-\alpha}{\alpha_n^\pm\hat{z}_n^\pm+i}\right| &\leq \delta_n\sum_{i=1}^{n-1}\frac{1}{i-\alpha_n^\pm} 
    \leq \delta_n \sum_{i=1}^{n-1}\frac{1}{i-t} 
    = 
    o(1/\log n) O_p(\log n) = o_p(1),\\
    \left|\sum_{i=1}^{n-1}\frac{\alpha (\alpha-\alpha_n^\pm)\hat{z}_n^\pm }{(\alpha_n^\pm\hat{z}_n^\pm + i)(\alpha \hat{z}_n^\pm + i)}\right| &\leq \alpha \delta_n |\hat{z}_n^\pm| \sum_{i=1}^{n-1}\frac{1}{(-t+i)(-\alpha+i)} = o(1/\log n) O_p(1) O(1) = o_p(1).
\end{align*}
This concludes $f_\alpha(\hat{z}_n^\pm) - \log (K_n/n^\alpha) = o_p(1)$ and completes the proof. 
\end{proof}
\subsection{Proof of \Cref{lm:uniform_control_delta_n}-(B)}\label{proof:uniform_control_delta_n_B}
 Let $B_n = [\alpha\pm\delta_n]$ with $\delta_n=o(1/\log n)$. Since $\delta_n\to 0$, we can take $[s,t]$ such that $B_n \subset [s,t] \subset (0,1)$ for sufficiently large $n$. By the triangle inequality, we have 
 \begin{align*}
     \sup_{x\in B_n} |\hat{\Psi}'_n(x) - \Psi'(x)| &\leq \sup_{x\in B_n} |\hat{\Psi}'_n(x) - \hat{\Psi}_{n,0}'(x)| + \sup_{x\in [s,t]} |\hat{\Psi}'_{n,0}(x) - \Psi'(x)|,
 \end{align*}
 where the second term is $o_p(1)$ by (B) of \Cref{lm:qPsi_n_conv}. Thus, it suffices to show that $\sup_{x\in B_n} |\hat{\Psi}'_n(x) - \hat{\Psi}_{n,0}'(x)| =o_p(1)$. Here, combining \eqref{eq:hat_Psi_x-Psi} and 
 \eqref{eq:df_qpsi_n} with $\plug = 0$, we have
 \begin{align}\label{eq:Psi_qPsi_0}
     \hat{\Psi}_n (x) - \hat{\Psi}_{n,0}(x) = -\frac{1}{K_n}\sum_{i=1}^{n-1} \frac{\hat{z}_n(x)}{x\hat{z}_n(x) + i}.
 \end{align}
Taking the derivative with respect to $x$, we get  
 \begin{align*}
     \hat{\Psi}_n' (x) - \hat{\Psi}_{n,0}'(x)
     &=\frac{1}{K_n}\sum_{i=1}^{n-1} \frac{(\hat{z}_n(x))^2} {(x\hat{z}_n(x) + i)^2} 
   - \frac{\hat{z}'_n(x)}{K_n} \sum_{i=1}^{n-1}
   \frac{i}{(x\hat{z}_n(x) + i)^2}.
 \end{align*}
 Note $x\hat{z}_n(x)  > -x$. Then, the triangle inequality and $B_n\subset [s,t]\subset (0,1)$ imply that 
 \begin{align*}
   \sup_{x\in B_n}\left|\hat{\Psi}'_n(x) - \hat{\Psi}'_{n,0}(x)\right| \leq \frac{\sup_{x\in B_n}|\hat{z}_n(x)|^2}{K_n}\sum_{i=1}^\infty\frac{1}{(i-t)^2} + \sup_{x\in B_n}\left|\frac{\hat{z}'_n(x)}{K_n} \sum_{i=1}^{n-1}
   \frac{i}{(x\hat{z}_n(x) + i)^2}\right|,
 \end{align*}
 where the first term is $o_p(1)$ considering $\sup_{x\in B_n}|\hat{z}_n(x)| = O_p(1)$ by (C) of \Cref{lm:z_n}. For the second term, the formula of $\hat{z}_n'(x)$ by
 (A) of \Cref{lm:z_n} implies that
 \begin{align*}
     &\frac{\hat{z}'_n(x)}{K_n} \sum_{i=1}^{n-1}
   \frac{i}{(x\hat{z}_n(x) + i)^2}= \frac{1}{K_n}\frac{\left(\sum_{i=1}^{n-1}
 i \cdot (x\hat{z}_n(x) + i)^{-2} \right)^2}{\sum_{i=1}^{K_n-1} (\hat{z}_n(x)+i)^{-2} - \sum_{i=1}^{n-1}{x^2}\cdot{(x\hat{z}_n(x)+i)^{-2}}}.
 \end{align*}
 Let $J_n(x):=\sum_{i=1}^{K_n-1} (\hat{z}_n(x)+i)^{-2} - \sum_{i=1}^{n-1}{x^2}\cdot{(x\hat{z}_n(x)+i)^{-2}}$ be the denominator. 
Using $x\hat{z}_n(x)> -x$ and $B_n\subset [s,t]\subset (0,1)$ again, we obtain
 \begin{align*}
    \sup_{x\in B_n}\left|\frac{\hat{z}'_n(x)}{K_n} \sum_{i=1}^{n-1}
   \frac{i}{(x\hat{z}_n(x) + i)^2}\right| &\leq \frac{1}{K_n} \left(\sum_{i=1}^{n-1}\frac{i}{(i-t)^2}\right)^2 \frac{1}{\inf_{x\in B_n}|J_n(x)|} = O_p\left(\frac{(\log n)^2}{n^\alpha}\right) \frac{1}{\inf_{x\in B_n}|J_n(x)|},
 \end{align*}
 so it remains to show  $(\inf_{x\in B_n}|J_n(x)|)^{-1} = O_p(1)$. The rest of the proof is technical, so we divide it into three steps:
 Letting $f_\alpha$ be the bijective function defined by \eqref{eq:df_f_alpha}, first we prove that $$\sup_{x\in B_n}|-J_n(x) + f'_\alpha(\hat{z}_n(x))| = o_p(1).$$ Next we show that $$\sup_{x\in B_n} |f'_\alpha(\hat{z}_n(x)) - f'_\alpha \circ f_\alpha^{-1}(\log (K_n/n^\alpha))| = o_p(1).$$ Putting together the above displays, we prove that
 \begin{align*}
     1/\inf_{x\in B_n} |J_n(x)| \leq \left(2 f_\alpha' \circ f_\alpha^{-1} (\log K_n/n^\alpha)\right)^{-1} = O_p(1).
 \end{align*}
 
 {Step 1. $\sup_{x\in B_n}|-J_n(x) + f'_\alpha(\hat{z}_n(x))| = o_p(1)$}:
 Using the basic equation of the trigamma function: $\psi^{(1)}(1+z)=\sum_{i=1}^\infty (i+z)^{-2}$ (for all $z>-1$), we write $J_n(x)$ as
 \begin{align*}
     J_n(x) 
     &= -\psi^{(1)}(K_n + \hat{z}_n(x)) + \psi^{(1)}(1+\hat{z}_n(x)) + x^2 \psi^{(1)}(x\hat{z}_n(x) + n) - x^2\psi^{(1)}(x\hat{z}_n(x) + 1).
 \end{align*}
 Observe that
 $f_\alpha'(z) = \psi^{(1)}(1+z) - \alpha^2\psi^{(1)}(1+\alpha z)$ by the definition $f_\alpha(z) = \psi(1+z)-\alpha \psi(1+\alpha z)$. Then, $-J_n(x) + f_\alpha'(\hat{z}_n(x))$ can be written as
 \begin{align*}
     -J_n(x) + f_\alpha'(\hat{z}_n(x)) = \psi^{(1)}(K_n + \hat{z}_n(x)) - x^2 \psi^{(1)}(x\hat{z}_n(x) + n) + x^2\psi^{(1)}(x\hat{z}_n(x) + 1)- \alpha^2 \psi^{(1)}(\alpha \hat{z}_n(x) + 1)
 \end{align*}
 Note that $\psi^{(1)}(x)$ is positive and decreasing on $(0,\infty)$. Then, using $\hat{z}_n(x) > -1$ and $B_n =[\alpha \pm \delta_n] (\subset [s,t]\subset(0,1))$ again, we obtain the following for all $x\in B_n$:
 \begin{align*}
 &|-J_n(x) +  f'_\alpha(\hat{z}_n(x))|\\
 &\leq \psi^{(1)}(K_n + \hat{z}_n(x)) + x^2 \psi^{(1)}(x\hat{z}_n(x) + n)
    + |x^2-\alpha^2| \psi^{(1)}(x\hat{z}_n(x) + 1) + \alpha^2 |\psi^{(1)}(x \hat{z}_n(x) + 1)-\psi^{(1)}(\alpha \hat{z}_n(x) + 1)|\\
    &\leq \psi^{(1)} (K_n -1) + \psi^{(1)}(n -1) + 2\delta_n \psi^{(1)} (1-t) + |\psi^{(1)}(x \hat{z}_n(x) + 1)-\psi^{(1)}(\alpha \hat{z}_n(x) + 1)|.
 \end{align*}
 The first, the second, and the third terms do not depend on $x$, and they are $o_p(1)$ by $\psi^{(1)}(x) \rightarrow 0$ as $x\rightarrow\infty$ and $\delta_n = o(1)$. Thus, we have
 $$\sup_{x\in B_n}|-J_n(x) + f'_\alpha(\hat{z}_n(x))| \leq o_p(1) + \sup_{x\in B_n} |\psi^{(1)}(x \hat{z}_n(x) + 1)-\psi^{(1)}(\alpha \hat{z}_n(x) + 1)|.$$
 The second term is bounded by
 \begin{align*}
   &\sup_{x\in B_n} |\psi^{(1)}(x \hat{z}_n(x) + 1)-\psi^{(1)}(\alpha \hat{z}_n(x) + 1)| \\
   &= \sup_{x\in B_n} \left|\sum_{i=1}^{\infty}\left(
    \frac{1}{(x \hat{z}_n(x) + i)^2} - \frac{1}{(\alpha \hat{z}_n(x) + i)^2}
   \right)\right|\\
   &= \sup_{x\in B_n} \left|\sum_{i=1}^\infty \frac{(\alpha-x) \hat{z}_n(x) (2 i + (\alpha + x) \hat{z}_n(x))}{(\alpha \hat{z}_n(x) + i)^2 (x \hat{z}_n(x) + i)^2}\right|\\
   &\leq \sup_{x\in B_n}|\alpha-x| \sum_{i=1}^{\infty}
   \frac{\sup_{x\in B_n} 2i|\hat{z}_n(x)| + 2|\hat{z}_n(x)|^2}{(-\alpha + i)^2 (-t + i)^2}\\
   &\leq \sup_{x\in B_n}|\alpha-x| \cdot \left(\sup_{x\in B_n}|\hat{z}_n(x)| + \sup_{x\in B_n}|\hat{z}_n(x)|^2\right)\cdot \sum_{i=1}^{\infty}\frac{2i}{(i-\alpha)^2(i-t)^2} \\
   &\overset{(\star)}{=} \delta_n O_p(1) O(1) = o_p(1),
 \end{align*}
 where $(\star)$ follows from $\sup_{x\in B_n} |\hat{z}_n(x)|= O_p(1)$ ((C) of \Cref{lm:z_n}). This completes the proof of Step 1. 
 
 {Step 2. $\sup_{x\in B_n} |f'_\alpha(\hat{z}_n(x)) - f'_\alpha \circ f_\alpha^{-1}(\log (K_n/n^\alpha))| = o_p(1)$}:
 $f''_\alpha < 0$ on $(-1,\infty)$ (\Cref{lm:f_alpha}) and $\hat{z}'_n < 0$ on $(0,1)$ ((A) of \Cref{lm:z_n}) imply that
 $(f'_\alpha\circ\hat{z}_n)'(x) > 0$, especially that $f'_\alpha\circ\hat{z}_n$ is monotone on $(0,1)$. Then 
 \begin{align*}
     &\sup_{x\in B_n} |f'_\alpha(\hat{z}_n(x)) - f'_\alpha \circ f_\alpha^{-1}(\log (K_n/n^\alpha))|\\
   &\leq |f'_\alpha(\hat{z}_n(\alpha \pm \delta_n)) - f'_\alpha \circ f_\alpha^{-1}(\log (K_n/n^\alpha))|\\
   &\leq |f'_\alpha(\hat{z}_n(\alpha \pm \delta_n)) - f'_\alpha \circ f_\alpha^{-1}(\log \Mit)|+ |  f'_\alpha \circ f_\alpha^{-1}(\log \Mit) - f'_\alpha \circ f_\alpha^{-1}(\log (K_n/n^\alpha))|,
 \end{align*} 
 where the first term is  $o_p(1)$ by $\hat{z}_n(\alpha \pm \delta_n) \to^p f_\alpha^{-1}(\log \Mit)$ ((C) of \Cref{lm:z_n}) and  the continuity of $f_\alpha'$, while the second term is also $o_p(1)$ by $K_n/n^\alpha \rightarrow \Mit$ (a.s.). This completes the proof of Step 2.
   
 {Step 3. $1/{\inf_{x\in B_n}|J_n(x)| = O_p(1)}$}: $f'_\alpha>0$ by \Cref{lm:f_alpha} implies 
 \begin{align*}
   \sup_{x\in B_n} \left| \frac{-J_n(x)}{(f'_\alpha \circ f_\alpha^{-1}(\log (K_n/n^\alpha)))} + 1
   \right| = \frac{\sup_{x\in B_n}|-J_n(x) + f_\alpha'\circ f_\alpha^{-1}(\log (K_n/n^\alpha))|}{f'_\alpha \circ f_\alpha^{-1}(\log (K_n/n^\alpha))},
 \end{align*}
 where the numerator is $o_p(1)$ from 
 Step 1 and Step 2, and the denominator converges to $f'_\alpha \circ f_\alpha^{-1}(\log \Mit) >0$. Thus, the right-hand side is $o_p(1)$, thereby 
 $$\sup_{x\in B_n} \left|
   -J_n(x)/(f'_\alpha \circ f_\alpha^{-1}(\log (K_n/n^\alpha))) + 1
   \right| < 1/2$$ holds with a high probability. 
   Combining this and $f_\alpha'>0$ again, we obtain
   \begin{align*}
   &\forall x\in B_n, \quad \frac{-|J_n(x)|}{(f'_\alpha \circ f_\alpha^{-1}(\log (K_n/n^\alpha)))} + 1
 \leq \left| \frac{-J_n(x)}{(f'_\alpha \circ f_\alpha^{-1}(\log (K_n/n^\alpha)))} + 1
   \right| < \frac{1}{2}\\
     \Rightarrow
     &\forall x\in B_n, \quad  |J_n(x)| > f'_\alpha \circ f_\alpha^{-1}(\log (K_n/n^\alpha))/2\\
     \Rightarrow &({\inf_{x\in B_n}|J_n(x)|})^{-1} < {2}/(f'_\alpha \circ f_\alpha^{-1}(\log (K_n/n^\alpha))).
   \end{align*}
 Putting this together with 
  ${2}/(f'_\alpha \circ f_\alpha^{-1}(\log (K_n/n^\alpha))) \rightarrow 2/f'_\alpha \circ f_\alpha^{-1}(\log \Mit) = O_p(1)$, we get $({\inf_{x\in B_n}|J_n(x)|})^{-1} = O_p(1)$.
\end{document}